\documentclass[a4paper,reqno,12pt]{amsart}
\usepackage[T1]{fontenc}
\usepackage[utf8]{inputenc}
\usepackage[english]{babel}

\usepackage{amsmath}
\usepackage{amsthm}
\usepackage{amssymb}
\usepackage{mathtools}
\usepackage{amsmath}
\usepackage{amsfonts}
\usepackage{mathrsfs}
\usepackage{esint}
\usepackage{yfonts}
\usepackage{quoting}
\usepackage{graphics}
\usepackage{booktabs}
\usepackage{bm}
\usepackage{bbm}
\usepackage{lmodern}
\usepackage{stmaryrd}
\usepackage{caption}
\usepackage{enumitem}			
\usepackage{tikz}
\usepackage{bookmark}
\usepackage{pgfplots}
\usepgfplotslibrary{fillbetween}
\usepackage{orcidlink}
\usepackage[mathscr]{euscript}
\usepackage[]{old-arrows}		

\usepackage{hyperref}
\hypersetup{
	colorlinks=true,
	linkcolor=blue,
	citecolor=blue,
	urlcolor=black,
	linktoc=all
}

\usepackage[margin=1.1in]{geometry}

\quotingsetup{font=small}

\theoremstyle{plain}
\newtheorem{proposition}{Proposition}[section]

\theoremstyle{plain}
\newtheorem{theorem}{Theorem}[section]
\numberwithin{equation}{section}	

\theoremstyle{plain}
\newtheorem{lemma}[theorem]{Lemma}

\theoremstyle{plain}
\newtheorem{corollary}{Corollary}[theorem]

\theoremstyle{definition}

\DeclarePairedDelimiter{\abs}{\lvert}{\rvert}
\DeclarePairedDelimiter{\norma}{\lVert}{\rVert}

\DeclareMathOperator{\defi}{def}
\DeclareMathOperator{\tr}{tr}
\DeclareMathOperator{\Id}{Id_\textit{n}}
\DeclareMathOperator{\diver}{div}

\newcommand{\numberset}{\mathbb}
\newcommand{\N}{\numberset{N}}			
\newcommand{\R}{\numberset{R}}			
\newcommand{\sfera}{\numberset{S}}		
\newcommand{\Haus}{\mathcal{H}}			
\newcommand{\past}{p^\ast}			
\newcommand{\loc}{{\rm loc}}
\newcommand{\stressu}{{a\!\left(\nabla u\right)}}
\newcommand{\stressv}{{a\!\left(\nabla v\right)}}

\def\Xint#1{\mathchoice
	{\XXint\displaystyle\textstyle{#1}}%
	{\XXint\textstyle\scriptstyle{#1}}%
	{\XXint\scriptstyle\scriptscriptstyle{#1}}%
	{\XXint\scriptscriptstyle\scriptscriptstyle{#1}}%
	\!\int}
\def\XXint#1#2#3{{\setbox0=\hbox{$#1{#2#3}{\int}$ }
		\vcenter{\hbox{$#2#3$ }}\kern-.6\wd0}}

\def\dashint{\Xint-}

\usepackage{enumitem}

\allowdisplaybreaks

\begin{document}
	
	\title[Stability of the critical~$p$-Laplace equation]{On the stability of the critical~$p$-Laplace equation}
	
	\author[Giulio Ciraolo]{Giulio Ciraolo \orcidlink{0000-0002-9308-0147}}
	\address[]{Giulio Ciraolo. Dipartimento di Matematica ‘Federigo Enriques’, Università degli Studi di Milano, Via Cesare Saldini 50, 20133, Milan, Italy}
	\email{giulio.ciraolo@unimi.it}
	\author[Michele Gatti]{Michele Gatti \orcidlink{0009-0002-6686-9684}}
	\address[]{Michele Gatti. Dipartimento di Matematica ‘Federigo Enriques’, Università degli Studi di Milano, Via Cesare Saldini 50, 20133, Milan, Italy}
	\email{michele.gatti1@unimi.it}
	
	\date{\today}
	\dedicatory{}
	
	\begin{abstract}
		For~$1<p<n$, it is well-known that non-negative, energy weak solutions to~$\Delta_p u + u^{\past-1} =0$ in~$\R^n$ are completely classified. Moreover, due to a fundamental result by Struwe and its extensions, this classification is stable up to bubbling.
		
		In the present work, we investigate the stability of perturbations of the critical~$p$-Laplace equation for any~$1<p<n$, under a condition that prevents bubbling. In particular, we show that any solution~$u \in \mathcal{D}^{1,p}(\R^n)$ to such a perturbed equation must be quantitatively close to a bubble. This result generalizes a recent work by the first author, together with Figalli and Maggi~\cite{cfm}, in which a sharp quantitative estimate was established for~$p=2$. However, our analysis differs completely from theirs and is based on a quantitative~$P$-function approach.
	\end{abstract}
	
	\maketitle
	
	
	\section{Introduction}
	
	For~$n \in \N$ and~$1<p<n$, the critical~$p$-Laplace equation
	\begin{equation}
	\label{eq:eqcritica}
		\Delta_p u + u^{\past-1} =0 \quad \mbox{in } \R^n
	\end{equation}
	arises as the Euler-Lagrange equation associated to the problem of finding the optimal
	constant in the Sobolev inequality, that is
	\begin{equation}
	\label{eq:sob-const}
		S_p = S \coloneqq \min_{u \in \mathcal{D}^{1,p}(\R^n) \setminus \{ 0 \}} \frac{ \norma*{\nabla u}_{L^p(\R^n)}}{\norma*{u}_{L^{\past}\!(\R^n)}}.
	\end{equation}
	This constant is attained by the functions belonging to the~$(n+2)$-dimensional manifold of the \textit{Talenti bubbles}
	\begin{equation*}
		\mathcal{M} \coloneqq \left\{U_{a,b,z} \,\Big\lvert\, U_{a,b,z}(x) \coloneqq a \left(1+b \,\abs*{x-z}^\frac{p}{p-1}\right)^{\!-\frac{n-p}{p}}, \; a \in \R\setminus\{0\},b>0,z \in \R^n\right\},
	\end{equation*}
	as shown by Aubin~\cite{au} and Talenti~\cite{tal}. Furthermore,~$\mathcal{M}$ coincides with the space of all weak solutions to
	\begin{equation*}
		\Delta_p u + S^p \,\norma*{u}_{L^{\past}\!(\R^n)}^{p-\past} \,\abs*{u}^{\past-2} u =0 \quad \mbox{in } \R^n
	\end{equation*}
	which do not change sign. Among these functions, a significant role is played by the $p$-\textit{bubbles}, defined as
	\begin{equation}
	\label{eq:pbubb}
		U_p[z,\lambda] (x) \coloneqq \left( \frac{\lambda^\frac{1}{p-1} \, n^\frac{1}{p} \left(\frac{n-p}{p-1}\right)^{\!\!\frac{p-1}{p}}}{\lambda^\frac{p}{p-1}+\abs*{x-z}^\frac{p}{p-1}} \right)^{\!\!\frac{n-p}{p}} \!,
	\end{equation}
	where~$\lambda>0$ and~$z \in \R^n$ are the scaling and translation parameters, respectively, and they satisfy~\eqref{eq:eqcritica}. Moreover, we notice that any~$p$-bubble~$U=U_p[z,\lambda]$ satisfies
	\begin{equation}
	\label{eq:energ-bubb}
		\norma{U}_{L^{\past}\!(\R^n)}^{\past} = \norma{\nabla U}_{L^{p}(\R^n)}^p = S^n.
	\end{equation}

	In the case~$p=2$, the seminal works of Gidas, Ni \& Nirenberg~\cite{gnn}, Caffarelli, Gidas \& Spruck~\cite{cgs}, and Chen \& Li~\cite{cl}, established that these functions are the only positive classical solutions to~\eqref{eq:eqcritica}.
	
	Moreover, for~$1<p<n$, it is known from the more recent works of Damascelli, Merch\'an, Montoro \& Sciunzi~\cite{dm}, V\'etois~\cite{vet}, and Sciunzi~\cite{sciu} that the~$p$-bubbles~\eqref{eq:pbubb} are the only positive weak solutions to~\eqref{eq:eqcritica} in the energy space~$\mathcal{D}^{1,p}(\R^n)$, where
	\begin{equation*}
		\mathcal{D}^{1,p}(\R^n) \coloneqq \{ u \in L^{\past}\!(\R^n) \,\lvert\, \nabla u \in L^p(\R^n) \}.
	\end{equation*}
	See also~\cite{cfr} for the classification in an anisotropic setting. Finally, we also mention that this classification result has recently been extended without the energy assumption for certain values of~$p$ depending on~$n$ -- see~\cite{catino,ou,vet-plap}. \newline

	A related question to the study of the optimal constant in~\eqref{eq:sob-const} is the stability of the Sobolev inequality
	\begin{equation}
	\label{eq:sobolev-ineq}
		\norma*{u}_{L^{\past}\!(\R^n)} \leq S^{-1} \norma*{\nabla u}_{L^p(\R^n)} \quad \mbox{for every } u \in \mathcal{D}^{1,p}(\R^n),
	\end{equation}
	which is naturally associated with the~$p$-\textit{Sobolev deficit}
	\begin{equation}
	\label{eq:Sob-def}
		\delta_{Sob}(u) \coloneqq \frac{ \norma*{\nabla u}_{L^p(\R^n)}}{\norma*{u}_{L^{\past}\!(\R^n)}} - S \quad \mbox{for } u \in \mathcal{D}^{1,p}(\R^n).
	\end{equation}
	Clearly,~$\delta_{Sob}(u) \geq 0$ and it vanishes only on the manifold of minimizers to the Sobolev quotient~\eqref{eq:sob-const}, i.e., on~$\mathcal{M}$.
	
	In Problem~A of~\cite{brezis-lieb}, Brezis \& Lieb asked whether, for~$p=2$ and a function~$u \in \mathcal{D}^{1,2}(\R^n)$, it is possible to control an appropriate distance between~$u$ and the manifold~$\mathcal{M}$ in terms of the~$2$-Sobolev deficit~\eqref{eq:Sob-def}. This question was affirmatively answered by Bianchi \& Egnell~\cite{bianchi-eg}, who proved that
	\begin{equation}
	\label{eq:bianchi-eg}
		c_n \inf_{v \in \mathcal{M}} \left(\frac{\norma*{\nabla u - \nabla v}_{L^2(\R^n)}}{\norma*{\nabla u}_{L^2(\R^n)}}\right)^{\!\!2} \leq \delta_{Sob}(u),
	\end{equation}
	for some dimensional constant~$c_n>0$. Moreover, the exponent at the left-hand side of~\eqref{eq:bianchi-eg} is sharp.
	
	The technique developed in~\cite{bianchi-eg} strongly relies on the Hilbert structure of~$\mathcal{D}^{1,2}(\R^n)$, making it unsuitable for addressing the problem in the general case~$1<p<n$. Consequently, the study of the stability issue for a general~$p$ has led to several significant contributions -- see, for instance,~\cite{cfmp,figalli-maggi-prat,figalli-neu,neumayer}. The final and sharp resolution of the problem, both in terms of the strength of the distance from~$\mathcal{M}$ and the optimal exponent, was provided by Figalli \& Zhang~\cite{figalli-zhang}. Specifically, they showed that
	\begin{equation*}
	\label{eq:est-fig-neu}
		c_{n,p} \inf_{v \in \mathcal{M}} \left(\frac{\norma*{\nabla u - \nabla v}_{L^{p}(\R^n)}}{\norma*{\nabla u}_{L^{p}(\R^n)}}\right)^{\!\max\left\{2,p\right\}} \leq \delta_{Sob}(u) \quad\mbox{for all } u \in \mathcal{D}^{1,p}(\R^n),
	\end{equation*}
	where the exponent is sharp, thereby fully settling the question of stability for~\eqref{eq:sobolev-ineq} à la Bianchi-Egnell. \newline
	

	Another relevant line of research concerns the stability issue for critical points of the Sobolev quotient~\eqref{eq:sob-const}. In the case~$p=2$, Struwe~\cite{struwe} established a qualitative stability result for~\eqref{eq:eqcritica}. In particular, he showed that if a non-negative~$u$ \textit{almost} solves~\eqref{eq:eqcritica}, then~$u$ is \textit{close} to the sum of $2$-bubbles in the~$\mathcal{D}^{1,2}$-norm. This phenomenon is known as \textit{bubbling}. We also mention that an analogous compactness result in~$\R^n$ was obtained by Benci \& Cerami~\cite{benci-cer}.

	From a PDEs perspective, this problem reduces to studying
	\begin{equation}
	\label{eq:maineq-bubb_p2}
		\Delta u + \kappa(x) u^{2^\ast -1} =0 \quad \mbox{in } \R^n,
	\end{equation}
	for~$\kappa$ close to a constant, which can be regarded as a perturbation of~\eqref{eq:eqcritica} when~$p=2$.

	A quantitative version of Struwe's theorem for \eqref{eq:maineq-bubb_p2} was studied by the first author, Figalli \& Maggi~\cite{cfm} under the a priori energy assumption
	\begin{equation}
	\label{eq:ipotesi-energ_p2}
		\frac{1}{2} S^n \leq \int_{\R^n} \,\abs*{\nabla u}^2 \, dx \leq \frac{3}{2} S^n.
	\end{equation}
	Recalling~\eqref{eq:energ-bubb}, we see that~\eqref{eq:ipotesi-energ_p2} ensures that the energy of~$u$ is nearly that of a single~$p$-bubble. Therefore,~\eqref{eq:ipotesi-energ_p2} is the main ingredient to prevent bubbling. A sharp study of this phenomenon was subsequently carried out by Figalli \& Glaudo~\cite{fg} and Deng, Sun \& Wei~\cite{dsw} by assuming the counterpart of~\eqref{eq:ipotesi-energ_p2} for multiple bubbles.
	
	In~\cite{cfm,dsw,fg}, the quantitative stability result is understood in the sense of~$\mathcal{D}^{1,2}(\R^n)$, thus providing estimates on the $L^2$-norm of the gradient of the difference between the solution and a suitable sum of $2$-bubbles in terms of some \textit{deficit}~$\defi(u,\kappa)$, which measures how close is~$\kappa$ to being constant. 

	The analyses conducted in~\cite{cfm,dsw,fg} mainly rely on the knowledge of the family~\eqref{eq:pbubb}, along with their spectral properties, and the Hilbert structure of the energy space~$\mathcal{D}^{1,2}(\R^n)$. To the best of our knowledge, these aspects have so far hindered the extension of the techniques developed therein to the case~$p \neq 2$.
	
	For completeness, we also mention that the stability of the fractional Sobolev inequality in the space~$\dot{W}^{s,2}(\R^n)$, following the approach of~\cite{cfm,dsw,fg}, has been studied by Aryan~\cite{aryan} and De Nitti \& K\"onig~\cite{denit-kon}. Moreover, stability issues for more general semilinear equations and for non-energy solutions have been investigated by the authors, together with Cozzi, in~\cite{ccg}.
	

	\subsection{Main results}
	
	We now return to the stability problem for~\eqref{eq:eqcritica} in the general case~$1<p<n$ and present our main results, along with a rough outline of the proof.
	
	When~$1<p<n$, Alves~\cite{alves} and Mercuri \& Willem~\cite{merc-will} extended Struwe's result to the~$p$-Laplacian. Nevertheless, as mentioned above, its quantitative counterpart and the stability of critical points of the Sobolev inequality~\eqref{eq:sobolev-ineq} remain mostly unexplored. 
	
	In this case, in order to obtain quantitative stability results, one has to face 
	\begin{equation}
	\label{eq:maineq-bubb}
		\Delta_p u + \kappa(x) u^{\past-1} =0 \quad \mbox{in } \R^n,
	\end{equation}
	for $\kappa$ close to a constant, which can be seen as a perturbation of~\eqref{eq:eqcritica}, with the a priori energy assumption 	
	\begin{equation}
	\label{eq:ipotesi-energ}
		\frac{1}{2} S^n \leq \int_{\R^n} \,\abs*{\nabla u}^p \, dx \leq \frac{3}{2} S^n.
	\end{equation}
		
	In particular, when~$\kappa \equiv \kappa_0$, any solution~$u \in \mathcal{D}^{1,p}(\R^n)$ to~\eqref{eq:maineq-bubb} must be a Talenti bubble and, by testing the equation with~$u$, it follows that
	\begin{equation}
	\label{eq:defk0}
		\kappa_0=\kappa_0(u) \coloneqq \frac{\int_{\R^n} \kappa(x) u^{\past} dx}{\int_{\R^n} u^{\past} dx} = \frac{\int_{\R^n} \,\abs{\nabla u}^p \, dx}{\int_{\R^n} u^{\past} dx}.
	\end{equation}
	Thus, it is natural to measure the proximity of~$\kappa$ to~$\kappa_0$ using the deficit
	\begin{equation}
		\label{eq:def_cfm}
		\defi(u,\kappa) \coloneqq \norma*{\left(\kappa-\kappa_0\right)u^{\past-1}}_{L^{(\past)'}\!(\R^n)}.
	\end{equation}
	For further details regarding this choice, we refer to~\cite{cfm}.
	
	With this notation, the main result of~\cite{cfm} states that if~$p=2$ and~$u \in \mathcal{D}^{1,2}(\R^n)$ is a non-negative function satisfying~\eqref{eq:maineq-bubb}--\eqref{eq:ipotesi-energ}, then there exists a~$2$-bubble~$\mathcal{U}_0$ of the form~\eqref{eq:pbubb}, i.e.,~$\mathcal{U}_0=U_2[z,\lambda]$ for some~$z\in\R^n$ and~$\lambda >0$, such that
	\begin{equation*}
		\norma*{u-\mathcal{U}_0}_{\mathcal{D}^{1,2}(\R^n)} \leq c_n \defi(u,\kappa),
	\end{equation*}
	for a dimensional constant~$c_n>0$.
	
	
	In the present paper, we aim to establish a quantitative stability result for positive weak solutions~$u \in \mathcal{D}^{1,p}(\R^n)$ to~\eqref{eq:maineq-bubb}, under a condition that prevents bubbling. Our main contributions are the following. 
	
	\begin{theorem}
	\label{th:main-bubbles}
		Let~$n \in \N$,~$1<p<n$, and~$\kappa \in L^\infty(\R^n) \cap C^{1,1}_{\loc}(\R^n)$ a positive function. Let~$u \in \mathcal{D}^{1,p}(\R^n)$ be a positive weak solution to~\eqref{eq:maineq-bubb} satisfying~\eqref{eq:ipotesi-energ}. Moreover, suppose that
		\begin{equation*}
			\kappa_0(u)=1,
		\end{equation*}
		where~$\kappa_0(u)$ is defined in~\eqref{eq:defk0}.
		
		Then, there exist a large constant~$C \geq 1$, a small~$\vartheta \in (0,1)$, and a~$p$-bubble~$\mathcal{U}_0$, of the form~\eqref{eq:pbubb}, such that
		\begin{equation}
		\label{eq:close-toabub}
			\norma*{u-\mathcal{U}_0}_{\mathcal{D}^{1,p}(\R^n)} \leq C \defi(u,\kappa)^{\vartheta}.
		\end{equation}
		The constant~$C$ depends only on~$n$,~$p$, and~$\norma*{\kappa}_{L^\infty(\R^n)}$, while~$\vartheta$ depends only on~$n$ and~$p$.
	\end{theorem}

	The general case can be reduced to~$\kappa_0(u)=1$ by scaling, as already noticed in~\cite{cfm}. Thus, we have the following corollary.
	
	\begin{corollary}
	\label{cor:main-bubbles}
		Let~$n \in \N$ be an integer,~$1<p<n$, and~$\kappa \in L^\infty(\R^n) \cap C^{1,1}_{\loc}(\R^n)$ a positive function. Let~$u \in \mathcal{D}^{1,p}(\R^n)$ be a positive weak solution to~\eqref{eq:maineq-bubb} satisfying
		\begin{equation}
		\label{eq:ip-energ-k0}
			\frac{1}{2} \,\kappa_0(u)^{\frac{p}{p-\past}} S^n \leq \int_{\R^n} \,\abs*{\nabla u}^p \, dx \leq \frac{3}{2} \,\kappa_0(u)^{\frac{p}{p-\past}} S^n,
		\end{equation}
		where~$\kappa_0(u)$ is defined in~\eqref{eq:defk0}.
		
		Then, there exist a large constant~$C \geq 1$, a small~$\vartheta \in (0,1)$, and a function~$\mathcal{U} \in \mathcal{D}^{1,p}(\R^n)$ given by
		\begin{equation*}
			\mathcal{U} = \kappa_0(u)^{\frac{1}{p-\past}} \,\mathcal{U}_0,
		\end{equation*}
		for some~$p$-bubble~$\mathcal{U}_0$ of the form~\eqref{eq:pbubb}, such that
		\begin{equation*}
			\norma*{u-\mathcal{U}}_{\mathcal{D}^{1,p}(\R^n)} \leq C \defi(u,\kappa)^{\vartheta}.
		\end{equation*}
		The constant~$C$ depends only on~$n$,~$p$,~$\norma*{\kappa}_{L^\infty(\R^n)}$, and~$\kappa_0(u)$, while~$\vartheta$ depends only on~$n$ and~$p$.
	\end{corollary}
	
	While writing this paper we learned that, independently of us, Liu \& Zhang~\cite{liu-zhang} have proved an alternative result, which will appear in a forthcoming paper. Their approach, inspired by~\cite{figalli-zhang}, allows them to catch the sharp exponent in~\eqref{eq:close-toabub}.

	As far as we know, Theorem~\ref{th:main-bubbles} and Corollary~\ref{cor:main-bubbles}, as well as the main result in~\cite{liu-zhang}, are the first results establishing the quantitative closeness of a solution to~\eqref{eq:maineq-bubb}, viewed as a perturbation of the critical equation~\eqref{eq:eqcritica}, to a $p$-bubble. The exponent~$\vartheta$ in~\eqref{eq:close-toabub} is far from optimal. However, we believe that our approach is suitable for being used in more general contexts -- such as in the anisotropic setting.

	Our strategy to proving Theorem~\ref{th:main-bubbles} follows a more PDE-oriented strategy compared to that of~\cite{cfm}. Specifically, we employ a quantitative argument based on integral identities, similar to the one used in~\cite{cfr} as well as in~\cite{catino,ou,vet-plap}, all of which are inspired by the seminal work of Serrin \& Zou~\cite{serr-zou}. However, in our case, instead of directly applying an inequality for vector fields, we establish a different integral inequality tailored to our setting -- this is the content of Proposition~\ref{prop:fund-ineq} below. This new inequality is more in the spirit of those used by the first author, Farina \& Polvara~\cite{cami} and Ou \cite{ou}, and is obtained by exploiting a positive subsolution to a suitable PDE -- see Subsection~\ref{subsec:out-proof} below.
	
	In the following subsection, we describe the proof strategy in more detail.
	
	
	\subsection{Outline of the proof.}
	\label{subsec:out-proof}
	
	The proof begins by applying the result in~\cite{merc-will} to identify the~$p$-bubble close to~$u$, denoted by~$U$. We then reduce the problem to the case where~$U = U_p[0,1]$. This reduction is possible due to the symmetries of~\eqref{eq:eqcritica} and the family~\eqref{eq:pbubb} -- see Section~\ref{sec:symm} and Lemma~\ref{lem:sym} below. Once~$U$ is fixed, we establish upper and lower quantitative decay estimates for~$u$ and derive an upper bound for its gradient in~\ref{step:decay}. Additionally, we prove that there must be at least one point~$x_0 \in \R^n$ where~$u$ attains its maximum, which will be further localized later in the proof.
	
	Next, we introduce the fundamental auxiliary function~$v$ and the~$P$-function, defined as
	\begin{equation}
	\label{eq:pfunct-intro}
		v \coloneqq u^{-\frac{p}{n-p}} \quad\mbox{and}\quad P \coloneqq n \,\frac{p-1}{p} v^{-1} \,\abs*{\nabla v}^p + \left(\frac{p}{n-p}\right)^{\! p-1} v^{-1}.
	\end{equation}
	We also define the stress field associated with~$v$ by setting~$\stressv \coloneqq \abs*{\nabla v}^{p-2} \,\nabla v$, the tensor~$W \coloneqq \nabla \stressv$, and its traceless version~$\mathring{W}$.
	
	In~\ref{step:Pfunct}, we show that~$v$ is trapped between two~$p$-paraboloids, using terminology analogous to the case~$p=2$ -- see also Figure~\ref{fig:functions-E}. Thus, one heuristically expects~$v$ to be close to a~$p$-paraboloid. The key idea of the proof is to construct such an approximation for~$v$ and then translate this information back to~$u$.
	
	It is well-established that~$v$ satisfies an elliptic equation involving~$P$ and a remainder term -- see~\eqref{eq:eqforv} below. Moreover, this equation can be rewritten in terms of~$W$. Exploiting this formulation, we deduce an integral inequality involving~$P$ and~$W$ -- the content of Proposition~\ref{prop:fund-ineq} below. This result is obtained through the rather technical~\ref{step:differ-id} and~\ref{step:approx}. Notice that this is where the regularity of~$\kappa$ is required.
	
	From Proposition~\ref{prop:fund-ineq}, in~\ref{step:quant-est} and~\ref{step:fund-W}, we infer an integral weighted estimate for~$\abs{\mathring{W}}$ -- see~\eqref{eq:stima-def}. Notably, for a~$p$-bubble, the~$P$-function defined by~\eqref{eq:pfunct-intro} must be constant and
	\begin{equation*}
		W = \frac{P}{n} \Id.
	\end{equation*}
	Therefore, when~$u$ is close to a~$p$-bubble, we expect~$P$ to be approximately constant and, consequently, close to its mean on a small ball~$B_\mathsf{t}(x_0)$ -- recall that~$x_0$ is the point where~$u$ attains its maximum and~$v$ its minimum -- denoted by~$\overline{P}$.  This observation suggests that we should seek a weighted estimate for~$\abs{W-\mu\Id}$, where~$\mu \in \R$ is properly chosen.
	
	It turns out that such an estimate holds if we set~$\mu = \overline{P} / n$ -- see~\eqref{eq:int-to-prove} below. However, we are only able to establish it within a large ball~$B_r$, with~$r$ to be determined by the end of the proof.
	
	Using~\eqref{eq:int-to-prove}, in~\ref{step:approx-func}, we construct a~$p$-paraboloid~$\mathsf{Q}$ centered at~$x_0$, with~$v(x_0)$ as its value at the center, such that both~$v-\mathsf{Q}$ and~$\nabla v - \nabla \mathsf{Q}$ are small -- see estimates~\eqref{eq:norma-p-grad} and~\eqref{eq:norma-p-func}. Unfortunately, when going back to~$u$, the~$p$-paraboloid~$\mathsf{Q}$ does not yield a~$p$-bubble of the form~\eqref{eq:pbubb}.
	
	To fix this issue, we introduce a refined approximation. Specifically, we define a new~$p$-paraboloid~$\mathcal{Q}$, also centered at~$x_0$, ensuring that when inverted to recover~$u$, it precisely yields a~$p$-bubble of the form~\eqref{eq:pbubb}. In practice, this requires selecting a parameter~$\lambda$, which represents the bubble's scaling factor, to uniquely determine~$\mathcal{Q}$. We impose the condition~$\nabla \mathsf{Q} = \nabla \mathcal{Q}$, thereby fixing~$\lambda$. Moreover, it turns out that, up to a factor,~$\lambda = 1/\overline{P}$. Finally, we aim to keep~$v - \mathcal{Q}$ small.
	
	At first glance, it may seem that we have no remaining degrees of freedom to adjust~$\mathcal{Q}$. Nevertheless, this poses no issue since~$v$ and~$\mathsf{Q}$ are already comparable. To estimate~$v - \mathcal{Q}$, it therefore suffices to control~$\mathsf{Q} - \mathcal{Q}$. Moreover, since~$\mathsf{Q}$ and~$\mathcal{Q}$ coincide at the first order, their difference is dictated by their zeroth-order terms. This reduces the problem to comparing~$v(x_0)$ and~$\lambda$, which, up to a factor, amounts to comparing~$v(x_0)$ and~$1/\overline{P}$. Such a comparison is possible by choosing the radius~$\mathsf{t}$ sufficiently small -- see also Figure~\ref{fig:functions-E} for the heuristics.
	
	Once we establish that both~$v-\mathcal{Q}$ and~$\nabla v - \nabla \mathcal{Q}$ are small in the appropriate weighted norm, we return to~$u$ by defining the~$p$-bubble
	\begin{equation*}
		\mathcal{U} \coloneqq \mathcal{Q}^{-\frac{n-p}{p}}.
	\end{equation*}
	This allows us to conclude that~$\nabla u - \nabla \mathcal{U}$ is small in~$L^p(B_r)$. Finally, we select~$r$ such that the decay estimates for~$\nabla u$ and~$\nabla \mathcal{U}$ ensure that~$\norma*{\nabla u - \nabla \mathcal{U}}_{L^p(\R^n \setminus B_r)}$ is also small. The result then follows by reversing the initial reduction argument, with the aid of Lemma~\ref{lem:sym}.
	

	\subsection{Structure of the paper.}
	In Section~\ref{sec:symm}, we recall some symmetry properties of the equation~\eqref{eq:eqcritica} and prove  Corollary~\ref{cor:main-bubbles}. Section~\ref{sec:mainsec-proofofmainres} is dedicated to the proof of Theorem~\ref{th:main-bubbles}. In Section~\ref{sec:sharp-res} we present an example which provides insights on the optimal exponent. Finally, as an application, we derive a quasi-symmetry result for solutions to~\eqref{eq:maineq-bubb} in Section~\ref{sec:application}.

	
	\section{Symmetries of the problem and proof of Corollary~\ref{cor:main-bubbles}}
	\label{sec:symm}
	
	This brief section is inspired by Section~2.1 in~\cite{fg}, where the analogous properties for the case~$p=2$ are listed. Additionally, we provide the proof of Corollary~\ref{cor:main-bubbles}. \newline
	
	Given~$\lambda >0$ and~$z \in \R^n$, let~$T_{z,\lambda}: C^\infty_c(\R^n) \to C^\infty_c(\R^n)$ be the operator defined as
	\begin{equation*}
		T_{z,\lambda}(\phi)(x) \coloneqq \lambda^{\frac{n-p}{p}} \phi\left(\lambda(x-z)\right).
	\end{equation*}
	Clearly, this operator can be extended to functions which are non smooth with compact support in the same fashion.
	
	The following lemma provides some fundamental properties of the family of transformations~$T_{z,\lambda}$ which will be useful in the following.
	
	\begin{lemma}
		\label{lem:sym}
		The operator~$T_{z,\lambda}$ enjoys the subsequent properties.
		\begin{enumerate}[leftmargin=*,label=$(\arabic*)$]
			\item \label{it:cons-normapas} For any~$\phi \in C^\infty_c(\R^n)$, it holds that
			\begin{equation*}
				\int_{\R^n} T_{z,\lambda}(\phi)^{\past} dx = \int_{\R^n} \phi^{\past} dx.
			\end{equation*}
			\item \label{it:cons-normap} For any~$\phi \in C^\infty_c(\R^n)$, it holds that
			\begin{equation*}
				\int_{\R^n} \,\abs*{\nabla T_{z,\lambda}(\phi)}^{p} \, dx = \int_{\R^n} \,\abs*{\nabla \phi}^{p} \, dx.
			\end{equation*}
			\item \label{it:inver} For any~$\phi \in C^\infty_c(\R^n)$, it holds that
			\begin{equation*}
				T_{z,\lambda}\!\left(T_{-z\lambda,1/\lambda}(\phi)\right) = T_{-z\lambda,1/\lambda}\!\left(T_{z,\lambda}(\phi)\right) = \phi.
			\end{equation*}
			\item \label{it:stillbub} For any $p$-bubble~$U$,~$T_{z,\lambda} \!\left(U\right)$ is still a $p$-bubble.
			\item \label{it:transf-unitbub} The $p$-bubbles satisfy
			\begin{equation*}
				U_p[z,\lambda] = T_{z,1/\lambda} \!\left(U_p[0,1]\right) \quad\mbox{and}\quad U_p[0,1] = T_{-z/\lambda,\lambda} \!\left(U_p[z,\lambda]\right).
			\end{equation*}
		\end{enumerate}
	\end{lemma}
	Obviously all the properties of Lemma~\ref{lem:sym} hold also if the functions are non smooth with compact support, provided that the involved integrals are finite. \newline
	
	We observe that both the quantities~$\kappa_0(u)$ and~$\defi(u,\kappa)$, given in~\eqref{eq:defk0} and~\eqref{eq:def_cfm}, respectively, are invariant under the action of the operators~$T_{z,\lambda}$.
	
	More precisely, if~$u \in \mathcal{D}^{1,p}(\R^n)$ is a weak solution to~\eqref{eq:maineq-bubb} and~$v \coloneqq T_{z,\lambda}(u)$, then~$v \in \mathcal{D}^{1,p}(\R^n)$ clearly is a weak solution to
	\begin{equation*}
		\Delta_p v + \widehat{\kappa}(x) v^{\past-1} =0 \quad \mbox{in } \R^n
	\end{equation*}
	where~$\widehat{\kappa}(x) = \kappa(\lambda(x-z))$, moreover
	\begin{equation}
	\label{eq:inv-norma}
		\widehat{\kappa} \in L^\infty(\R^n) \quad\mbox{with}\quad \norma*{\widehat{\kappa}}_{L^\infty(\R^n)}=\norma*{\kappa}_{L^\infty(\R^n)},
	\end{equation}
	and also
	\begin{equation}
	\label{eq:inv-k}
		\kappa_0(u) = \kappa_0(v) \quad\mbox{and}\quad \defi(u,\kappa) = \defi(v,\widehat{\kappa}).
	\end{equation}

	The transformations~$T_{z,\lambda}$ play a central role in the study of the Sobolev inequality, as they preserve the two quantities~$\norma*{\cdot}_{L^{\past}\!(\R^n)}$ and~$\norma*{\nabla \cdot}_{L^{p}(\R^n)}$. In particular, we will use these symmetries to reduce the problem to the case where, instead of considering a generic~$p$-bubble, we can take the fundamental bubble~$U_p[0,1]$. \newline
	
	We conclude this section by providing the proof of Corollary~\ref{cor:main-bubbles}.

	\begin{proof}[Proof of Corollary~\ref{cor:main-bubbles}]
		The proof is a direct consequence of Theorem~\ref{th:main-bubbles} and a scaling argument. Let us set~$\kappa_0=\kappa_0(u)$ and
		\begin{equation*}
			w \coloneqq \kappa_0^{\frac{1}{\past-p}} \, u.
		\end{equation*}
		It is clear that~$w \in \mathcal{D}^{1,p}(\R^n)$ is a positive weak solution to
		\begin{equation*}
			\Delta_p w + \kappa_0^{-1} \kappa(x) w^{\past-1} =0 \quad \mbox{in } \R^n.
		\end{equation*}
		Furthermore, by~\eqref{eq:ip-energ-k0},~$w$ satisfies~\eqref{eq:ipotesi-energ} and, by~\eqref{eq:defk0}, also
		\begin{equation*}
			\kappa_0(w) = \frac{\int_{\R^n} \,\abs{\nabla w}^p \, dx}{\int_{\R^n} w^{\past} dx} = \kappa_0^{-1} \,\frac{\int_{\R^n} \,\abs{\nabla u}^p \, dx}{\int_{\R^n} u^{\past} dx} = 1.
		\end{equation*}
		By applying Theorem~\ref{th:main-bubbles}, we infer that there exist a large~$C \geq 1$, a small~$\vartheta \in (0,1)$, and a~$p$-bubble~$\mathcal{U}_0$, of the form~\eqref{eq:pbubb}, such that
		\begin{equation}
		\label{eq:D1p-to-multip}
			\norma*{w-\mathcal{U}_0}_{\mathcal{D}^{1,p}(\R^n)} \leq C \defi\!\left(w,\kappa_0^{-1}\kappa\right)^{\!\vartheta} \!,
		\end{equation}
		where the constant~$C$ depends only on~$n$,~$p$,~$\norma*{\kappa}_{L^\infty(\R^n)}$, and~$\kappa_0$, whereas~$\vartheta$ depends only on~$n$ and~$p$. Therefore, we define
		\begin{equation*}
			\mathcal{U} \coloneqq \kappa_0^{\frac{1}{p-\past}} \,\mathcal{U}_0 \in \mathcal{D}^{1.p}(\R^n)
		\end{equation*}
		and observe that
		\begin{equation*}
			\defi\!\left(w,\kappa_0^{-1}\kappa\right) = \kappa_0^{\frac{p-1}{\past-p}} \defi(u,\kappa).
		\end{equation*}
		Hence, multiplying~\eqref{eq:D1p-to-multip} by~$\kappa_0^{\frac{1}{p-\past}}$, we conclude that
		\begin{equation*}
			\norma*{u-\mathcal{U}}_{\mathcal{D}^{1,p}(\R^n)} \leq C \kappa_0^{\frac{(p-1)\vartheta-1}{\past-p}} \defi(u,\kappa)^{\vartheta}.
		\end{equation*}
	\end{proof}
	
	
	\section{Proof of Theorem~\ref{th:main-bubbles}}
	\label{sec:mainsec-proofofmainres}
	
	Since~$\kappa \in L^\infty(\R^n)$, it is well-known, by the results of Peral~\cite{peral-ictp}, Serrin~\cite{serr}, DiBenedetto~\cite{diben}, and Tolksdorf~\cite{tolk}, that any solution~$u \in \mathcal{D}^{1,p}(\R^n)$ of~\eqref{eq:maineq-bubb} actually satisfies
	\begin{equation}
	\label{eq:u-C1alpha}
		u \in W^{1,\infty}(\R^n) \cap C^{1,\alpha}_{\loc}(\R^n) \quad\mbox{for some } \alpha \in (0,1).
	\end{equation}
	Moreover, by V\'etois' Lemma~2.2 in~\cite{vet}, we also have that~$u \in L^{p_\ast-1,\infty}(\R^n)$, where~$p_\ast\coloneqq\frac{p(n-1)}{n-p}$. Hence, by interpolation,
	\begin{equation}
		\label{eq:integrabilita}
		u \in L^q(\R^n) \quad\mbox{for every } q \in \left(p_\ast-1,+\infty\right].
	\end{equation}
	Furthermore, Theorem~1.1 in~\cite{vet} provides a priori estimate for~$u$ and its gradient, namely 
	\begin{equation}
		\label{eq:est-vetois}
		u(x) \leq \frac{C_u}{1+\abs*{x}^\frac{n-p}{p-1}} \quad \mbox{and} \quad \abs*{\nabla u(x)} \leq \frac{C_u}{1+\abs*{x}^\frac{n-1}{p-1}} \quad \mbox{for every } x \in \R^n
	\end{equation}
	for some constant~$C_u>0$ depending on~$u$.
	
	Concerning the regularity of the solution~$u$ to~\eqref{eq:maineq-bubb}, further results have been established by Antonini, the fist author \& Farina~\cite{carlos} -- see also the references therein. Let us define the critical set of~$u$ as
	\begin{equation*}
		\mathcal{Z}_u \coloneqq \left\{x \in \R^n \,\lvert\, \nabla u(x)=0 \right\},
	\end{equation*}
	and the \textit{stress field} associated with~$u$ by
	\begin{equation*}
		\stressu \coloneqq \abs*{\nabla u}^{p-2} \,\nabla u,
	\end{equation*}
	which is extended to zero on~$\mathcal{Z}_u$. Then, the set~$\mathcal{Z}_u$ is negligible, i.e.~$\abs*{\mathcal{Z}_u}=0$, and
	\begin{gather}
		\notag
		\stressu \in W^{1,2}_{\loc}(\R^n), \\
		\label{eq:regu}
		u \in W^{2,2}_{\loc}(\R^n \setminus \mathcal{Z}_u) \quad \mbox{and} \quad \abs*{\nabla u}^{p-2} \,\nabla^2 u \in L^{2}_{\loc}(\R^n \setminus \mathcal{Z}_u) \quad \mbox{for every } 1<p<n, \\
		\notag
		u \in W^{2,2}_{\loc}(\R^n) \quad \mbox{and} \quad \abs*{\nabla u}^{p-2} \,\nabla^2 u \in L^{2}_{\loc}(\R^n) \quad \mbox{for every } 1<p \leq 2.
	\end{gather}
	Finally, as~$\kappa \in C^{1,1}_{\loc}(\R^n)$, by Theorem~6.4 on page~284 of  Ladyzhenskaya \& Ural'tseva~\cite{lad-ur}, we have
	\begin{equation}
	\label{eq:u-C3alpha}
		u \in C^{3,\alpha}_{\loc} \!\left(\R^n \setminus \mathcal{Z}_u\right).
	\end{equation}
	
	Of course, it suffices to prove the theorem under the assumption that
	\begin{equation}
	\label{eq:def-small}
		\defi(u,\kappa) \leq \gamma,
	\end{equation}
	for some small~$\gamma \in (0,1)$, depending only on~$n$,~$p$, and~$\norma*{\kappa}_{L^\infty(\R^n)}$. Indeed, if~$\defi(u,\kappa) > \gamma$, then it follows that
	\begin{equation*}
		\norma*{u-\mathcal{U}_0}_{\mathcal{D}^{1,p}(\R^n)} \leq \norma*{\nabla u}_{L^{p}(\R^n)} + \norma*{\nabla \mathcal{U}_0}_{L^{p}(\R^n)} \leq 4 S^{\frac{n}{p}} \leq \frac{4 S^{\frac{n}{p}}}{\gamma} \defi(u,\kappa)
	\end{equation*}
	for any~$p$-bubble~$\mathcal{U}_0$, using~\eqref{eq:energ-bubb} and~\eqref{eq:ipotesi-energ}. Therefore, in what follows, we will assume that~\eqref{eq:def-small} is in force. \newline
	
	To improve readability, we divide the proof into several steps.
	
	\renewcommand\thesubsection{\bfseries Step \arabic{subsection}}
	
	
	\subsection{Application of the Struwe-type result and reduction.}
	\label{step:struwe}
	
	For any~$\epsilon>0$, by the Struwe-type result of Mercuri \& Willem~\cite{merc-will} -- see also~\cite[Theorem~2]{alves} for the case~$p>2$ -- and considering~\eqref{eq:ipotesi-energ}, there exits a~$\delta>0$, depending on~$\epsilon$, and a $p$-bubble~$U_\epsilon \coloneqq U_p[z_\epsilon,\lambda_\epsilon]$, for a couple of parameters~$z_\epsilon \in \R^n$ and~$\lambda_\epsilon>0$, such that if
	\begin{equation}
	\label{eq:delta-to-fix}
		\defi(u,\kappa) \leq \delta,
	\end{equation}
	then
	\begin{equation}
		\label{eq:struwe-0}
		\norma*{\nabla \!\left(u-U_\epsilon\right)}_{\mathcal{D}^{1,p}(\R^n)} \leq \epsilon.
	\end{equation}
	We now scale and translate our functions in order to lead~$U_\epsilon$ to be the~$p$-bubble with center the origin and unit scale factor. Taking into account point~\ref{it:transf-unitbub} of Lemma~\ref{lem:sym}, we set
	\begin{equation*}
		U \coloneqq U_p[0,1]=T_{-z_\epsilon/\lambda_\epsilon,\lambda_\epsilon} \!\left(U_\epsilon\right)
	\end{equation*}
	and define
	\begin{equation*}
		u_\epsilon \coloneqq T_{-z_\epsilon/\lambda_\epsilon,\lambda_\epsilon} \!\left(u\right).
	\end{equation*}
	Therefore, by point~\ref{it:cons-normap} of Lemma~\ref{lem:sym} and~\eqref{eq:struwe-0}, it follows that
	\begin{equation}
		\label{eq:struwe}
		\norma*{\nabla \!\left(u_\epsilon-U\right)}_{\mathcal{D}^{1,p}(\R^n)} \leq \epsilon.
	\end{equation}
	As noticed above,~$u_\epsilon$ is a weak solution to~\eqref{eq:maineq-bubb} for a different~$\kappa$ satisfying~\eqref{eq:inv-norma}--\eqref{eq:inv-k} and~$u_\epsilon$ enjoys the regularity properties listed in~\eqref{eq:regu}.
	
	Since, according to points~\ref{it:cons-normapas}--\ref{it:cons-normap} of Lemma~\ref{lem:sym} and~\eqref{eq:inv-norma}--\eqref{eq:inv-k}, all the relevant quantities are invariant, we will omit the subscript in what follows -- writing therefore~$u$ instead of~$u_\epsilon$. The subscript will be restored when needed in~\ref{step:conclusion}.
	
	\renewcommand\thesubsection{\bfseries Step \arabic{subsection}}
	
	
	\subsection{Sharp decay estimates in quantitative form.}
	\label{step:decay}
	
	We claim that, if~$\epsilon$ is properly selected, there exist two constants~$c_0,C_0>0$, depending only on~$n$,~$p$, and~$\norma*{\kappa}_{L^\infty(\R^n)}$, such that
	\begin{equation}
		\label{eq:bounds-u}
		\frac{c_0}{1+\abs*{x}^{\frac{n-p}{p-1}}} \leq u(x) \leq \frac{C_0}{1+\abs*{x}^\frac{n-p}{p-1}} \quad\mbox{for every } x \in \R^n.
	\end{equation}
	Furthermore, there exists a constant~$C_1 \geq 1$, depending only on~$n$,~$p$, and~$\norma*{\kappa}_{L^\infty(\R^n)}$, such that
	\begin{equation}
		\label{eq:bound-gradu}
		\abs*{\nabla u(x)} \leq \frac{C_1}{1+\abs*{x}^\frac{n-1}{p-1}} \quad\mbox{for every } x \in \R^n.
	\end{equation}
	Both the upper and lower bounds in~\eqref{eq:bounds-u} are essentially a consequence of~\eqref{eq:struwe} and the fact that~$u$ solves~\eqref{eq:maineq-bubb}. The upper bound is obtained via a quantification the proofs of Lemma~3.1 and Theorem~1.1 in~\cite{vet}. The lower bound is a consequence of this upper bound via the quantitative argument of Lemma~5.2 in~\cite{plap}.
	
	We start by proving a weaker version of the upper bound in~\eqref{eq:bounds-u} which is needed to reach the final sharp estimate. In particular, we first show that
	\begin{equation}
		\label{eq:ub-1}
		u(x) \leq C \abs*{x}^{\frac{p-n}{p}} \quad\mbox{for every } x \in \R^n \setminus B_{r_1},
	\end{equation}
	for some~$r_1>0$ and a constant~$C \geq 1$, depending only on~$n$,~$p$, and~$\norma*{\kappa}_{L^\infty(\R^n)}$, and where, from now on, we denote~$B_r \coloneqq B_r(0)$ for~$r>0$. By~\eqref{eq:struwe} and Sobolev inequality, we immediately get that
	\begin{equation*}
		\label{eq:normapast-esterno}
		\abs*{\norma*{u}_{L^{\past}\!(\R^n \setminus B_{r})} - \norma*{U}_{L^{\past}\!(\R^n \setminus B_{r})}} \leq \norma*{u-U}_{L^{\past}\!(\R^n \setminus B_{r})} \leq S^{-1} \epsilon,
	\end{equation*}
	which, in turn, implies
	\begin{equation}
		\label{eq:massau-infty}
		\int_{\R^n \setminus B_{r}} u^{\past} dx \leq 2^{\past-1} \int_{\R^n \setminus B_{r}} U^{\past} dx + 2^{\past-1} S^{-\past} \epsilon^{\past}
	\end{equation}
	for every~$r>0$. We now arbitrarily fix some~$\chi \in (0,1)$. Since~$U \in L^{\past}\!(\R^n)$, there exists a sufficiently large~$r_1>0$, depending only on~$n$ and~$p$, such that
	\begin{equation*}
		2^{\past-1} \int_{\R^n \setminus B_{r}} U^{\past} dx  \leq \chi \,\frac{S^n}{4} \quad\mbox{for every } r \geq r_1.
	\end{equation*}
	By further requiring that
	\begin{equation*}
		\epsilon \leq \left(2^{1-\past} \chi \,\frac{S^{n+\past}}{4}\right)^{\!\frac{1}{\past}} \eqqcolon \epsilon_0,
	\end{equation*}
	we infer from~\eqref{eq:massau-infty} that
	\begin{equation}
		\label{eq:boud-massa-ext}
		\int_{\R^n \setminus B_{r}} u^{\past} dx \leq \chi \,\frac{S^n}{2} \quad\mbox{for every } r \geq r_1.
	\end{equation}
	We can now conclude that~\eqref{eq:ub-1} holds true by Lemma~3.1 in~\cite{vet}.
	
	As second purpose, we shall prove a universal upper bound for~$u$. By~\eqref{eq:integrabilita}, we know that~$u \in L^\infty(\R^n)$. Moreover, we have already shown that~$u$ must decay at infinity, thus there exists a point~$x_0 \in \R^n$ where~$u$ attains its maximum. We now claim that
	\begin{equation}
		\label{eq:bound-Linf-univ}
		\norma*{u}_{L^\infty(\R^n)} = u(x_0) \leq \mathscr{M}
	\end{equation}
	for some~$\mathscr{M} \geq 1$ depending only on~$n$,~$p$, and~$\norma*{\kappa}_{L^\infty(\R^n)}$. Our strategy to prove~\eqref{eq:bound-Linf-univ} is that of exploiting~\eqref{eq:struwe} and the proof of Theorem~E.0.20 in~\cite{peral-ictp} in order to derive a universal upper bound for a suitably large Lebesgue norm of~$u$. Ultimately, we will deduce~\eqref{eq:bound-Linf-univ} by applying Theorem~1 in~\cite{serr}. For the sake of completeness we reproduce the full argument in~\cite{peral-ictp}.
	
	To this end, let us define, for~$k \geq 0$, the functions~$F_k, G_k: [0,+\infty) \to \R$ by
	\begin{equation*}
		F_k(t) \coloneqq
		\begin{dcases}
			t^\beta                                 		& \quad \mbox{if } t \leq k, \\
			\beta k^{\beta-1} \left(t-k\right) +k^\beta		& \quad \mbox{if } t > k, \\
		\end{dcases}
	\end{equation*}
	and
	\begin{equation*}
		G_k(t) \coloneqq
		\begin{dcases}
			t^{\left(\beta-1\right)p+1}   & \quad \mbox{if } t \leq k, \\
			\left(\left(\beta-1\right)p+1 \right)\beta \, k^{\left(\beta-1\right)p} \left(t-k\right)+k^{\left(\beta-1\right)p+1}    & \quad \mbox{if } t > k,
		\end{dcases}
	\end{equation*}
	for some~$\beta >1$ to be determined soon. Note that~$F_k,G_k \in C^{0,1}([0,+\infty))$, moreover one can check that the following inequalities hold true
	\begin{align}
		\label{eq:rel1}
		G_k(t) &\leq t G_k'(t), \\
		\label{eq:rel2}
		c_{\beta,p} \left(F_k'(t)\right)^p &\leq G_k'(t), \\
		\label{eq:rel3}
		t^{p-1} G_k(t) &\leq C_{\beta,p} \left(F_k(t)\right)^p \!,
	\end{align}
	for a couple of constants~$c_{\beta,p},C_{\beta,p}>0$ depending only on~$\beta$ and~$p$. Fix~$\beta >1$, depending on~$n$ and~$p$, such that~$\beta p < \past$, and consider
	\begin{equation*}
		\xi_1 \coloneqq \eta^p G_k(u),
	\end{equation*}
	where~$\eta \in C^\infty_c(\R^n)$ will be chosen later. We can take~$\xi_1$ as a test function in~\eqref{eq:maineq-bubb}, so that
	\begin{align}
		\label{eq:test-1}
			\int_{\R^n} &\,\abs*{\nabla u}^{p-2}  \left\langle \nabla u, \nabla\xi_1 \right\rangle dx \\
		\notag
			&= \int_{\R^n} p \,\abs*{\nabla u}^{p-2}  \eta^{p-1} G_k(u) \left\langle \nabla u, \nabla\eta \right\rangle + \abs*{\nabla u}^{p} \eta^p G'_k(u) \, dx = \int_{\R^n} \kappa u^{\past-1} \eta^p G_k(u) \, dx.
	\end{align}
	By exploiting~\eqref{eq:rel1} and Young's inequality we get
	\begin{equation*}
		\begin{split}
			&\abs*{\int_{\R^n} p \,\abs*{\nabla u}^{p-2}  \eta^{p-1} G_k(u) \left\langle \nabla u, \nabla\eta \right\rangle dx} \\
			&\quad\leq p \int_{\R^n} \,\abs{\nabla u}^{p-1} \eta^{p-1} (G_k'(u))^{1/p'} \, \left(u^{p-1}G_k(u)\right)^{\!1/p} \abs{\nabla \eta} \, dx \\
			&\quad\leq \left(p-1\right)\sigma \int_{\R^n} \,\abs{\nabla u}^{p} \eta^{p} G_k'(u) \, dx + \frac{1}{\sigma} \int_{\R^n} u^{p-1} G_k(u) \,\abs{\nabla \eta}^p \, dx.
		\end{split}
	\end{equation*}
	By taking~$\sigma=\left(2\left(p-1\right)\right)^{-1}$, the latter, together with~\eqref{eq:test-1}, yields
	\begin{equation}
		\label{eq:test-2}
		\begin{split}
			\int_{\R^n} &\,\abs{\nabla u}^p \eta^p G_k'(u) \, dx\\
			&\leq 4\left(p-1\right) \int_{\R^n} u^{p-1} G_k(u) \abs{\nabla \eta}^p \, dx + 2\norma*{\kappa}_{L^\infty(\R^n)} \int_{\R^n} u^{\past-1} G_k(u) \eta^p \, dx \\
			&\leq 4C_{\beta,p}\left(p-1\right) \int_{\R^n} F_k^p(u) \abs{\nabla \eta}^p \, dx + 2C_{\beta,p} \norma*{\kappa}_{L^\infty(\R^n)} \int_{\R^n} u^{\past-p} F_k^p(u) \eta^p \, dx,
		\end{split} 
	\end{equation}
	where we used~\eqref{eq:rel3} in the last line. Taking advantage of~\eqref{eq:rel2}, we can estimate the left-hand side of~\eqref{eq:test-2} getting
	\begin{align*}
		c_{\beta,p} \int_{\R^n} &\,\abs{\nabla u}^p \eta^p \left(F_k'(u)\right)^p \, dx\\
		&\leq 4C_{\beta,p}\left(p-1\right) \int_{\R^n} F_k^p(u) \abs{\nabla \eta}^p \, dx + 2C_{\beta,p} \norma*{\kappa}_{L^\infty(\R^n)} \int_{\R^n} u^{\past-p} F_k^p(u) \eta^p \, dx,
	\end{align*}
	which, in turn, implies
	\begin{equation*}
		\int_{\R^n} \,\abs{\nabla \left(\eta F_k(u)\right)}^p \, dx \leq C_2 \int_{\R^n} F_k^p(u) \abs{\nabla \eta}^p \, dx + C_3 \int_{\R^n} u^{\past-p} F_k^p(u) \eta^p \, dx
	\end{equation*}
	for some~$C_2>0$, depending only on~$n$ and~$p$, and with~$C_3 \coloneqq 2 \frac{C_{\beta,p}}{c_{\beta,p}} \norma*{\kappa}_{L^\infty(\R^n)}$.
	By Sobolev inequality we finally deduce
	\begin{equation}
		\label{eq:Fk-past}
		\left(\int_{\R^n} F_k^{\past}\!(u) \eta^{\past} dx \right)^{\!\frac{p}{\past}} \leq S^{-p}  C_2 \int_{\R^n} F_k^p(u) \abs{\nabla \eta}^p \, dx + S^{-p} C_3 \int_{\R^n} u^{\past-p} F_k^p(u) \eta^p \, dx.
	\end{equation}
	We now choose~$\eta \in C^\infty_c(\R^n)$ such that~$\mathrm{supp} \,\eta = B_{2 r_2}(x_0)$, for some universal~$r_2>0$ to be determined in such a way that
	\begin{equation}
		\label{eq:r2-tba}
		\norma*{u}^{\past-p}_{L^{\past}\!(B_{2 r_2}(x_0))} \leq \frac{S^p}{2C_3}.
	\end{equation}
	To prove the existence of an~$r_2>0$ for which~\eqref{eq:r2-tba} holds true, we argue as above to deduce
	\begin{equation*}
		\norma*{u}^{\past-p}_{L^{\past}\!(B_{r}(x_0))} \leq 2^{\max\left\{0,\past-p-1\right\}} \left(\norma*{U}^{\past-p}_{L^{\past}\!(B_{r}(x_0))} +S^{p-\past} \epsilon^{\past-p}\right) \quad\mbox{for every } r>0.
	\end{equation*}
	We shall choose~$r>0$ in such a way that
	\begin{equation}
		\label{eq:stima-Upast-r}
		\norma*{u}^{\past-p}_{L^{\past}\!(B_{r}(x_0))} \leq 2^{\max\left\{0,\past-p-1\right\}} \norma*{U}^{\past-p}_{L^{\past}\!(B_{r}(x_0))} \leq \frac{S^p}{4C_3},
	\end{equation}
	and
	\begin{equation*}
		\epsilon \leq \left(\frac{S^{\past}}{2^{\max\left\{2,\past-p+1\right\}} \, C_3}\right)^{\!\frac{1}{\past-p}} \eqqcolon \epsilon_1.
	\end{equation*}
	Since we can determine an~$r_2>0$, depending only on~$n$,~$p$, and~$\norma*{\kappa}_{L^\infty(\R^n)}$, such that
	\begin{equation*}
		\norma*{U}^{\past-p}_{L^{\past}\!(B_{2r_2}(x_0))} \leq \norma*{U}^{\past-p}_{L^{\past}\!(B_{2r_2})} \leq \frac{S^p}{ 2^{\max\left\{2,\past-p+1\right\}} \, C_3},
	\end{equation*}
	the second inequality in~\eqref{eq:stima-Upast-r} holds, and we conclude that~\eqref{eq:r2-tba} is verified. By H\"older inequality and~\eqref{eq:r2-tba}, we finally infer form~\eqref{eq:Fk-past} that
	\begin{equation*}
		\left(\int_{\R^n} F_k^{\past}\!(u) \eta^{\past} dx \right)^{\!\frac{p}{\past}} \leq 2S^{-p}  C_2 \int_{\R^n} F_k^p(u) \abs{\nabla \eta}^p \, dx,
	\end{equation*}
	and, taking the limit as~$k \to +\infty$, the monotone convergence theorem ensures that
	\begin{equation*}
		\left(\int_{\R^n} u^{\beta \past} \eta^{\past} dx \right)^{\!\frac{p}{\past}} \leq 2S^{-p}  C_2 \int_{\R^n} u^{\beta p} \abs{\nabla \eta}^p \, dx.
	\end{equation*}
	If we now suppose that~$\eta = 1$ in~$B_{r_2}(x_0)$, recalling that~$\beta p < \past$, we deduce
	\begin{equation*}
		\norma*{u}_{L^{\beta\past}\!(B_{r_2}(x_0))} \leq C,
	\end{equation*}
	for some~$C>0$ depending only on~$n$,~$p$, and~$\norma*{\kappa}_{L^\infty(\R^n)}$. With this estimate at hand, we can apply Theorem~1 in~\cite{serr}, from which~\eqref{eq:bound-Linf-univ} directly follows.
	
	As in the proof of Theorem~1.1 in~\cite{vet}, we now show that the upper bound~\eqref{eq:ub-1} can be promoted to the stronger one in~\eqref{eq:bounds-u}. We first claim that
	\begin{equation}
		\label{eq:norma-past-1}
		\norma*{u}_{L^{\past-1}(\R^n)} \leq C,
	\end{equation}
	for some~$C >0$ depending only on~$n$,~$p$, and~$\norma*{\kappa}_{L^\infty(\R^n)}$. We immediately note that, by~\eqref{eq:norma-past-1} and Lemma~2.2 in~\cite{vet}, we have
	\begin{equation}
		\label{eq:norma-weak-leb}
		\norma*{u}_{L^{p_\ast-1,\infty}(\R^n)} \leq C,
	\end{equation}
	for some~$C >0$ depending only on~$n$,~$p$, and~$\norma*{\kappa}_{L^\infty(\R^n)}$.
	
	We plan to achieve~\eqref{eq:norma-past-1} by adapting to our case a procedure which originates from Brezis \& Kato~\cite{brezis-kato}. To this aim, let~$\psi \in C^\infty_c(\R^n)$ be a cut-off function such that~$\psi \equiv 1$ in~$B_r(-z_\epsilon)$,~$\psi \equiv 0$ in~$\R^n \setminus B_{2r}(-z_\epsilon)$, and~$\abs*{\nabla \psi} \leq 2/r$ in~$B_{2r}(-z_\epsilon) \setminus B_{r}(-z_\epsilon)$. We now test~\eqref{eq:maineq-bubb} with
	\begin{equation*}
		\xi_2 \coloneqq u^\sigma \psi
	\end{equation*}
	for~$\sigma>0$, and get
	\begin{align*}
		\int_{\R^n} &\,\abs{\nabla u}^{p-2} \left\langle \nabla u, \nabla \xi_2 \right\rangle dx \\
		&= \int_{\R^n} \sigma \,\abs*{\nabla u}^{p}  u^{\sigma-1} \psi + \abs*{\nabla u}^{p-2} u^\sigma \left\langle \nabla u, \nabla \psi \right\rangle dx = \int_{\R^n} \kappa u^{\past-1+\sigma} \psi \, dx.
	\end{align*}
	Thus, we deduce
	\begin{align}
	\notag
		\sigma \left(\frac{p}{p+\sigma-1}\right)^{\! p} \int_{\R^n} \,\abs*{\nabla \!\left(u^{\frac{p+\sigma-1}{p}}\right)}^p \psi \, dx + \int_{\R^n} &\,\abs*{\nabla u}^{p-2} u^\sigma \left\langle \nabla u, \nabla \psi \right\rangle dx \\
	\label{eq:passare-limite}
		&= \int_{\R^n} \kappa u^{\past-1+\sigma} \psi \, dx.
	\end{align}
	We observe that, for~$r$ large enough, we have
	\begin{equation*}
		\begin{split}
			\abs*{\int_{\R^n} \,\abs*{\nabla u}^{p-2} u^\sigma \left\langle \nabla u, \nabla \psi \right\rangle dx} &\leq \int_{B_{2r}(-z_\epsilon) \setminus B_{r}(-z_\epsilon)} \,\abs*{\nabla u}^{p-1} \abs*{\nabla \psi} \, u^\sigma \, dx \\
			&\leq \frac{C}{r} \int_{B_{2r}(0) \setminus B_{r}(0)} \frac{dx}{\abs*{x}^{n+\frac{n-p}{p-1} \sigma-1}} \leq C \Haus^{n-1}(\sfera^{n-1}) \, r^{-\frac{n-p}{p-1} \sigma},
		\end{split}
	\end{equation*}
	for some~$C>0$ depending on~$n$,~$p$,~$u$, and~$\epsilon$, where we used a properly rescaled and translated version of~\eqref{eq:est-vetois}. Therefore, letting~$r \to +\infty$ in~\eqref{eq:passare-limite} and taking into account~\eqref{eq:integrabilita},  the monotone convergence theorem ensures
	\begin{equation}
		\label{eq:passata-limite}
		\sigma \left(\frac{p}{p+\sigma-1}\right)^{\! p} \int_{\R^n} \,\abs*{\nabla \!\left(u^{\frac{p+\sigma-1}{p}}\right)}^p \, dx = \int_{\R^n} \kappa u^{\past-1+\sigma} \, dx.
	\end{equation}
	Since~$\left(p+\sigma-1\right) \past > \left(p_\ast-1\right) p$ for every~$\sigma>0$, as a consequence of~\eqref{eq:integrabilita} and~\eqref{eq:passata-limite}, we deduce that~$u^{\frac{p+\sigma-1}{p}} \in \mathcal{D}^{1,p}(\R^n)$. Hence, Sobolev inequality and~\eqref{eq:passata-limite} yield
	\begin{equation*}
		\sigma \left(\frac{p \, S}{p+\sigma-1}\right)^{\! p} \left(\int_{\R^n} u^{\frac{p+\sigma-1}{p} \past} dx\right)^{\!\!\frac{p}{\past}} \leq \int_{\R^n} \kappa u^{\past-1+\sigma} \, dx.
	\end{equation*}
	Let~$r>0$ to be chosen shortly. By H\"older inequality, the previous estimate entails
	\begin{multline}
	\label{eq:toreabsord}
		\sigma \left(\frac{p \, S}{p+\sigma-1}\right)^{\! p} \left(\int_{\R^n} u^{\frac{p+\sigma-1}{p} \past} dx\right)^{\!\!\frac{p}{\past}} \\
		\leq \int_{B_r} \kappa u^{\past-1+\sigma} \, dx + \norma*{\kappa}_{L^\infty(\R^n)} \left(\int_{\R^n \setminus B_r} u^{\frac{p+\sigma-1}{p} \past} \, dx\right)^{\!\!\frac{p}{\past}} \left(\int_{\R^n \setminus B_r} u^{\past} dx\right)^{\!\!\frac{\past-p}{\past}} \!.
	\end{multline}
	By taking advantage of the bound~\eqref{eq:bound-Linf-univ}, for the first integral on the right-hand side of~\eqref{eq:toreabsord} we deduce
	\begin{equation}
		\label{eq:stima-primo-int}
		\int_{B_r} \kappa u^{\past-1+\sigma} \, dx \leq  \mathscr{M}^{\past-1+\sigma} \norma*{\kappa}_{L^\infty(\R^n)} \abs*{B_1} \, r^n.
	\end{equation}
	To handle the second summand on the right-hand side of~\eqref{eq:toreabsord} we start by noticing that
	\begin{equation*}
		\norma*{u}_{L^{\past}\!(\R^n \setminus B_{r})}^{\past-p} \leq
		2^{\max\left\{0,\past-p-1\right\}} \left(\norma*{U}_{L^{\past}\!(\R^n \setminus B_{r})}^{\past-p} +S^{p-\past} \epsilon^{\past-p}\right).
	\end{equation*}
	We now choose~$\epsilon>0$ sufficiently small and~$r>0$ large enough in such a way that the second term on the right-hand side of~\eqref{eq:toreabsord} can be reabsorbed. Specifically, we require that
	\begin{equation}
		\label{eq:ep-r-tochose}
		2^{\max\left\{0,\past-p-1\right\}} \norma*{\kappa}_{L^\infty(\R^n)} \left(\norma*{U}_{L^{\past}\!(\R^n \setminus B_{r})}^{\past-p} +S^{p-\past} \epsilon^{\past-p}\right) \leq \frac{\sigma}{2} \left(\frac{p \, S}{p+\sigma-1}\right)^{\! p}\!.
	\end{equation}
	Note that we will choose~$\sigma>0$ shortly depending only on~$n$ and~$p$. By assuming that
	\begin{equation*}
		\epsilon \leq \left\{\frac{\sigma S^{\past-p} }{2^{\max\left\{2,\past-p+1\right\}} \norma*{\kappa}_{L^\infty(\R^n)}} \left(\frac{p \, S}{p+\sigma-1}\right)^{\! p}\right\}^{\!\frac{1}{\past-p}} \eqqcolon \epsilon_2,
	\end{equation*}
	which depends only on~$n$,~$p$, and~$\norma*{\kappa}_{L^\infty(\R^n)}$, and since there exists~$r_3>0$, depending only on~$n$,~$p$, and~$\norma*{\kappa}_{L^\infty(\R^n)}$, such that
	\begin{equation*}
		2^{\max\left\{0,\past-p-1\right\}} \norma*{\kappa}_{L^\infty(\R^n)} \norma*{U}_{L^{\past}\!(\R^n \setminus B_{r_3})}^{\past-p} \leq \frac{\sigma}{4} \left(\frac{p \, S}{p+\sigma-1}\right)^{\! p},
	\end{equation*}
	we infer that~\eqref{eq:ep-r-tochose} holds true. Therefore, combining~\eqref{eq:toreabsord},~\eqref{eq:stima-primo-int}, and~\eqref{eq:ep-r-tochose} we get
	\begin{equation*}
		\frac{\sigma}{2} \left(\frac{p \, S}{p+\sigma-1}\right)^{\! p} \left(\int_{\R^n} u^{\frac{p+\sigma-1}{p} \past} dx\right)^{\!\!\frac{p}{\past}} \leq C,
	\end{equation*}
	for some~$C >0$ depending only on~$n$,~$p$, and~$\norma*{\kappa}_{L^\infty(\R^n)}$. Hence,~\eqref{eq:norma-past-1} follows by choosing~$\sigma=1-\frac{p}{\past}$ in the previous estimate. We remark that the above approximation procedure used to obtain~\eqref{eq:passata-limite} is necessary.
	
	As a consequence of~\eqref{eq:bound-Linf-univ},~\eqref{eq:norma-weak-leb}, and of the proof of Theorem~1.1 in~\cite{vet}, we deduce the upper bound in~\eqref{eq:bounds-u}.
	
	Finally, the lower bound in~\eqref{eq:bounds-u} is obtained by noticing that, from~\eqref{eq:ipotesi-energ},~\eqref{eq:boud-massa-ext}, and the assumption~$\kappa_0=1$, we get
	\begin{equation*}
		\int_{B_{r_1}} u^{\past} dx \geq \left(1-\chi\right) \frac{S^n}{2}
	\end{equation*}
	and arguing as in Lemma~5.2 of~\cite{plap} -- see also~\cite[Lemma~2.3]{serr-zou}.
	
	Overall, by setting~$\epsilon\coloneqq \min\left\{\epsilon_0,\epsilon_1,\epsilon_2\right\}$, we deduce that~\eqref{eq:bounds-u} holds true. This also fixes the value of~$\delta$ in~\eqref{eq:delta-to-fix}. To prove the gradient bound, we observe that~\eqref{eq:bound-Linf-univ} and Theorem~1 in~\cite{diben} yield
	\begin{equation*}
		\norma*{u}_{C^1(\R^n)} \leq C_1,
	\end{equation*}
	for some~$C_1 \geq 1$, depending only on~$n$,~$p$, and~$\norma*{\kappa}_{L^\infty(\R^n)}$. Therefore, the validity of~\eqref{eq:bound-gradu} follows from the proof of Theorem~1.1 in~\cite{vet}, up to possibly enlarging~$C_1$.
	
	
	\subsection{Introduction of the~$P$-function and of the relevant vector fields.}
	\label{step:Pfunct}
	
	As~$u>0$, we now introduce the function
	\begin{equation}
	\label{eq:defv}
		v \coloneqq u^{-\frac{p}{n-p}}.
	\end{equation}
	A pictorial representation of the construction we consider is given in Figure~\ref{fig:functions-E}.
	\begin{figure}
		\centering
		\begin{tikzpicture}[scale=1.2]
			\draw plot[smooth] file {pgfplots/pgfmanual-bub.table};
			\draw[dashed] plot[smooth] file {pgfplots/pgfmanual-quadsotto.table};
			\draw plot[smooth,yshift=0.2cm] file {pgfplots/pgfmanual-invbub.table};
			\draw[dashed] plot[smooth] file {pgfplots/pgfmanual-quadsopra.table};
			\draw[blue,densely dotted] plot[smooth,yshift=0.2cm] file {pgfplots/pgfmanual-2supfunct.table};
			
			\draw[dotted] (0.5,2.632148026) -- (0.5,0);
			\draw[rotate around={-45:(-2.8,0.199835391)},teal] (-2.79,0.3) -- (-2.79,0.1);
			\draw[rotate around={-45:(3.39,0.199835391)},teal] (3.39,0.3) -- (3.39,0.1);
			\draw[teal] (-2.8,0.199835391) -- (3.39,0.199835391);
			
			\draw (0.5,0) node [anchor=north][inner sep=0.75pt]  [align=center] {$x_0$};
			\draw (-0.375,0.9) node [anchor=south][inner sep=0.75pt]  [align=center] {$u$};
			\draw (-0.375,2.3) node [anchor=south][inner sep=0.75pt]  [align=center] {$v$};
			\draw (2,0.17) node [anchor=north][inner sep=0.75pt]  [align=center] {$B_r$};
			\draw (-3.15,0.2) node [anchor=south][inner sep=0.75pt]  [align=center] {$1/P$};
		\end{tikzpicture}
		\caption{The functions considered and the region~$B_r$ where the energy is concentrated. The dotted blue line is~$1/P$ up to the correct factor.}
		\label{fig:functions-E}
	\end{figure}
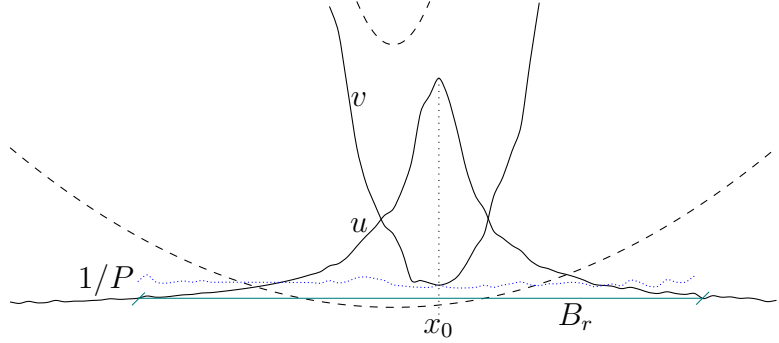
	From this definition and~\eqref{eq:maineq-bubb}, it follows that~$v$ weakly solves
	\begin{equation}
		\label{eq:eqforv}
		\Delta_p v = P + R \quad \mbox{in } \R^n,
	\end{equation}
	where
	\begin{gather}
	\label{eq:defPfunct}
		P \coloneqq n \,\frac{p-1}{p} v^{-1} \,\abs*{\nabla v}^p + \left(\frac{p}{n-p}\right)^{\! p-1} v^{-1}, \\
	\label{eq:defRem}
		R \coloneqq \left(\frac{p}{n-p}\right)^{\! p-1} \left(\kappa-1\right) v^{-1}
	\end{gather}
	The function defined in~\eqref{eq:defPfunct} is the so-called~$P$-function exploited in~\cite{ou,vet-plap} to obtain the classification result for local weak solutions of the~$p$-Laplace equation and in~\cite{cami} for the classification of solution to semilinear PDEs on Riemmanian manifolds. Whereas~\eqref{eq:defRem} represents the remainder term arising from the perturbation, i.e., due to the fact that~$\kappa$ is not necessarily constant.
	
	We first observe that
	\begin{equation}
	\label{eq:D1p}
		u \in \mathcal{D}^{1,p}(\R^n) \quad\mbox{if and only if}\quad \int_{\R^n} v^{-n} \, dx +  \int_{\R^n} v^{-n} \,\abs*{\nabla v}^p \, dx < +\infty.
	\end{equation}
	Furthermore,~$v$ inherits some regularity properties from those of~$u$ listed in~\eqref{eq:u-C1alpha} and in~\eqref{eq:regu}--\eqref{eq:u-C3alpha}, more precisely
	\begin{gather}
	\notag
		\mathcal{Z}_v = \mathcal{Z}_u, \\
	\label{eq:regv}
		v \in C^{1,\alpha}_{\loc}(\R^n) \cap C^{3,\alpha}_{\loc} \!\left(\R^n \setminus \mathcal{Z}_v\right) \quad \mbox{and} \quad \stressv \in W^{1,2}_{\loc}(\R^n), \\
	\notag
		v \in W^{2,2}_{\loc}(\R^n \setminus \mathcal{Z}_v) \quad \mbox{and} \quad \abs*{\nabla v}^{p-2} \,\nabla^2 v \in L^{2}_{\loc}(\R^n \setminus \mathcal{Z}_v) \quad \mbox{for every } 1<p<n, \\
	\notag
		v \in W^{2,2}_{\loc}(\R^n) \quad \mbox{and} \quad \abs*{\nabla v}^{p-2} \,\nabla^2 v \in L^{2}_{\loc}(\R^n) \quad \mbox{for every } 1<p \leq 2.
	\end{gather}
	We claim that~\eqref{eq:bounds-u} and~\eqref{eq:bound-gradu} give some pointwise bounds for~$v$, its gradient, and the $P$-function. Indeed, from the sharp decay estimates~\eqref{eq:bounds-u}, we deduce
	\begin{equation}
	\label{eq:bounds-v}
		\hat{c}_0 \left(1+\abs*{x}^\frac{p}{p-1}\right) \leq v(x) \leq \widehat{C}_0 \left(1+\abs*{x}^\frac{p}{p-1}\right) \quad\mbox{for every } x \in \R^n.
	\end{equation}
	for a couple of constants~$\hat{c}_0,\widehat{C}_0>0$ depending only on~$n$,~$p$, and~$\norma*{\kappa}_{L^\infty(\R^n)}$. Moreover, by~\eqref{eq:bound-gradu} and~\eqref{eq:bounds-v}, we get
	\begin{equation}
		\label{eq:bound-gradv}
		\abs*{\nabla v(x)} \leq \widehat{C}_1 \left(1+\abs*{x}^\frac{1}{p-1}\right) \quad\mbox{for every } x \in \R^n,
	\end{equation}
	for some~$\widehat{C}_1 \geq 1$ depending only on~$n$,~$p$, and~$\norma*{\kappa}_{L^\infty(\R^n)}$. As a consequence of~\eqref{eq:bounds-v} and~\eqref{eq:bound-gradv}, it follows that~$v \geq \hat{c}_0$ and
	\begin{equation}
	\label{eq:bound-fgrad}
		v^{-1} \,\abs*{\nabla v}^p \in L^\infty(\R^n) \quad\mbox{with}\quad v^{-1} \,\abs*{\nabla v}^p \leq \widehat{C}_1^p \,\hat{c}_0^{-1},
	\end{equation}
	thus
	\begin{gather}
	\label{eq:P-Linf}
		P \in L^\infty(\R^n) \quad\mbox{with}\quad \norma*{P}_{L^\infty(\R^n)} \leq n \,\frac{p-1}{p} \,\widehat{C}_1^p \,\hat{c}_0^{-1} + \left(\frac{p}{n-p}\right)^{\! p-1} \hat{c}_0^{-1}, \\
	\label{eq:R-Linf}
		R \in L^\infty(\R^n) \quad\mbox{with}\quad \norma*{R}_{L^\infty(\R^n)} \leq \left(\frac{p}{n-p}\right)^{\! p-1} \hat{c}_0^{-1} \,\norma*{\kappa-1}_{L^\infty(\R^n)}.
	\end{gather}

	For simplicity of notation, we also define the function~$V: \R^n \to [0,+\infty)$ by
	\begin{equation*}
		V \!\left(\xi\right) \coloneqq \frac{\abs*{\xi}^p}{p}.
	\end{equation*}
	Note that~$a(\xi)= \nabla V(\xi)$. The chain rule,~\eqref{eq:u-C1alpha},~\eqref{eq:regu}--\eqref{eq:u-C3alpha},~\eqref{eq:defPfunct}, and~\eqref{eq:regv} entail
	\begin{gather}
	\notag
		V \!\left(\nabla v\right) \in W^{1,2}_{\loc}(\R^n) \cap C^{0,\alpha}_{\loc}(\R^n) \cap C^{2,\alpha}_{\loc}(\R^n \setminus \mathcal{Z}_v), \\
	\label{eq:regP}
		P \in W^{1,2}_{\loc}(\R^n) \cap C^{0,\alpha}_{\loc}(\R^n) \cap C^{2,\alpha}_{\loc}(\R^n \setminus \mathcal{Z}_v).
	\end{gather}
	
	Finally, we introduce the relevant vector field
	\begin{equation*}
	\label{eq:def-gradA}
		W \coloneqq \nabla \stressv,
	\end{equation*}
	and extend it to zero on~$\mathcal{Z}_v$. We also define the associated traceless version as
	\begin{equation}
	\label{eq:tracless-W}
		\mathring{W} = \mathring{\nabla} \stressv \coloneqq \nabla \stressv - \frac{\tr \nabla \stressv}{n} \Id,
	\end{equation}
	where~$\Id$ is the~$n$-dimensional identity matrix. This vector field will play a central role in our proof.
	
	From~\eqref{eq:regv}, we deduce that
	\begin{equation*}
		W \in L^2_{\loc}(\R^n) \cap C^{1,\alpha}_{\loc} \!\left(\R^n \setminus \mathcal{Z}_v\right).
	\end{equation*}
	From the chain rule and~\eqref{eq:regv}, we can also write
	\begin{equation}
	\label{eq:defW}
		W(x)=
		\begin{dcases}
			A(x) D^2v(x)	& \quad \mbox{if } x \in \R^n \setminus \mathcal{Z}_v, \\
			0				& \quad \mbox{if } x \in \mathcal{Z}_v,
		\end{dcases}
	\end{equation}
	where~$A$ is the matrix with components
	\begin{equation*}
		\left[A(x)\right]_{ij} = \alpha_{ij}(x) \coloneqq \partial_{\xi_i\xi_j} V \!\left(\nabla v\right)\!(x) \quad\mbox{for } x \in \R^n \setminus \mathcal{Z}_v,
	\end{equation*}
	that is
	\begin{equation}
	\label{eq:def-A}
		A = \abs*{\nabla v}^{p-2} \Id + \left(p-2\right) \abs*{\nabla v}^{p-4} \,\nabla v \otimes \nabla v \quad\mbox{on } \R^n \setminus \mathcal{Z}_v.
	\end{equation}
	In particular,~$A$ is well-defined and of class~$C_{\loc}^{2,\alpha}$ in~$\R^n \setminus \mathcal{Z}_v$, which has full measure, hence~$A$ is measurable and we extend it to zero on~$\mathcal{Z}_v$.
	
	With this notation at hand, we can rewrite~\eqref{eq:eqforv} as
	\begin{equation}
	\label{eq:trW}
		\tr W = \alpha_{ij} v_{ij} = P + R \quad \mbox{a.e.\ in } \R^n.
	\end{equation}

	We conclude by noticing that a direct computation -- see also~\cite[Lemma~2.1]{ou} -- reveals
	\begin{equation}
	\label{eq:gradP}
		\nabla P = \frac{n}{v} \left(\nabla \stressv^{\!\mathsf{T}} - \frac{P}{n} \Id\right) \nabla v = \frac{n}{v} \left(W^\mathsf{T} - \frac{P}{n} \Id\right) \nabla v \quad \mbox{on } \R^n \setminus \mathcal{Z}_v.
	\end{equation}

	
	\subsection{A differential identity.}
	\label{step:differ-id}
	
	In this step, we prove a differential identity that holds for sufficiently smooth functions. This identity will be used in~\ref{step:approx}.
	
	\begin{proposition}
	\label{prop:id-diff}
		Let~$\Omega\subseteq \R^n$ be an open set and~$w\in C^3(\Omega)$ be a positive function. Let~$\mathscr{V}:\R^n \to [0,+\infty)$ be of class $ C^3(\R^n\setminus\{0\})$, and define
		\begin{equation*}
			\mathsf{P} \coloneqq \frac{n \left(p-1\right)}{w} \, \mathscr{V}\!\left(\nabla w\right) + \left(\frac{p}{n-p}\right)^{\! p-1} \frac{1}{w}
		\end{equation*}
		Then, by setting
		\begin{equation*}
			\mathsf{a}\!\left(\nabla w\right)\coloneqq \nabla_\xi \mathscr{V}\!\left(\nabla w\right),\quad \left[\mathsf{A}\right]_{ij} =\tilde{\alpha}_{ij}\coloneqq\partial_{\xi_i\xi_j}\mathscr{V}\!\left(\nabla w\right),\quad\text{and}\quad \mathscr{W}  \coloneqq\nabla \mathsf{a}\!\left(\nabla w\right),
		\end{equation*}
		the following differential identity holds in~$\Omega\setminus\{x \in \Omega \,\lvert\, \nabla w(x)=0\}$
		\begin{multline}
		\label{eq:id-diff}
			\mathrm{div}\left(w^{2-n}\mathsf{A}\,\nabla \mathsf{P} \right)=w^{1-n}\Bigl\{-n \left\langle \mathsf{A}\,\nabla\mathsf{P}, \nabla w\right\rangle-\mathsf{P} \tr\mathscr{W}
			\\
			+n\left(p-1\right)\left[\tr\mathscr{W}^2+ \left\langle \nabla\!\left(\tr\mathscr{W}\right), \mathsf{a}\!\left(\nabla w\right) \right\rangle \right]-\mathsf{P} w_j\partial_{\xi_i\xi_j\xi_k}\mathscr{V}\!\left(\nabla w\right) w_{ki}\Bigr\}.
		\end{multline}
	\end{proposition}
	\begin{proof}
		In the proof all the calculations are done in the open set~$\Omega\setminus\{x \in \Omega \,\lvert\, \nabla w(x)\neq 0\}$, where~$\mathsf{P}$ is of class~$C^2$, thanks to our assumptions of~$w$ and~$\mathscr{V}$.
		
		By the chain rule, we have
		\begin{gather}
		\label{eq:deriv-V}
			\partial_j \mathscr{V}\!\left(\nabla w\right)=\mathsf{a}_k \!\left(\nabla w\right) w_{kj} \\
		\label{eq:deriv-W}
			\mathscr{W}_{ik}=\partial_k \mathsf{a}_i(\nabla w)=\tilde{\alpha}_{ij} w_{jk},
		\end{gather}
		where~$\mathsf{a}_k(\xi)$ denotes the~$k$-th component of~$\mathsf{a}(\xi)$.
		
		Then, we compute
		\begin{equation}
		\label{eq:Pi}
			\mathsf{P}_{\! j}=
			\frac{n \left(p-1\right)}{w}\, \partial_j \mathscr{V}\!\left(\nabla w\right) - \frac{\mathsf{P}}{w}\, w_j,
		\end{equation}
		so that
		\begin{align}
		\notag
			\mathsf{P}_{\! ij}&=-\frac{\mathsf{P}_{\! i} w_j}{w}+\frac{\mathsf{P}}{w^2}\, w_{i}w_j-\frac{\mathsf{P}}{w}\, w_{ij}-\frac{n\left(p-1\right)}{w^2}\, w_i\partial_{j}\mathscr{V}(\nabla w)+\frac{n\left(p-1\right)}{w}\, \partial_{ij}\mathscr{V}(\nabla w) \\
		\label{eq:Pij}
			&=-\frac{\mathsf{P}_{\! i} w_j}{w}-\frac{\mathsf{P}_{\! j} w_i}{w}-\frac{\mathsf{P}}{w}\, w_{ij}+\frac{n\left(p-1\right)}{w}\, \partial_{ij}\mathscr{V}(\nabla w),
		\end{align}
		where in the last equality we used~\eqref{eq:Pi}.
		
		Our goal is to compute
		\begin{equation}
		\label{eq:diver-to-comp}
			\begin{split}
				\diver\!\left(w^{2-n}\tilde{\alpha}_{ij} \mathsf{P}_{\! j}  \right)&=w^{1-n}\left[w\,\tilde{\alpha}_{ij}\mathsf{P}_{\! ij}+w\,\partial_i\tilde{\alpha}_{ij}\mathsf{P}_{\! j}+\left(2-n\right)w_i\tilde{\alpha}_{ij}\mathsf{P}_{\! j} \right] \\
				&\eqqcolon w^{1-n}\left[A_1+A_2+\left(2-n\right)w_i\tilde{\alpha}_{ij}\mathsf{P}_{\! j} \right].
			\end{split}
		\end{equation}
		We now proceed with the computation of~$A_1$ and~$A_2$. From~\eqref{eq:Pij}, and using the symmetry property~$\tilde{\alpha}_{ij}=\tilde{\alpha}_{ji}$, we obtain
		\begin{equation}
		\label{eq:A1-1}
			\begin{split}
				A_1=&  -2\tilde{\alpha}_{ij}\mathsf{P}_{\! i} w_j-\mathsf{P}\tilde{\alpha}_{ij} w_{ij}+n\left(p-1\right)\tilde{\alpha}_{ij}\partial_{ij}\mathscr{V}(\nabla w)
				\\
				&=-2\tilde{\alpha}_{ij} \mathsf{P}_{\! i} w_j-\mathsf{P} \tr\mathscr{W}+n\left(p-1\right)\tilde{\alpha}_{ij}\partial_{ij}\mathscr{V}(\nabla w),
			\end{split}
		\end{equation}
		where in the last equality we used~\eqref{eq:deriv-W}. For the last summand in~\eqref{eq:A1-1}, exploiting~\eqref{eq:deriv-V}, we get
		\begin{equation*}
			\begin{split}
				\tilde{\alpha}_{ij}\partial_{ij}\mathscr{V}(\nabla w)&=\tilde{\alpha}_{ij}\partial_i \!\left(\partial_j \mathscr{V}(\nabla w)\right)=\tilde{\alpha}_{ij}\partial_i \!\left(\mathsf{a}_k \!\left(\nabla w\right)w_{kj}\right)=
				\\
				&=\tilde{\alpha}_{ij}\tilde{\alpha}_{kl}w_{kj}w_{li}+\tilde{\alpha}_{ij} \mathsf{a}_k \!\left(\nabla w\right) w_{kji}=\tr\mathscr{W}^2+\tilde{\alpha}_{ij} \mathsf{a}_k \!\left(\nabla w\right) w_{kji}=
				\\
				&=\tr\mathscr{W}^2+\partial_k \!\left(\tilde{\alpha}_{ij} w_{ij}\right)\mathsf{a}_k(\nabla w)-\left(\partial_k \tilde{\alpha}_{ij}\right) \mathsf{a}_k \!\left(\nabla w\right) w_{ij}.
			\end{split}
		\end{equation*}
		Since 
		\begin{equation}
		\label{eq:taijk}
			\partial_k \tilde{\alpha}_{ij} = \partial_{\xi_i\xi_j\xi_l}\mathscr{V} \!\left(\nabla w\right)w_{lk},
		\end{equation}
		from the previous equation, we obtain
		\begin{equation*}
			\begin{split}
				\tilde{\alpha}_{ij}\partial_{ij}\mathscr{V}(\nabla w) & =\tr\mathscr{W}^2+\partial_k \!\left(\tr\mathscr{W}\right)\mathsf{a}_k(\nabla w)-\partial_{\xi_i\xi_j\xi_l}\mathscr{V}\!\left(\nabla w\right) w_{lk} \mathsf{a}_k\!\left(\nabla w\right) w_{ij} \\
				& =\tr\mathscr{W}^2+\partial_k\!\left(\tr\mathscr{W}\right)\mathsf{a}_k(\nabla w)-\partial_{\xi_i\xi_j\xi_l}\mathscr{V}\!\left(\nabla w\right)\partial_l \mathscr{V}\!\left(\nabla w\right) w_{ij},
			\end{split}
		\end{equation*}
		where in the last equality we used~\eqref{eq:deriv-V}. Substituting this into~\eqref{eq:A1-1}, we get 
		\begin{multline}
		\label{eq:A1-2}
			A_1= -2\tilde{\alpha}_{ij} \mathsf{P}_{\! i} w_j-\mathsf{P} \tr\mathscr{W} \\
			+n\left(p-1\right) \left\{\tr\mathscr{W}^2+\partial_k\!\left(\tr\mathscr{W}\right)\mathsf{a}_k(\nabla w)-\partial_{\xi_i\xi_j\xi_l}\mathscr{V}\!\left(\nabla w\right)\partial_l \mathscr{V}\!\left(\nabla w\right) w_{ij}\right\}.
		\end{multline}
		
		Now, from~\eqref{eq:Pi} and~\eqref{eq:taijk}, we deduce that
		\begin{equation}
		\label{eq:A2}
			A_2=\left(-\mathsf{P}w_j+n\left(p-1\right)\partial_j \mathscr{V}\!\left(\nabla w\right) \right)\partial_{\xi_i\xi_j\xi_k} \mathscr{V}\!\left(\nabla w\right) w_{ki},
		\end{equation} 
		therefore, inserting~\eqref{eq:A1-2} and~\eqref{eq:A2} into~\eqref{eq:diver-to-comp}, we establish the validity of~\eqref{eq:id-diff}.
	\end{proof}

	
	\subsection{The approximation procedure.}
	\label{step:approx}
	
	In this step, we define an approximation for~\eqref{eq:eqforv} to ensure that solutions enjoy the necessary regularity to apply Proposition~\ref{prop:id-diff} of~\ref{step:differ-id}. We aim to derive from this an integral identity valid for solutions of the approximate problem -- namely~\eqref{eq:fund-ident} below --, and ultimately take the limit to establish an integral inequality for solutions of the original problem~\eqref{eq:eqforv}.
	
	We begin by defining the approximation. First, let~$r>0$ be fixed. Given~$\varepsilon \in (0,1)$ and~$\xi \in \R^n$, we define
	\begin{equation*}
		V^\varepsilon(\xi) \coloneqq \frac{1}{p} \left[\varepsilon^2+\abs*{\xi}^2\right]^{\frac{p}{2}} \quad\mbox{and}\quad a^\varepsilon(\xi) \coloneqq \nabla V^\varepsilon(\xi).
	\end{equation*}
	
	Let~$v^{\varepsilon} = v^{\varepsilon,r} \in W^{1,p}(B_r)$ be a weak solution to
	\begin{equation}
	\label{eq:approx}
		\begin{cases}
			\diver\left(a^\varepsilon \!\left(\nabla v^\varepsilon\right) \right) = P + R                    & \quad \mbox{in } B_r, \\
			v^\varepsilon = v        & \quad \mbox{on } \partial B_r. \\
		\end{cases}
	\end{equation}
	The direct method of the calculus of variations -- see, for instance,~\cite[Theorem~3.30]{dac} -- applied to the functional
	\begin{equation*}
		\mathcal{F}_\varepsilon(v) \coloneqq \int_{B_r} V^\varepsilon(\nabla v^\varepsilon) + \left(P+R\right) v^\varepsilon \, dx
	\end{equation*}
	ensures the existence of a unique minimizer~$v^{\varepsilon} \in W^{1,p}(B_r)$, with~$v^{\varepsilon} - v \in W^{1,p}_0(B_r)$, satisfying~\eqref{eq:approx} in weak sense.
	
	By the regularity of~$v$,~$P$, and~$\kappa$, classical elliptic estimates -- see, for instance,~\cite{lad-ur} -- give that  
	\begin{gather}
	\label{eq:reg-vep}
		v^\varepsilon \in W^{2,2}_{\loc}(B_r) \cap C^{1,\alpha}_{\loc}(B_r), \\
	\label{eq:conv-vep}
		v^\varepsilon \to v \quad \mbox{in } W^{1,p}(B_r) \cap C^{1,\alpha}_{\loc}(B_r),
	\end{gather}
	see also~\cite{carlos} for~\eqref{eq:conv-vep}. Moreover, for any~$B_{\rho} \Subset B_r$, there exists a constant~$C>0$, independent of~$\varepsilon$, such that 
	\begin{gather}
	\label{eq:ub1}
		\int_{B_{\rho}} \left[\varepsilon^2+\abs*{\nabla v^\varepsilon}^2\right]^{ p-2} \norma{D^2 v^\varepsilon}^2 \, dx \leq C, \\
	\label{eq:ub3}
		\norma*{v^\varepsilon}_{C^{1,\alpha}(B_{\rho})} \leq C \quad \mbox{and} \quad \norma*{a^\varepsilon \!\left(\nabla v^\varepsilon\right)}_{W^{1,2}(B_{\rho})} \leq C.
	\end{gather}
	Finally, by setting~$\mathcal{Z}_{v^\varepsilon} \coloneqq \left\{x \in B_r \,\lvert\, \nabla v^\varepsilon(x)=0 \right\}$, we have
	\begin{equation}
	\label{eq:vep-C3}
		\abs*{\mathcal{Z}_{v^\varepsilon}} = 0 \quad\mbox{and}\quad v^\varepsilon \in C^3 \!\left(B_r \setminus \left(\mathcal{Z}_{v^\varepsilon} \cup \mathcal{Z}_{v}\right)\right).
	\end{equation}

	Taking advantage of the regularity of~$v$ and~$v^\varepsilon$, by defining
	\begin{equation*}
	\label{eq:def-Wep}
		W^\varepsilon \coloneqq \nabla a^\varepsilon \!\left(\nabla v^\varepsilon\right),
	\end{equation*}
	we have that~$W^\varepsilon \in L^2_{\loc}(B_r)$ and we can write
	\begin{equation}
	\label{eq:defWep}
		W^\varepsilon(x)=
		\begin{dcases}
			A^\varepsilon(x) D^2 v^\varepsilon(x)	& \quad \mbox{if } x \in B_r \setminus \left(\mathcal{Z}_{v^\varepsilon} \cup \mathcal{Z}_{v}\right), \\
			0				& \quad \mbox{if } x \in \mathcal{Z}_{v^\varepsilon} \cup \mathcal{Z}_{v}.
		\end{dcases}
	\end{equation}
	Here,~$A^\varepsilon$ is the matrix with components
	\begin{equation}
	\label{eq:defA-ep}
		\left[A^\varepsilon(x)\right]_{ij} = \alpha^\varepsilon_{ij}(x) \coloneqq \partial_{\xi_i\xi_j} V^\varepsilon \!\left(\nabla v^\varepsilon\right)\!(x) \quad\mbox{for } x \in \R^n \setminus \mathcal{Z}_{v^\varepsilon},
	\end{equation}
	that is
	\begin{equation}
	\label{eq:defA-ep-esplicit}
		A^\varepsilon = \left[\varepsilon^2+\abs*{\nabla v^\varepsilon}^2\right]^{ \frac{p-2}{2}} \Id + \left(p-2\right) \left[\varepsilon^2+\abs*{\nabla v^\varepsilon}^2\right]^{ \frac{p-4}{2}} \nabla v^\varepsilon \otimes \nabla v^\varepsilon \quad\mbox{on } B_r \setminus \mathcal{Z}_{v^\varepsilon}.
	\end{equation}
	As for the matrix~$A$ defined in~\eqref{eq:def-A}, we extend~$A^\varepsilon$ to zero on~$\mathcal{Z}_{v^\varepsilon}$ and observe that it is of class~$C^{1,\alpha}_{\loc}$ in~$B_r \setminus \left(\mathcal{Z}_{v^\varepsilon} \cup \mathcal{Z}_{v}\right)$. Hence,~$A^\varepsilon$ is measurable on~$B_r$. In addition, we also have
	\begin{equation}
	\label{eq:trWep}
		\tr W^\varepsilon = \alpha^\varepsilon _{ij} v^\varepsilon_{ij} = P + R \quad \mbox{a.e.\ in } B_r.
	\end{equation}

	We are now in position to state and prove the fundamental integral identity related to our approximation~\eqref{eq:approx}. This is exactly the content of the following result.
	
	\begin{proposition}
	\label{prop:fund-ident}
		Let~$v^\varepsilon=v^{\varepsilon,r}$ be the unique weak solution of~\eqref{eq:approx} and set
		\begin{equation}
		\label{eq:P-ep}
			P^\varepsilon \coloneqq \frac{n \left(p-1\right)}{v^\varepsilon} \, V^\varepsilon\!\left(\nabla v^\varepsilon\right) + \left(\frac{p}{n-p}\right)^{\! p-1} \frac{1}{v^\varepsilon}.
		\end{equation}
		Then, for any~$\phi\in C^\infty_{c}(B_r)$ and any~$t \in \R$, the following identity holds
		\begin{align}
			\notag
				-\int_{B_r} &(v^\varepsilon)^{2-n} \left\langle A^\varepsilon\nabla P^\varepsilon, \nabla (P^t \phi) \right\rangle dx = \int_{B_r}\diver\!\left((v^\varepsilon)^{2-n}A^\varepsilon\nabla P^\varepsilon \right) P^t \phi \, dx \\
			\label{eq:fund-ident}
				&=\int_{B_r}(v^\varepsilon)^{1-n} \,\Bigl\{-n \left\langle A^\varepsilon\nabla P^\varepsilon, \nabla v^\varepsilon \right\rangle -P^\varepsilon \tr W^\varepsilon \\
			\notag
				&\hspace{0.29cm}+ n \left(p-1\right) \Bigl[\tr[W^\varepsilon]^2+ \left\langle \nabla\!\left(\tr W^\varepsilon\right),  a^\varepsilon(\nabla v^\varepsilon) \right\rangle \Bigr]
				-P^\varepsilon v^\varepsilon_j\,\partial_{\xi_i\xi_j\xi_k}V^\varepsilon(\nabla v^\varepsilon)v^\varepsilon_{ki}\Bigr\}\, P^t \phi \, dx,
		\end{align}
		where~$P$ is the~$P$-function defined in~\eqref{eq:defPfunct} and~$A^\varepsilon$ is the matrix given by~\eqref{eq:defA-ep}.
	\end{proposition}
	\begin{proof}
		From~\eqref{eq:reg-vep}, we know that~$v^\varepsilon \in W^{2,2}_{\loc}(B_r)$ is a strong solution to~\eqref{eq:trWep}. Moreover, from the explicit expression of~$A^\varepsilon$ given by~\eqref{eq:defA-ep-esplicit}, we deduce that~$\alpha^\varepsilon_{ij} \in C^{0,\alpha}_{\loc}(B_r)$. Since~$P \in C^{0,\alpha}(B_r)$, by~\eqref{eq:regP}, and~$R \in C^{0,\alpha}(B_r)$, by~\eqref{eq:regv} and the regularity of~$\kappa$, Theorem~9.19 in~\cite{gt} yields
		\begin{equation*}
			v^\varepsilon \in C^{2,\alpha}_{\loc}(B_r),
		\end{equation*}
		which, in turn, implies~$\alpha^\varepsilon_{ij} \in C^{1,\alpha}_{\loc}(B_r)$ and~$P^\varepsilon \in C^{1,\alpha}_{\loc}(B_r)$.
		
		Again, by Theorem~9.19 in~\cite{gt}, we have that~$v^\varepsilon \in W^{3,2}_{\loc}(B_r)$, which implies that~$V^\varepsilon \!\left(\nabla v^\varepsilon\right) \in W^{2,2}_{\loc}(B_r)$, since, by the chain rule
		\begin{equation*}
			\nabla V^\varepsilon \!\left(\nabla v^\varepsilon\right) = a^\varepsilon \!\left(\nabla v^\varepsilon\right) D^2 v^\varepsilon.
		\end{equation*}
		As a consequence, from~\eqref{eq:P-ep}, we infer
		\begin{equation*}
			P^\varepsilon \in W^{2,2}_{\loc}(B_r),
		\end{equation*}
		and thus~$\left(A^\varepsilon \nabla P^\varepsilon\right)_i = \alpha^\varepsilon_{ij} P^\varepsilon_j \in W^{1,2}_{\loc}(B_r)$.
		
		Therefore, for any test function~$\phi\in C^\infty_{c}(B_r)$, integrating by parts, we deduce
		\begin{equation*}
			- \int_{B_r} (v^\varepsilon)^{2-n} \left\langle A^\varepsilon\nabla P^\varepsilon, \nabla (P^t \phi) \right\rangle dx = \int_{B_r}\diver\!\left((v^\varepsilon)^{2-n}A^\varepsilon\nabla P^\varepsilon \right) P^t \phi \, dx.
		\end{equation*}
	
		Finally, taking advantage of~\eqref{eq:vep-C3}, we apply Proposition~\ref{prop:id-diff} with~$w=v^\varepsilon$,~$\mathscr{V}=V^\varepsilon$, and~$\Omega = B_r \setminus \left(\mathcal{Z}_{v^\varepsilon} \cup \mathcal{Z}_{v}\right)$ to conclude that~\eqref{eq:id-diff} holds in~$\Omega$. Since both~$\mathcal{Z}_{v^\varepsilon}$ and~$\mathcal{Z}_{v}$ are negligible sets in~$B_r$, we deduce that~\eqref{eq:id-diff} holds a.e.\ in~$B_r$, and the conclusion follows.
	\end{proof}

	We now extend Proposition~\ref{prop:fund-ident} to a solution of~\eqref{eq:eqforv} in order to derive the fundamental inequality required for our proof. This extension relies on a subtle argument to obtain a meaningful inequality in its full strength. The key difficulty is that the tensor~$W$ is not necessarily symmetric, except when~$p=2$.
	
	\begin{proposition}
	\label{prop:fund-ineq}
		Let~$v$ be a weak solution of~\eqref{eq:eqforv}. Then, we have
		\begin{equation*}
		\label{eq:toprove-AgradP}
			A \nabla P \in L^2_{\loc}(\R^n).
		\end{equation*}
		In addition, for every non-negative~$\phi \in C^\infty_c(\R^n)$ such that~$\mathrm{supp}\,\phi \subseteq B_{r}$ and any~$t \in \R$, it follows that
		\begin{equation}
			\label{eq:fund-ineq}
			\begin{split}
				-\int_{B_{r}} &v^{2-n} \left\langle A\nabla P, \nabla (P^t \phi) \right\rangle dx \geq n \left(p-1\right) \left(1-c_p\right) \int_{B_{r}} v^{1-n} \,\abs{\mathring{W}}^2 \, P^t \phi \, dx \\
				&+ \left(p-1\right) \int_{B_{r}} v^{1-n} \,\bigl\{n \left\langle \nabla R,  \stressv \right\rangle + R \left(P+R\right) \bigr\}\, P^t \phi \, dx,
			\end{split}
		\end{equation}
		where~$c_p \in [0,1)$ is defined as
		\begin{equation*}
			c_p \coloneqq \frac{\left(1-\varrho_p\right)^2}{1+\varrho_p^2} \quad\mbox{with}\quad \varrho_{p} \coloneqq \left(p-1\right)^{\mathrm{sgn}(2-p)}\!.
		\end{equation*}
	\end{proposition}

	We first recall an elementary yet fundamental observation from Lemma~3.1 in~\cite{guarn-mosconi}, due to Guarnotta \& Mosconi, whose proof can be found therein.
	
	\begin{lemma}
	\label{lem:gu-mosc}
		Let~$X=PS$, where~$P,S \in \mathrm{Sym}(n)$, with~$P$ being positive definite. Moreover, let~$\lambda_{\mathrm{min}}$ and~$\lambda_{\mathrm{max}}$ denote the minimum and maximum eigenvalues of~$P$, respectively, and define the ratio $\varrho \coloneqq \lambda_{\mathrm{min}} / \lambda_{\mathrm{max}} \in (0,1]$. Then, we have
		\begin{equation*}
			\abs*{X-X^\mathsf{T}}^{2} \leq 2c \,\abs{X}^2
		\end{equation*}
		with
		\begin{equation}
		\label{eq:def-c-gm}
			c \coloneqq \frac{\left(1-\varrho\right)^2}{1+\varrho^2} \in [0,1).
		\end{equation}
		Additionally, the map~$\varrho \mapsto c = c(\varrho)$ is decreasing on~$(0,1]$.
	\end{lemma}

	Building on this, we establish the following algebraic lemma. Although this result is also elementary, we provide a detailed proof due to its significance.
	\begin{lemma}
		\label{lem:matrix-est}
		Let~$X=PS$, where~$P,S \in \mathrm{Sym}(n)$, with~$P$ being positive definite. Then, the following inequality holds
		\begin{equation*}
			\tr\! X^2 - \abs*{X}^2 \geq -c \,\abs{\mathring{X}}^2,
		\end{equation*}
		where~$c \in [0,1)$ is the constant defined in~\eqref{eq:def-c-gm}.
	\end{lemma}
	\begin{proof}
		We begin by computing
		\begin{align}
			\notag
			\tr\! X^2 - \abs*{X}^2 &= \tr\! X^2 - \tr\!\left[X X^\mathsf{T}\right] = \tr \!\left[X \left(X-X^\mathsf{T}\right)\right] = 2 \tr\!\left[X X_\mathrm{A}\right] \\
			\label{eq:diff-tr-norm}
			&= 2 \tr\!\left[\left(X_\mathrm{S}+X_\mathrm{A}\right) X_\mathrm{A}\right] = 2 \tr\!\left[X_\mathrm{S} X_\mathrm{A} + X_\mathrm{A}^2\right]
		\end{align}
		where~$X_\mathrm{S}$ and~$X_\mathrm{A}$ denote the symmetric and the antisymmetric part of~$X$, respectively. Moreover, we have 		\begin{equation*}
		\label{eq:XS-XA-tenere}
			4 X_\mathrm{S} X_\mathrm{A} = \left(X+X^\mathsf{T}\right) \left(X-X^\mathsf{T}\right) = X^2 - X X^\mathsf{T} + X^\mathsf{T} X - \left(X^\mathsf{T}\right)^{\!2},
		\end{equation*}
		whence, by the cyclic property of the trace, we have  
		\begin{equation}
		\label{eq:tracce}
			\tr\!\left[X_\mathrm{S} X_\mathrm{A}\right]=0.
		\end{equation}
		Therefore, using~\eqref{eq:diff-tr-norm},~\eqref{eq:tracce}, and the antisymmetry of~$X_\mathrm{A}$, we deduce that
		\begin{equation*}
			\tr\! X^2 - \abs*{X}^2 = - 2 \tr\!\left[X_\mathrm{A} X_\mathrm{A}^\mathsf{T}\right] = -2 \,\abs*{X_\mathrm{A}}^2
		\end{equation*}
		We observe that~$\mathring{X}_\mathrm{A}=X_\mathrm{A}$, thus the latter implies
		\begin{equation}
			\label{eq:tr-norm-senzatr}
			\tr\! X^2 - \abs*{X}^2 = -2 \,\abs*{\mathring{X}_\mathrm{A}}^2.
		\end{equation}
		Since~$P \in \mathrm{Sym}(n)$ is positive definite, it admits an inverse~$P^{-1} \in \mathrm{Sym}(n)$, allowing us to write
		\begin{equation*}
			\mathring{X} = PS - \frac{\tr X}{n} \Id = P \left(S - \frac{\tr X}{n} P^{-1}\right) \!,
		\end{equation*}
		which is the product of two symmetric matrices with the first one being positive definite. By applying Lemma~\ref{lem:gu-mosc}, we obtain
		\begin{equation*}
			\abs*{\mathring{X}-\mathring{X}^\mathsf{T}}^{2} \leq 2c \,\abs{\mathring{X}}^2,
		\end{equation*}
		that is
		\begin{equation*}
			2 \,\abs*{\mathring{X}_\mathrm{A}}^{2} \leq c \,\abs{\mathring{X}}^2
		\end{equation*}
		Combining this with~\eqref{eq:tr-norm-senzatr}, the conclusion easily follows.
	\end{proof}

	We are now in position to proceed with the proof of Proposition~\ref{prop:fund-ineq}.

	\begin{proof}[Proof of Proposition~\ref{prop:fund-ineq}]
		Let~$r>0$ be fixed and~$\phi \in C^\infty_c(\R^n)$ be non-negative and such that~$\mathrm{supp}\,\phi \subseteq B_{r}$.
		
		First, by~\eqref{eq:trWep}, we observe that
		\begin{equation*}
			\nabla\!\left(\tr W^\varepsilon\right) = \nabla P + \nabla R \quad\mbox{a.e.\ in } B_r.
		\end{equation*}
		Therefore, from this,~\eqref{eq:fund-ident}, and taking into account~\eqref{eq:trWep} once more, we deduce
		\begin{align}
		\notag
			-\int_{B_{r}} &(v^\varepsilon)^{2-n} \left\langle A^\varepsilon\nabla P^\varepsilon, \nabla (P^t \phi) \right\rangle dx \\
		\label{eq:passare-al-limite}
			&=\int_{B_{r}}(v^\varepsilon)^{1-n} \,\bigl\{-n \left\langle A^\varepsilon\nabla P^\varepsilon, \nabla v^\varepsilon \right\rangle -P^\varepsilon P -P^\varepsilon R + n \left(p-1\right) \bigl[\tr[W^\varepsilon]^2 \\
		\notag
			&\hspace{0,45cm} + \left\langle \nabla P,  a^\varepsilon(\nabla v^\varepsilon) \right\rangle + \left\langle \nabla R,  a^\varepsilon(\nabla v^\varepsilon) \right\rangle \bigr] 
			-P^\varepsilon v^\varepsilon_j\,\partial_{\xi_i\xi_j\xi_k}V^\varepsilon(\nabla v^\varepsilon)v^\varepsilon_{ki}\bigr\}\, P^t \phi \, dx.
		\end{align}
	
		We now shall let~$\varepsilon \to 0$ in~\eqref{eq:passare-al-limite}. To this aim we study the convergence of the terms on both sides of~\eqref{eq:passare-al-limite}.
		
		We first prove that
		\begin{gather}
		\label{eq:conv-V}
			V^\varepsilon \!\left(\nabla v^\varepsilon\right) \to V \!\left(\nabla v\right) \mbox{ in } C^{0,\alpha}_{\loc}(B_{r}) \mbox{ and } V^\varepsilon \!\left(\nabla v^\varepsilon\right) \rightharpoonup V \!\left(\nabla v\right) \mbox{ weakly in } W^{1,2}_{\loc}(B_{r}), \\
		\label{eq:conv-stress}
			a^\varepsilon \!\left(\nabla v^\varepsilon\right) \to a \!\left(\nabla v\right) \mbox{ in } C^{0}_{\loc}(B_{r}) \mbox{ and } a^\varepsilon \!\left(\nabla v^\varepsilon\right) \rightharpoonup a \!\left(\nabla v\right) \mbox{ weakly in } W^{1,2}_{\loc}(B_{r}), \\
		\label{eq:conv-P}
		P^\varepsilon \to  P \mbox{ in } C^{0,\alpha}_{\loc}(B_{r}) \mbox{ and } P^\varepsilon \rightharpoonup P \mbox{ weakly in } W^{1,2}_{\loc}(B_{r}).
		\end{gather}
		The H\"older and uniform convergences in~\eqref{eq:conv-V} and~\eqref{eq:conv-stress} follow from~\eqref{eq:conv-vep} and since~$V^\varepsilon \in C^1(\R^n)$.
		
		Next, taking advantage of the uniform bound~\eqref{eq:ub3}, we can extract a subsequence, relabeled as~$a^\varepsilon \!\left(\nabla v^\varepsilon\right)$, such that
		\begin{equation*}
			 a^\varepsilon \!\left(\nabla v^\varepsilon\right) \rightharpoonup \stressv \quad \mbox{weakly in } W^{1,2}_{\loc}(B_{r}).
		\end{equation*}
		Then, we have
		\begin{equation*}
			a^\varepsilon_k \!\left(\nabla v^\varepsilon\right) = \partial_{\xi_k} V^\varepsilon \!\left(\nabla v^\varepsilon\right) = \left[\varepsilon^2+\abs*{\nabla v^\varepsilon}^2\right]^{\frac{p-2}{2}} v^\varepsilon_k,
		\end{equation*}
		where~$a^\varepsilon_k(\xi)$ denotes the~$k$-th component of~$a_k(\xi)$. By the chain rule
		\begin{equation}
		\label{eq:djVep}
			\partial_j V^\varepsilon \!\left(\nabla v^\varepsilon\right) = a^\varepsilon_k \!\left(\nabla v^\varepsilon\right) v^\varepsilon_{kj} \quad \mbox{a.e.\ in } B_{r},
		\end{equation}
		so that, using~\eqref{eq:ub1} and~\eqref{eq:ub3}, we deduce
		\begin{equation*}
			\norma*{\nabla V^\varepsilon \!\left(\nabla v^\varepsilon\right)}_{L^2(B_\rho)} \leq C \quad \mbox{for any } \rho \in (0,r),
		\end{equation*}
		and for some~$C>0$ depending on~$\rho$. Thus, the weak convergence in~\eqref{eq:conv-V} follows, up to a subsequence. Since the argument holds for any subsequence, it follows that the weak
		convergences hold for the whole sequence.
	
		Finally, the convergences in~\eqref{eq:conv-P} are a consequence of the definition of~$P^\varepsilon$ in~\eqref{eq:P-ep}, \eqref{eq:conv-vep}, and~\eqref{eq:conv-V}.
		
		We now deal with the left-hand side of~\eqref{eq:passare-al-limite} and show that
		\begin{equation}
		\label{eq:conv-AgradP}
			A^\varepsilon\nabla P^\varepsilon \rightharpoonup A\nabla P \quad \mbox{weakly in } L^{2}_{\loc}(B_{r}).
		\end{equation}
		Recall that the matrices~$A$ and~$A^\varepsilon$, given in~\eqref{eq:def-A} and~\eqref{eq:defA-ep-esplicit} respectively, are well-defined and of class~$C^1$ in the set
		\begin{equation*}
			\mathcal{B}_r^\varepsilon \coloneqq B_{r} \setminus \left(\mathcal{Z}_{v^\varepsilon} \cup \mathcal{Z}_{v}\right),
		\end{equation*}
		which has full measure in~$B_{r}$. By~\eqref{eq:Pi}, we have
		\begin{equation}
		\label{eq:weak-conv-AgradP}
			\left(A^\varepsilon\nabla P^\varepsilon\right)_i = \alpha_{ij}^\varepsilon P_{j}^\varepsilon = - \frac{P^\varepsilon}{v^\varepsilon}\, \alpha^\varepsilon_{ij} v^\varepsilon_j + \frac{n \left(p-1\right)}{v^\varepsilon}\, \alpha^\varepsilon_{ij} \partial_j V^\varepsilon\!\left(\nabla v^\varepsilon\right) \quad \mbox{in } \mathcal{B}_r^\varepsilon,
		\end{equation}
		where, thanks to~\eqref{eq:defA-ep-esplicit}, we have
		\begin{equation*}
			\alpha^\varepsilon_{ij} = \left[\varepsilon^2+\abs*{\nabla v^\varepsilon}^2\right]^{ \frac{p-2}{2}} \delta_{ij} + \left(p-2\right) \left[\varepsilon^2+\abs*{\nabla v^\varepsilon}^2\right]^{ \frac{p-4}{2}} v^\varepsilon_i v^\varepsilon_j \quad \mbox{in } \mathcal{B}_r^\varepsilon.
		\end{equation*}
		Let~$\rho \in (0,r)$ be fixed. The latter, together with~\eqref{eq:ub3}, implies
		\begin{equation}
		\label{eq:bounded}
			\alpha^\varepsilon_{ij} v^\varepsilon_j \leq C \quad \mbox{in } \mathcal{B}_\rho^\varepsilon,
		\end{equation}
		for some~$C>0$ independent of~$\varepsilon$.
		
		We now study the pointwise convergence of~$A^\varepsilon\nabla P^\varepsilon$. To this end, we extract a subsequence~$\left\{\varepsilon_m\right\}_{m \in \N}$ and define
		\begin{equation*}
			\mathcal{B}_r \coloneqq \bigcap_{m \in \N} \mathcal{B}_r^{\varepsilon_m},
		\end{equation*}
		which is a Borel set of full measure in~$B_{r}$.
		
		Since~$\nabla_\xi^2 V^\varepsilon \to \nabla_\xi^2 V$ in~$\R^n \setminus \left\{0\right\}$, form~\eqref{eq:conv-vep}, we deduce
		\begin{equation*}
			\alpha^{\varepsilon_m}_{ij} v^{\varepsilon_m}_j \to \alpha_{ij} v_j \quad \mbox{in } \mathcal{B}_r
		\end{equation*}
		as~$m \to +\infty$. Taking advantage of~\eqref{eq:bounded}, the dominated convergence theorem entails
		\begin{equation}
		\label{eq:convL2-alpv}
			\alpha^{\varepsilon_m}_{ij} v^{\varepsilon_m}_j \to \alpha_{ij} v_j \quad \mbox{in } L^2(B_\rho).
		\end{equation}
		By using~\eqref{eq:djVep} and~\eqref{eq:defWep}, we obtain
		\begin{equation}
		\label{eq:id-alpV}
			\alpha^\varepsilon_{ij}\partial_j V^\varepsilon \!\left(\nabla v^\varepsilon\right) = \alpha^\varepsilon_{ij} a^\varepsilon_k \!\left(\nabla v^\varepsilon\right) v^\varepsilon_{kj} = a^\varepsilon_k \!\left(\nabla v^\varepsilon\right) \left(W^\varepsilon\right)_{ik} \quad \mbox{ in } \mathcal{B}_r,
		\end{equation}
		and, thanks to~\eqref{eq:conv-stress}, that
		\begin{equation}
		\label{eq:weakconv-aW}
			a^\varepsilon_k \!\left(\nabla v^\varepsilon\right) \left(W^\varepsilon\right)_{ik} \rightharpoonup a_k \!\left(\nabla v\right) W_{ik} \quad \mbox{weakly in } L^{2}(B_\rho).
		\end{equation}
		Furthermore, from~\eqref{eq:defW} and the chain rule~$\partial_j V \!\left(\nabla v\right) = a_k \!\left(\nabla v\right) v_{kj}$, we also get
		\begin{equation*}
			a_k \!\left(\nabla v\right) W_{ik} = a_k \!\left(\nabla v\right) \alpha_{ij} v_{jk} = \alpha_{ij} \partial_j V \!\left(\nabla v\right) \quad \mbox{in } \R^n \setminus \mathcal{Z}_v,
		\end{equation*}
		where we also used~\eqref{eq:regv}. The latter, together with~\eqref{eq:id-alpV} and~\eqref{eq:weakconv-aW}, yields
		\begin{equation}
		\label{eq:weak-conv-alpV}
			\alpha^\varepsilon_{ij}\partial_j V^\varepsilon \rightharpoonup \alpha_{ij} \partial_j V \!\left(\nabla v\right) \quad \mbox{weakly in } L^{2}(B_\rho).
		\end{equation}
	
		Piecing~\eqref{eq:conv-vep},~\eqref{eq:conv-P},~\eqref{eq:convL2-alpv}, and~\eqref{eq:weak-conv-alpV} with~\eqref{eq:weak-conv-AgradP}, we are led to
		\begin{equation*}
			\alpha_{ij}^{\varepsilon_m} P_{j}^{\varepsilon_m} \rightharpoonup - \frac{P}{v}\, \alpha_{ij} v_j + \frac{n \left(p-1\right)}{v}\, \alpha_{ij} \partial_j V \!\left(\nabla v\right) = \alpha_{ij} P_j \quad \mbox{weakly in } L^{2}(B_\rho),
		\end{equation*}
		where the last identity follows from~\eqref{eq:Pi}. As a consequence,~$A \nabla P \in L^2(B_\rho)$ and~\eqref{eq:conv-AgradP} follows since the previous argument hold for every subsequence~$\left\{\varepsilon_m\right\}_{m \in \N}$.
		
		We now handle the right-hand side of~\eqref{eq:passare-al-limite}. First, by exploiting~\eqref{eq:conv-vep},~\eqref{eq:conv-stress},~\eqref{eq:conv-P}, and~\eqref{eq:conv-AgradP}, we infer
		\begin{equation}
		\label{eq:conv-1}
			-n \left\langle A^\varepsilon\nabla P^\varepsilon, \nabla v^\varepsilon \right\rangle \rightharpoonup -n \left\langle A\nabla P, \nabla v \right\rangle \quad \mbox{weakly in } L^{2}_{\loc}(B_{r}),
		\end{equation}
		and
		\begin{multline}
		\label{eq:conv-2}
			-P^\varepsilon P -P^\varepsilon R + n \left(p-1\right) \bigl[ \left\langle \nabla P,  a^\varepsilon(\nabla v^\varepsilon) \right\rangle + \left\langle \nabla R,  a^\varepsilon(\nabla v^\varepsilon) \right\rangle \bigr] \to \\
			-P^2 -P R + n \left(p-1\right) \bigl[ \left\langle \nabla P,  \stressv \right\rangle + \left\langle \nabla R, \stressv \right\rangle \bigr] \quad \mbox{in } C^{0}_{\loc}(B_{r}).
		\end{multline}
		We shall now show that
		\begin{equation}
		\label{eq:weakconv-toprove}
			v^\varepsilon_j \,\partial_{\xi_i\xi_j\xi_k}V^\varepsilon(\nabla v^\varepsilon)v^\varepsilon_{ki} \rightharpoonup \left(p-2\right) \left(P+R\right) \quad \mbox{weakly in } L^{2}_{\loc}(B_{r}).
		\end{equation}
		We start by proving that
		\begin{align}
		\label{eq:reltrprove}
			v^\varepsilon_j \,\partial_{\xi_i\xi_j\xi_k}V^\varepsilon(\nabla v^\varepsilon)v^\varepsilon_{ki} &= \left(p-2\right) \left[1-\frac{\varepsilon^2}{\varepsilon^2 + \abs*{\nabla v^\varepsilon}^2}\right] \left(P+R\right) \\
		\notag
			&\quad+ 2 \left(p-2\right) \frac{\varepsilon^2 \, v^\varepsilon_i \, v^\varepsilon_k}{\left[\varepsilon^2 + \abs*{\nabla v^\varepsilon}^2\right]^2} \left[\varepsilon^2 + \abs*{\nabla v^\varepsilon}^2\right]^{\frac{p-2}{2}} v^\varepsilon_{ik} \quad \mbox{in } B_{r} \setminus \mathcal{Z}_{v^\varepsilon}.
		\end{align}
		In~$\R^n \setminus \left\{0\right\}$, a direct computation reveals
		\begin{gather}
		\notag
			a_i^\varepsilon(\xi) = \partial_{i}V^\varepsilon(\xi) = \left[\varepsilon^2 + \abs*{\xi}^2\right]^{\frac{p-2}{2}} \xi_i, \\
		\label{eq:Vij}
			\partial_{ij}V^\varepsilon(\xi) = \left[\varepsilon^2 + \abs*{\xi}^2\right]^{\frac{p-2}{2}} \delta_{ij} + \left(p-2\right) \left[\varepsilon^2 + \abs*{\xi}^2\right]^{\frac{p-4}{2}} \xi_i \,\xi_j,
		\end{gather}
		and
		\begin{equation*}
			\begin{split}
				\partial_{ijk}V^\varepsilon(\xi) &= \left(p-2\right) \left[\varepsilon^2 + \abs*{\xi}^2\right]^{\frac{p-4}{2}} \left(\xi_k \delta_{ij} + \xi_i \delta_{jk} +\xi_j \delta_{ik}\right) \\
				&\quad+ \left(p-2\right) \left(p-4\right) \left[\varepsilon^2 + \abs*{\xi}^2\right]^{\frac{p-6}{2}} \xi_i \xi_j \xi_k.
			\end{split}
		\end{equation*}
		The latter implies
		\begin{equation}
		\label{eq:torewrite}
			\begin{split}
				\partial_{ijk}V^\varepsilon(\xi) \xi_j &= \left(p-2\right) \left[\varepsilon^2 + \abs*{\xi}^2\right]^{\frac{p-4}{2}} \left(2 \xi_i \xi_k + \abs*{\xi}^2 \delta_{ik} \right) \\
				&\quad+ \left(p-2\right) \left(p-4\right) \left[\varepsilon^2 + \abs*{\xi}^2\right]^{\frac{p-6}{2}} \abs*{\xi}^2 \,\xi_i \xi_k \quad \mbox{in } \R^n \setminus \left\{0\right\}.
			\end{split}
		\end{equation}
		Finally, by~\eqref{eq:Vij}, we can rewrite~\eqref{eq:torewrite} in the form
		\begin{equation*}
			\partial_{ijk}V^\varepsilon(\xi) \xi_j = \left(p-2\right) \left[1-\frac{\varepsilon^2}{\varepsilon^2 + \abs*{\nabla v^\varepsilon}^2}\right] \partial_{ik}V^\varepsilon(\xi) + \frac{2 \left(p-2\right) \varepsilon^2 \,\xi_i \xi_k}{\left[\varepsilon^2 + \xi^2\right]^2} \left[\varepsilon^2 + \abs*{\xi}^2\right]^{\frac{p-2}{2}} \!.
		\end{equation*}
		Therefore, by~\eqref{eq:defA-ep}, we conclude that
		\begin{align*}
			v^\varepsilon_j \,\partial_{\xi_i\xi_j\xi_k}V^\varepsilon(\nabla v^\varepsilon)v^\varepsilon_{ki} &= \left(p-2\right) \left[1-\frac{\varepsilon^2}{\varepsilon^2 + \abs*{\nabla v^\varepsilon}^2}\right] \alpha_{ik}^\varepsilon v^\varepsilon_{ik}\\
			&\quad+ 2 \left(p-2\right) \frac{\varepsilon^2 \, v^\varepsilon_i \, v^\varepsilon_k}{\left[\varepsilon^2 + \abs*{\nabla v^\varepsilon}^2\right]^2} \left[\varepsilon^2 + \abs*{\nabla v^\varepsilon}^2\right]^{\frac{p-2}{2}} v^\varepsilon_{ik} \quad \mbox{in } B_{r} \setminus \mathcal{Z}_{v^\varepsilon},
		\end{align*}
		from which~\eqref{eq:reltrprove} follows by~\eqref{eq:trWep}.
		
		Since~$P, R \in L^\infty(\R^n)$, by the dominated convergence theorem
		\begin{equation}
		\label{eq:strongconv-lastterm}
			 \left(p-2\right) \left[1-\frac{\varepsilon^2}{\varepsilon^2 + \abs*{\nabla v^\varepsilon}^2}\right] \left(P+R\right) \to \left(p-2\right) \left(P+R\right) \quad \mbox{in } L^2_{\loc}(B_{r}).
		\end{equation}
		Then, we simply have
		\begin{equation*}
			\frac{\varepsilon^2 \,\abs*{v^\varepsilon_i \, v^\varepsilon_k}}{\left[\varepsilon^2 + \abs*{\nabla v^\varepsilon}^2\right]^2} \leq \frac{\varepsilon^2 \,\abs*{\nabla v^\varepsilon}^2}{\left[\varepsilon^2 + \abs*{\nabla v^\varepsilon}^2\right]^2} \leq \frac{1}{4}.
		\end{equation*}
		Thus, taking advantage of~\eqref{eq:ub1}, we deduce
		\begin{equation}
		\label{eq:L2bound}
			\begin{split}
				\int_{B_\rho} &\left\{ \varepsilon^2 \, v^\varepsilon_i \, v^\varepsilon_k \left[\varepsilon^2 + \abs*{\nabla v^\varepsilon}^2\right]^{-2} \left[\varepsilon^2 + \abs*{\nabla v^\varepsilon}^2\right]^{\frac{p-2}{2}} v^\varepsilon_{ik}\right\}^{\!2} dx \\
				&\leq \frac{1}{16} \, \norma*{\left[\varepsilon^2 + \abs*{\nabla v^\varepsilon}^2\right]^{\frac{p-2}{2}} D^2 v^\varepsilon}_{L^2(B_\rho)}^2 \leq C.
			\end{split}
		\end{equation}
		As above, we extract a subsequence~$\left\{\varepsilon_m\right\}_{m \in \N}$ and notice that
		\begin{equation}
		\label{eq:pointconv-frac}
			\frac{\varepsilon^2_m \, v^{\varepsilon_m}_i \, v^{\varepsilon_m}_k}{\left[\varepsilon_m^2 + \abs*{\nabla v^{\varepsilon_m}}^2\right]^2} \left[\varepsilon^2 + \abs*{\nabla v^\varepsilon}^2\right]^{\frac{p-2}{2}} v^{\varepsilon_m}_{ik} \to 0 \quad \mbox{in } \mathcal{B}_r.
		\end{equation}
		From~\eqref{eq:L2bound},~\eqref{eq:pointconv-frac}, and since the subsequence is arbitrary, we conclude that
		\begin{equation}
		\label{eq:weakconv-2term}
			2 \left(p-2\right) \frac{\varepsilon^2_m \, v^{\varepsilon_m}_i \, v^{\varepsilon_m}_k}{\left[\varepsilon_m^2 + \abs*{\nabla v^{\varepsilon_m}}^2\right]^2} \left[\varepsilon^2 + \abs*{\nabla v^\varepsilon}^2\right]^{\frac{p-2}{2}} v^{\varepsilon_m}_{ik} \rightharpoonup 0 \quad \mbox{weakly in } L^{2}_{\loc}(B_{r}).
		\end{equation}
		Hence,~\eqref{eq:weakconv-toprove} immediately follows from~\eqref{eq:strongconv-lastterm} and~\eqref{eq:weakconv-2term}.
		
		We now handle the remaining term in~\eqref{eq:passare-al-limite}, namely that involving~$\tr[W^\varepsilon]^2$, and notice that we cannot directly conclude by~\eqref{eq:conv-stress}. To this end, we write
		\begin{equation*}
			\tr[W^\varepsilon]^2 = \abs*{W^\varepsilon}^2 + \tr[W^\varepsilon]^2 - \abs*{W^\varepsilon}^2
		\end{equation*}
		and apply Lemma~\ref{lem:matrix-est} to~$W^\varepsilon = A^\varepsilon D^2v^\varepsilon$. Note that all these computations hold on~$B_r \setminus \left(\mathcal{Z}_{v^\varepsilon} \cup \mathcal{Z}_{v}\right)$ which has full measure in~$B_r$. Therefore, we deduce that
		\begin{equation}
		\label{eq:est-below-trW2}
			\tr[W^\varepsilon]^2 \geq \abs*{W^\varepsilon}^2 - c_{p,\varepsilon} \,\abs{\mathring{W}^\varepsilon}^2.
		\end{equation}
		In particular, the constant~$c_{p,\varepsilon}$ given in~\eqref{eq:def-c-gm} is determined by the ratio
		\begin{equation*}
			\varrho_{p,\varepsilon} \coloneqq \frac{\lambda_{\mathrm{min}}(A^\varepsilon)}{\lambda_{\mathrm{max}}(A^\varepsilon)}.
		\end{equation*}
		Moreover, if~$\varrho_{p}$ denotes the same ratio referred to~$A$, which is
		\begin{equation*}
			\varrho_{p} \coloneqq \frac{\min\left\{1,p-1\right\}}{\max\left\{1,p-1\right\}} = \left(p-1\right)^{\mathrm{sgn}(2-p)},
		\end{equation*}
		and a straightforward computation reveals that~$\varrho_{p,\varepsilon} \geq \varrho_{p}$. As a consequence, Lemma~\ref{lem:gu-mosc} yields~$c_{p,\varepsilon} \leq c_{p}$, where~$c_{p}$ is defined in~\eqref{eq:def-c-gm} and referred to~$A$, which depends only on~$p$. Combining this estimate with~\eqref{eq:est-below-trW2}, we obtain
		\begin{equation}
		\label{eq:est-below-trW2-2}
			\tr[W^\varepsilon]^2 \geq \abs*{W^\varepsilon}^2 - c_{p} \,\abs{\mathring{W}^\varepsilon}^2.
		\end{equation}
		Now, we notice that
		\begin{equation}
		\label{eq:norma-Wep-senzatr}
			\abs{\mathring{W}^\varepsilon}^2 = \abs{W^\varepsilon}^2 - \frac{1}{n} \left(\tr W^\varepsilon\right)^2 = \abs{W^\varepsilon}^2 - \frac{1}{n} \left(P+R\right)^2 \quad \mbox{a.e.\ in } B_r,
		\end{equation}
		where we used~\eqref{eq:trWep}. From~\eqref{eq:est-below-trW2-2} and~\eqref{eq:norma-Wep-senzatr}, we conclude that
		\begin{equation}
		\label{eq:trW2-final}
			\tr[W^\varepsilon]^2 \geq \left(1-c_p\right) \abs {\mathring{W}^\varepsilon}^2 + \frac{1}{n} \left(P+R\right)^2.
		\end{equation}
		
		Finally, we shall show that
		\begin{equation}
		\label{eq:liminf-senzatr}
			\liminf_{\varepsilon \to 0} \int_{B_{r}}(v^\varepsilon)^{1-n} \,\abs{\mathring{W}^\varepsilon}^2 P^t \phi \, dx \geq \int_{B_{r}} v^{1-n} \,\abs{\mathring{W}}^2 P^t \phi \, dx.
		\end{equation}
		We first prove that
		\begin{equation}
		\label{eq:liminf-toprove}
			\liminf_{\varepsilon \to 0} \int_{B_{r}}(v^\varepsilon)^{1-n} \,\abs*{\nabla a^\varepsilon \!\left(\nabla v^\varepsilon\right)}^2 P^t \phi \, dx \geq \int_{B_{r}} v^{1-n} \,\abs*{\nabla \stressv}^2 P^t \phi \, dx.
		\end{equation}
		From this,~\eqref{eq:conv-vep}, and~\eqref{eq:norma-Wep-senzatr}, we infer the validity of~\eqref{eq:liminf-senzatr}.
		
		We start by writing
		\begin{equation*}
			\begin{split}
				\int_{B_r}(v^\varepsilon)^{1-n} \,\abs*{\nabla a^\varepsilon \!\left(\nabla v^\varepsilon\right)}^2 P^t \phi \, dx - \int_{B_r} v^{1-n} \,\abs*{\nabla \stressv}^2 P^t \phi \, dx = I_1^\varepsilon + I_2^\varepsilon,
			\end{split}
		\end{equation*}
		where
		\begin{align*}
			I_1^\varepsilon &\coloneqq \int_{B_r} \left[(v^\varepsilon)^{1-n}-v^{1-n}\right] \,\abs*{\nabla a^\varepsilon \!\left(\nabla v^\varepsilon\right)}^2 \, P^t \phi \, dx, \\
			I_2^\varepsilon &\coloneqq \int_{B_r} v^{1-n} \left[\abs*{\nabla a^\varepsilon \!\left(\nabla v^\varepsilon\right)}^2 - \abs*{\nabla \stressv}^2\right] P^t \phi \, dx.
		\end{align*}
		Since~$v$ is bounded below, by~\eqref{eq:P-Linf},~\eqref{eq:conv-vep}, and~\eqref{eq:ub3}, we easily deduce that
		\begin{equation}
		\label{eq:limI1}
			\lim_{\varepsilon \to 0} I_1^\varepsilon = 0.
		\end{equation}
		Then, the weak lower semicontinuity of the~$L^2$-norm, together with~\eqref{eq:conv-stress}, yields
		\begin{equation*}
			\liminf_{\varepsilon \to 0} \int_{B_r} \,\abs*{\nabla a^\varepsilon \!\left(\nabla v^\varepsilon\right) v^{\frac{1-n}{2}} (P^t \phi)^{\frac{1}{2}}}^2 \, dx \geq \int_{B_r} v^{1-n} \,\abs*{\nabla a \!\left(\nabla v\right)}^2 \, P^t \phi \, dx. 
		\end{equation*}
		As a consequence, we infer that
		\begin{equation*}
			\liminf_{\varepsilon \to 0} I^\varepsilon_2 \geq 0.
		\end{equation*}
		This, together with~\eqref{eq:limI1}, implies the validity of~\eqref{eq:liminf-toprove}.
		
		By using~\eqref{eq:trW2-final}, we can rewrite~\eqref{eq:passare-al-limite} as
		\begin{align}
			\notag
			-\int_{B_{r}} &(v^\varepsilon)^{2-n} \left\langle A^\varepsilon\nabla P^\varepsilon, \nabla (P^t \phi) \right\rangle dx \geq\int_{B_{r}}(v^\varepsilon)^{1-n} \,\Bigl\{-n \left\langle A^\varepsilon\nabla P^\varepsilon, \nabla v^\varepsilon \right\rangle -P^\varepsilon P -P^\varepsilon R \\
			\notag
			& + n \left(p-1\right) \left[\left(1-c_p\right) \abs {\mathring{W}^\varepsilon}^2 + \frac{1}{n} \left(P+R\right)^2 + \left\langle \nabla P,  a^\varepsilon(\nabla v^\varepsilon) \right\rangle + \left\langle \nabla R,  a^\varepsilon(\nabla v^\varepsilon) \right\rangle \right] \\
			\label{eq:passare-al-limite-final}
			&-P^\varepsilon v^\varepsilon_j\,\partial_{\xi_i\xi_j\xi_k}V^\varepsilon(\nabla v^\varepsilon)v^\varepsilon_{ki}\Bigr\}\, P^t \phi \, dx.
		\end{align}
		Thus, by~\eqref{eq:conv-vep},~\eqref{eq:conv-AgradP},~\eqref{eq:conv-1},~\eqref{eq:conv-2},~\eqref{eq:weakconv-toprove}, and~\eqref{eq:liminf-toprove}, we let~$\varepsilon \to 0$ in~\eqref{eq:passare-al-limite-final} and deduce
		\begin{align}
		\notag
			-\int_{B_{r}} &v^{2-n} \left\langle A\nabla P, \nabla (P^t \phi) \right\rangle dx \geq n \left(p-1\right) \left(1-c_p\right) \int_{B_{r}} v^{1-n} \,\abs{\mathring{W}}^2 P^t \phi \, dx\\
		\label{eq:quasi-final}
			&+ \int_{B_{r}} v^{1-n} \,\bigl\{-n \left\langle A \nabla P, \nabla v \right\rangle + n \left(p-1\right) \bigl[\left\langle \nabla P,  \stressv \right\rangle + \left\langle \nabla R, \stressv \right\rangle \bigr]  \\
		\notag
			&\hspace{1cm} + \left(p-1\right) \left(P+R\right)^2 - \left(p-1\right) P \left(P+R\right) \bigr\}\, P^t \phi \, dx.
		\end{align}
		To conclude, we observe that by~$p$-homogeneity of~$V$, we have
		\begin{equation*}
			\alpha_{ij} v_i = \partial_{\xi_i\xi_j} V \!\left(\nabla v\right) v_i = \left(p-1\right) \partial_{\xi_j} V \!\left(\nabla v\right) = \left(p-1\right) a_j \!\left(\nabla v\right) \quad\mbox{in } \R^n \setminus \mathcal{Z}_v,
		\end{equation*}
		therefore
		\begin{equation}
		\label{eq:final-comp}
			\begin{split}
				-n \left\langle A \nabla P, \nabla v \right\rangle &= - n \,\alpha_{ij}  v_i P_j = -n \left(p-1\right)  P_j \, a_j \!\left(\nabla v\right) \\
				&= - n \left(p-1\right) \left\langle \nabla P,  \stressv \right\rangle \quad\mbox{a.e.\ in } \R^n.
			\end{split}
		\end{equation}
		From~\eqref{eq:quasi-final} and~\eqref{eq:final-comp}, we easily deduce~\eqref{eq:fund-ineq}.
	\end{proof}

	
	\subsection{Quantitative integral estimate.}
	\label{step:quant-est}
	
	We now aim to derive a quantitative estimate for an integral involving~$\abs{\mathring{W}}$, the traceless tensor defined in~\eqref{eq:tracless-W}, by exploiting Proposition~\ref{prop:fund-ineq}. Specifically, for any~$t \geq 1$, we shall prove that
	\begin{equation}
	\label{eq:fund-est-def}
		\int_{\R^n} v^{1-n} \,\abs{\mathring{W}}^2 P^t \, dx + \int_{\R^n} v^{2-n}\,\abs*{\nabla v}^{p-2} P^{t-1} \,\abs*{\nabla P}^{2} \, dx \leq C_\sharp \defi(u,\kappa),
	\end{equation}
	for some constant~$C_\sharp>0$ depending only on~$n$,~$p$,~$\norma*{\kappa}_{L^\infty(\R^n)}$, and~$t$.

	First, recalling~\eqref{eq:def-A}, note that
	\begin{equation}
	\label{eq:AgradP}
		A \nabla P = \abs*{\nabla v}^{p-2} \,\nabla P + \left(p-2\right) \abs*{\nabla v}^{p-4} \left\langle \nabla v, \nabla P \right\rangle \nabla v \quad\mbox{a.e.\ in } \R^n,
	\end{equation}
	therefore
	\begin{equation*}
		\left\langle A \nabla P , \nabla P \right\rangle = \abs*{\nabla v}^{p-2} \,\abs*{\nabla P}^2 + \left(p-2\right) \abs*{\nabla v}^{p-4} \,\abs*{\left\langle \nabla v, \nabla P \right\rangle}^2 \quad\mbox{a.e.\ in } \R^n.
	\end{equation*}
	Let us now define the annulus~$A_r \coloneqq B_{2r} \setminus B_r$. Moreover, let~$\phi \in C^\infty_c(\R^n)$ be a cut-off function such that~$0 \leq \phi \leq 1$, with~$\phi=1$ in~$B_{r}$,~$\phi=0$ in~$\R^n \setminus B_{2r}$, and~$\abs*{\nabla \phi} \leq 1/r$ in~$A_{r}$. By using~$\phi^2$ in~\eqref{eq:fund-ineq} as a test function, we obtain
	\begin{multline}
	\label{eq:est1}
		n \left(p-1\right) \left(1-c_p\right) \int_{B_{2r}} v^{1-n} \,\abs{\mathring{W}}^2 P^t \phi^2 \, dx + t \int_{B_{2r}} v^{2-n}\,\abs*{\nabla v}^{p-2} P^{t-1} \,\abs*{\nabla P}^{2} \phi^2 \, dx \\
		\leq I_1^r + I_2^r + I_3^r + I_4^r,
	\end{multline}
	where
	\begin{align*}
		I_1^r &\coloneqq \left(2-p\right) t \int_{B_{2r}} v^{2-n} \,\abs*{\nabla v}^{p-4} P^{t-1} \,\abs*{\left\langle \nabla v, \nabla P\right\rangle}^2 \,\phi^2  \, dx, \\
		I_2^r &\coloneqq - 2\int_{A_{r}} v^{2-n} \left\langle A\nabla P, \nabla \phi \right\rangle  P^t \phi \, dx, \\
		I_3^r &\coloneqq \left(1-p\right) \int_{B_{2r}} v^{1-n} \left(R^2+PR\right) P^t \phi^2 \, dx, \\
		I_4^r &\coloneqq n \left(1-p\right) \int_{B_{2r}} v^{1-n}  \left\langle \nabla R, \stressv \right\rangle P^t \phi^2 \, dx.
	\end{align*}
	Now, we shall estimate each of these integrals. In the following calculations~$C$ and~$C'$ will be positive constants, possibly differing from line to line, which do not depend on~$r$, but only on~$n$,~$p$,~$\norma*{\kappa}_{L^\infty(\R^n)}$, and~$t$.
	
	For the first integral, we trivially have
	\begin{equation}
	\label{eq:est2}
		I_1^r \leq \left(2-p\right)_{+} t \int_{B_{2r}} v^{2-n} \,\abs*{\nabla v}^{p-2} \,\abs*{\nabla P}^2 \,\phi^2  \, dx,
	\end{equation}
	and notice that~$1-\left(2-p\right)_{+} = \min \left\{p-1,1\right\}$.
	
	For~$I_2^r$, taking advantage of~\eqref{eq:trW} and~\eqref{eq:gradP}, we deduce
	\begin{equation*}
		\abs*{\nabla P} \leq C \left(v^{-1} \,\abs{\mathring{W}} \abs*{\nabla v} + v^{-1} \,\abs*{\nabla v} \abs{R} \right).
	\end{equation*}
	Therefore, by exploiting~\eqref{eq:AgradP} and applying Young's inequality, we get
	\begin{align}
	\notag
		\abs*{I_2^r} &\leq C \int_{A_{r}} v^{2-n} \,\abs*{\nabla v}^{p-2} \,\abs*{\nabla P} \abs*{\nabla \phi} P^t \phi \, dx \\
	\notag
		&\leq C \int_{A_r} v^{1-n} \,\abs*{\nabla v}^{p-1} \,\abs{\mathring{W}} \,\abs*{\nabla\phi} P^t  \phi \, dx +  C' \int_{A_r} v^{1-n} \,\abs*{\nabla v}^{p-1} \,\abs*{\nabla\phi} \abs{R} P^t \phi \, dx \\
	\notag
		&\leq \frac{n}{2} \left(p-1\right) \left(1-c_p\right) \int_{B_{2r}} v^{1-n} \,\abs{\mathring{W}}^2 P^t  \phi^2 \, dx +  C \int_{A_r} v^{1-n} \,\abs*{\nabla v}^{2(p-1)} \,\abs*{\nabla\phi}^2 P^t  \, dx \\
	\label{eq:est3}
		&\quad+ C' \int_{A_r} v^{1-n} \,\abs*{\nabla v}^{p-1} \,\abs*{\nabla\phi} \abs{R} P^t \phi \, dx.
	\end{align}

	We then observe that, for~$r$ large enough, by exploiting~\eqref{eq:defPfunct},~\eqref{eq:D1p},~\eqref{eq:bound-gradv}, and since~$t \geq 1$, we have
	\begin{equation}
	\label{eq:est4}
		\int_{A_r} v^{1-n} \,\abs*{\nabla v}^{2\left(p-1\right)} \,\abs*{\nabla\phi}^2 P^t \, dx \leq C \int_{A_r} v^{-n} + v^{-n} \,\abs*{\nabla v}^p \, dx \to 0,
	\end{equation}
	as~$r \to +\infty$, and analogously, by exploiting~\eqref{eq:defRem},~\eqref{eq:D1p},~\eqref{eq:bound-gradv}, and~\eqref{eq:P-Linf}, we get
	\begin{equation}
	\label{eq:est4'}
		\int_{A_r} v^{1-n} \,\abs*{\nabla v}^{p-1} \,\abs*{\nabla\phi} \abs{R} P^t \phi \, dx \leq C \int_{A_r} v^{-n} \, dx \to 0,
	\end{equation}
	as~$r \to +\infty$.
	
	For the third integral $I_3^r$, since~$P$ satisfies~\eqref{eq:P-Linf}, we start by writing
	\begin{equation*}
		I_3^r \leq C \int_{B_{2r}} v^{1-n} R^2 \,\phi^2 \, dx + C' \int_{B_{2r}} v^{1-n} \,\abs{R} P \,\phi^2 \, dx.
	\end{equation*}
	Then, by~\eqref{eq:defRem},~\eqref{eq:P-Linf}, and  H\"older inequality, we have
	\begin{multline}
	\label{eq:similar}
			\int_{B_{2r}} v^{1-n} \,\abs{R} P \,\phi^2 \, dx \leq C \int_{B_{2r}} v^{-n} \,\abs{\kappa-1} \, dx \\
			\leq C \left(\int_{B_{2r}} v^{-n} \, dx\right)^{\!\!\frac{1}{\past}} \left(\int_{B_{2r}} v^{-n} \,\abs{\kappa-1}^{(\past)'} \, dx\right)^{\!\!\frac{1}{(\past)'}} \leq C \defi(u,\kappa).
	\end{multline}
	Analogously, exploiting also~\eqref{eq:R-Linf}, we deduce
	\begin{equation*}
		\int_{B_{2r}} v^{1-n} R^2 \,\phi^2 \, dx \leq C \defi(u,\kappa).
	\end{equation*}
	As a result,
	\begin{equation}
	\label{eq:est5}
		\abs*{I_3^r} \leq C \defi(u,\kappa).
	\end{equation}

	For~$I_4^r$ we perform an integration by parts to obtain
	\begin{equation}
	\label{eq:est6}
		I_4^r = n \left(1-p\right) \left(J_1^r+J_2^r+J_3^r+J_4^r\right) \!,
	\end{equation}
	where
	\begin{align*}
		J_1^r &\coloneqq \left(1-n\right) \int_{B_{2r}} v^{-n}  \left\langle \nabla v,  \stressv \right\rangle  RP^t \phi^2 \, dx, \\
		J_2^r &\coloneqq t \int_{B_{2r}} v^{1-n}  \left\langle \nabla P, \stressv \right\rangle R P^{t-1} \phi^2 \, dx, \\
		J_3^r &\coloneqq  2 \int_{A_{r}} v^{1-n}  \left\langle \nabla \phi, \stressv \right\rangle  RP^t \phi \, dx, \\
		J_4^r &\coloneqq \int_{B_{2r}} v^{1-n} \left[\tr W \right] RP^t \phi^2 \, dx = \int_{B_{2r}} v^{1-n} \left(P+R\right) RP^t \phi^2 \, dx.
	\end{align*}

	For~$J_1^r$ we exploit~\eqref{eq:defRem},~\eqref{eq:P-Linf}, and~\eqref{eq:bound-fgrad} to deduce
	\begin{equation}
	\label{eq:est7}
		\abs*{J_1^r} \leq C \int_{B_{2r}} v^{-n} \,\abs*{\kappa-1} \, v^{-1}\,\abs*{\nabla v}^p \, dx \leq C \defi(u,\kappa),
	\end{equation}
	arguing as for~\eqref{eq:similar}.
	
	For the second integral we have
	\begin{align*}
	\notag
		\abs{J_2^r} &\leq t \int_{B_{2r}} v^{1-n}  \, \abs*{\nabla P} \abs*{\nabla v}^{p-1} \abs{R} P^{t-1} \phi^2 \, dx \leq C \int_{B_{2r}} v^{-n}  \, \abs*{\nabla P} \abs*{\nabla v}^{p-1} \,\abs*{\kappa-1} P^{t-1} \phi^2 \, dx \\
	\notag
		&= C \int_{B_{2r}} v^{\frac{2-n}{2}} \,\abs*{\nabla v}^{\frac{p-2}{2}} P^{\frac{t-1}{2}} \, \abs*{\nabla P} \,\phi \; v^{-\frac{n}{2}-1} \,\abs*{\nabla v}^{\frac{p}{2}} \,\abs*{\kappa-1} P^{\frac{t-1}{2}} \phi \, dx \\
	\notag
		&\leq \frac{t}{2} \min \left\{p-1,1\right\} \int_{B_{2r}} v^{2-n} \,\abs*{\nabla v}^{p-2} P^{t-1} \, \abs*{\nabla P}^2 \,\phi^2 \, dx \\
	\notag
		&\quad+ C \int_{B_{2r}} v^{-n-2} \,\abs*{\nabla v}^{p} \,\abs*{\kappa-1}^2 P^{t-1} \phi^2 \, dx \\
	\notag
		&\leq \frac{t}{2} \min \left\{p-1,1\right\} \int_{B_{2r}} v^{2-n} \,\abs*{\nabla v}^{p-2} P^{t-1} \,\abs*{\nabla P}^2 \,\phi^2 \, dx \\
	\notag
		&\quad+ C \,\norma*{\kappa-1}_{L^\infty(\R^n)} \int_{B_{2r}} v^{-2} \,\abs*{\nabla v}^{p} \, v^{-n} \,\abs*{\kappa-1} P^{t-1} \phi^2 \, dx.
	\end{align*}	
	Hence, using~\eqref{eq:bound-fgrad} and that~$v$ is bounded below in the last inequality, we obtain
	\begin{equation}
	\label{eq:est8}
		\abs{J_2^r} \leq \frac{t}{2} \min \left\{p-1,1\right\} \int_{B_{2r}} v^{2-n} \,\abs*{\nabla v}^{p-2} P^{t-1} \,\abs*{\nabla P}^2 \, dx + C \defi(u,\kappa).
	\end{equation}
	
	For~$J_3^r$ we exploit~\eqref{eq:defRem},~\eqref{eq:P-Linf}, and~\eqref{eq:D1p} to deduce that, for~$r$ sufficiently large,
	\begin{equation}
	\label{eq:est9}
		\abs*{J_3^r} \leq C \int_{A_{r}} v^{-n} \,\abs*{\nabla v}^{p-1} \,\abs*{\nabla \phi} \, dx \leq C \int_{A_{r}} v^{-n} \, dx \to 0, 
	\end{equation}
	as~$r \to +\infty$.
	
	The fourth term can be handled in a completely analogous way to~$I_3^r$, yielding
	\begin{equation}
	\label{eq:est10}
		\abs*{J_4^r} \leq C \defi(u,\kappa).
	\end{equation}
	
	By collecting~\eqref{eq:est1}--\eqref{eq:est4'} with~\eqref{eq:est5}--\eqref{eq:est10}, letting~$r \to +\infty$, and applying Fatou's lemma, we infer~\eqref{eq:fund-est-def}.
	
	
	\subsection{Fundamental estimate for~$W$.}
	\label{step:fund-W}
	
	Let~$r>1$ to be determined later in dependence of~$\defi(u,\kappa)$. We will choose~$r$ in such a way that~$r \to +\infty$ as~$\defi(u,\kappa) \to 0$.
	
	We aim to derive an estimate for~$\nabla \stressv$ in~$L^q_{w}(B_r)$ with~$q \in [1,2]$ and weight~$w\coloneqq v^{-n}$. Specifically, we will show that, by appropriately selecting a constant~$\mu$, it follows that
	\begin{equation}
	\label{eq:int-to-prove}
		\int_{B_{r}} v^{-n} \,\abs*{\nabla \stressv - \mu \Id}^q \, dx \leq C_\flat \,\mathscr{F}_{\! q} \quad \mbox{for every } q \in [1,2],
	\end{equation}
	for some~$C_\flat>0$ and some function of the deficit~$\mathscr{F}_{\! q}$.
	
	Notice that, by definition of the~$P$-function~\eqref{eq:defPfunct},~\eqref{eq:fund-est-def} with~$t=1$ easily yields
	\begin{equation}
		\label{eq:stima-def}
		\int_{\R^n} v^{-n} \,\abs{\mathring{W}}^2 \, dx \leq C_4 \defi(u,\kappa),
	\end{equation}
	for some~$C_4>0$ depending only  on~$n$,~$p$, and~$\norma*{\kappa}_{L^\infty(\R^n)}$. Therefore, by H\"older inequality, we deduce
	\begin{equation}
	\label{eq:est-W-traceless}
		\int_{\R^n} v^{-n}\,\abs{\mathring{W}}^q \, dx \leq \left(\int_{\R^n} v^{-n}\,\abs{\mathring{W}}^2 \, dx \right)^{\!\!\frac{q}{2}} \left(\int_{\R^n} v^{-n} \, dx \right)^{\!\!\frac{2-q}{2}} \leq C_5 \defi(u,\kappa)^{\frac{q}{2}},
	\end{equation}
	for some~$C_5>0$ depending only  on~$n$,~$p$,~$\norma*{\kappa}_{L^\infty(\R^n)}$, and~$q$.
	
	Recall that in~\ref{step:decay} we have identified a point~$x_0 \in \R^n$ where~$u$ attains its maximum. Consequently,~$v$ attains its minimum at~$x_0$. By taking advantage of~\eqref{eq:bounds-v}, we can quantitatively locate this point, indeed we must have
	\begin{equation}
		\label{eq:bound-vx0}
		v(x_0) \leq v(0) \leq \widehat{C}_0.
	\end{equation}
	Hence, by setting
	\begin{equation}
		\label{eq:defR-rad}
		\mathscr{R} \coloneqq \left(\frac{\widehat{C}_0-\hat{c}_0}{\hat{c}_0}\right)^{\!\!\frac{p-1}{p}},
	\end{equation}
	which depends only on~$n$,~$p$, and~$\norma*{\kappa}_{L^\infty(\R^n)}$, we can conclude that~$x_0 \in B_{\mathscr{R}}$. Let~$\mathsf{t} \in (0,1)$ to be determined later on. From this, we have~$B_\mathsf{t}(x_0) \subseteq B_r$ provided that
	\begin{equation}
		\label{eq:cond-r}
		r \geq \mathscr{R}+2.
	\end{equation}

	We now define
	\begin{equation*}
		\overline{P} \coloneqq \dashint_{B_\mathsf{t}(x_0)} P \, dx
	\end{equation*}
	and observe that, by exploiting~\eqref{eq:trW}, we can write
	\begin{multline}
	\label{eq:norma-L1-final}
		\int_{B_{r}} v^{-n}\,\abs*{W - \frac{\overline{P}}{n} \Id}^q \, dx \\
		\leq 4^{q-1} \int_{B_{r}} v^{-n}\,\abs{\mathring{W}}^q \, dx + 4^{q-1} \int_{B_{r}} v^{-n}\,\abs{P-\overline{P}}^q \, dx 
 		+ 4^{q-1} \int_{B_{r}} v^{-n}\,\abs*{R}^q \, dx \\ 
 		\eqqcolon 4^{q-1} \left\{\int_{B_{r}} v^{-n}\,\abs{\mathring{W}}^q \, dx + \mathcal{I}_1 + \mathcal{I}_2 \right\}\!.
	\end{multline}
	We shall proceed with the estimate of both~$\mathcal{I}_1$ and~$\mathcal{I}_2$. In the subsequent calculations~$C$ will denote a positive constant, possibly differing from line to line, which depends only on~$n$,~$p$,~$\norma*{\kappa}_{L^\infty(\R^n)}$, and~$q$.
	
	By~\eqref{eq:regP}, the Poincar\'e inequality of Equation~(7.45) in~\cite{gt},~\eqref{eq:bounds-v}, and since~$v$ is bounded below, we get 
	\begin{equation}
		\label{eq:est-1'}
		\mathcal{I}_1 \leq C \, r^{\frac{np}{p-1}} \, \mathcal{C}_P \!\left(r,\mathsf{t}\right)^q \int_{B_{r}} v^{-n}\,\abs{\nabla P}^q \, dx,
	\end{equation}
	where
	\begin{equation}
		\label{eq:def-CP}
		\mathcal{C}_P \!\left(r,\mathsf{t}\right) \coloneqq 2^n r^n \, \mathsf{t}^{1-n}.
	\end{equation}
	From~\eqref{eq:trW} and~\eqref{eq:gradP}, we deduce
	\begin{align}
	\notag
		\int_{B_{r}} v^{-n}\,\abs{\nabla P}^q \, dx &\leq C \int_{B_{r}} v^{-n-q} \,\abs{\mathring{W}}^q\abs{\nabla v}^q \, dx + C \int_{B_{r}} v^{-n-q} \,\abs{\nabla v}^q \abs{R}^q \, dx \\
	\label{eq:est-2'}
		&\eqqcolon \mathcal{J}_1 + \mathcal{J}_2.
	\end{align}
	Exploiting the definition of the remainder~\eqref{eq:defRem},~\eqref{eq:bound-fgrad}, and that~$v$ is bounded below, we have
	\begin{align}
	\notag
		\mathcal{J}_2 &\leq C \int_{B_{r}} v^{-n} v^{-2q+\frac{q}{p}} \, v^{-\frac{q}{p}}\,\abs{\nabla v}^q \abs*{\kappa-1}^q \, dx \leq C \int_{B_{r}} v^{-n} \, \abs*{\kappa-1} \, dx \\
	\label{eq:est-3'}
		&\leq C \defi(u,\kappa),
	\end{align}
	where for the last inequality one argues as for~\eqref{eq:similar}. In a similar fashion, we get
	\begin{equation}
	\label{eq:est-4'}
		\mathcal{J}_1 \leq C \int_{B_{r}} v^{-n}v^{-q+\frac{q}{p}} \, v^{-\frac{q}{p}}\,\abs{\nabla v}^q \abs{\mathring{W}}^q \, dx \leq C \defi(u,\kappa)^{\frac{q}{2}},
	\end{equation}
	where we also used~\eqref{eq:est-W-traceless}.
	
	For the second integral, by~\eqref{eq:defRem} and since~$v$ is bounded below, we have
	\begin{equation}
	\label{eq:est-5'}
		\mathcal{I}_2 \leq C \int_{B_r} v^{-n-q} \,\abs*{\kappa-1}^q \, dx \leq C \int_{B_r} v^{-n} \,\abs*{\kappa-1} \, dx \leq C \defi(u,\kappa).
	\end{equation}

	Since~\eqref{eq:def-small} is in force and we may assume~$r>1$, combining~\eqref{eq:norma-L1-final} with~\eqref{eq:est-W-traceless},~\eqref{eq:est-1'}, and~\eqref{eq:est-2'}--\eqref{eq:est-5'}, we deduce that
	\begin{equation*}
	\label{eq:stima-L1-gradA}
		\int_{B_{r}} v^{-n} \,\abs*{\nabla \stressv - \frac{\overline{P}}{n} \Id}^q \, dx \leq C_\flat \, r^{\frac{np}{p-1}} \left[1+ \mathcal{C}_P \!\left(r,\mathsf{t}\right)^q\right] \defi(u,\kappa)^{\frac{q}{2}},
	\end{equation*}
	for some~$C_\flat>0$, depending only on~$n$,~$p$,~$\norma*{\kappa}_{L^\infty(\R^n)}$, and~$q$. This is precisely~\eqref{eq:int-to-prove} with
	\begin{equation*}
		\mathscr{F}_{\! q} \coloneqq r^{\frac{np}{p-1}} \left[1+ \mathcal{C}_P \!\left(r,\mathsf{t}\right)^q\right] \defi(u,\kappa)^{\frac{q}{2}} \quad\mbox{and}\quad \mu = \frac{\overline{P}}{n},
	\end{equation*}
	where~$\mathcal{C}_P \!\left(r,\mathsf{t}\right)$ is given by~\eqref{eq:def-CP}.
	
	
	\subsection{Construction of the approximating functions.}
	\label{step:approx-func}
	
	Our first goal here is to construct a function~$\mathsf{Q}$ that approximates~$v$ in~$L^p_{\omega}(B_{r})$, for~$\omega \coloneqq v^{-n-p+1}$, and whose gradient approximates~$\nabla v$ in~$L^p_{w}(B_{r})$, with~$w=v^{-n}$. This will mainly follow from~\eqref{eq:int-to-prove}. Indeed, it captures proximity information between~$W$ and a suitable multiple of~$\Id$, which is arguably the strongest information one might hope to obtain.
	
	We start by defining the function
	\begin{equation*}
		\mathsf{Q}(x) \coloneqq v(x_0) + \frac{p-1}{p} \left(\frac{\overline{P}}{n}\right)^{\!\frac{1}{p-1}} \abs*{x-x_0}^{\frac{p}{p-1}},
	\end{equation*}
	which is of class~$C^\infty(\R^n \setminus \left\{x_0\right\})$ with~$\mathcal{Z}_{\mathsf{Q}} = \left\{x_0\right\}$. Moreover, up to uniquely extending the stress field on~$\mathcal{Z}_{\mathsf{Q}}$, we notice that
	\begin{gather}
	\label{eq:stressQ}
		a \!\left(\nabla \mathsf{Q}\right)\!(x) = \frac{\overline{P}}{n} \left(x-x_0\right) \quad\mbox{for } x \in \R^n, \\
	\notag
		\nabla a \!\left(\nabla \mathsf{Q}\right)\!(x) = \frac{\overline{P}}{n} \Id \quad\mbox{for } x \in \R^n.
	\end{gather}
	Therefore, we have~$a \!\left(\nabla \mathsf{Q}\right) \in C^\infty(\R^n)$ and, from~\eqref{eq:int-to-prove}, we deduce that
	\begin{equation}
	\label{eq:to-rewrite}
		\int_{B_{r}} v^{-n} \,\abs*{\nabla \stressv - \nabla a \!\left(\nabla \mathsf{Q}\right)}^q \, dx \leq C_\flat \,\mathscr{F}_{\! q} \quad \mbox{for every } q \in [1,2].
	\end{equation}
	For convenience of notation, we set
	\begin{equation*}
	\label{eq:def-rho}
		\zeta \coloneqq \stressv - a \!\left(\nabla \mathsf{Q}\right),
	\end{equation*}
	which is well-defined in~$\R^n$ since~$\stressv$ has been extended to zero on~$\mathcal{Z}_v$, and notice that we can rewrite~\eqref{eq:to-rewrite} as
	\begin{equation}
	\label{eq:stimagrad}
		\int_{B_{r}} v^{-n}\,\abs*{\nabla \zeta}^q \, dx \leq C_\flat \,\mathscr{F}_{\! q} \quad \mbox{for every } q \in [1,2],
	\end{equation}
	with~$\zeta \in W^{1,2}_{\loc}(\R^n)$ by~\eqref{eq:regv}.
	
	We now introduce the small parameter~$\tau \in (0,1)$, to be chosen later, and define, for each component of~$\zeta$,
	\begin{equation*}
		\overline{\zeta_{i}} \coloneqq \dashint_{B_\tau(x_0)} \zeta_{i} \, dx = \dashint_{B_\tau(x_0)} a_i \!\left(\nabla v\right) dx,
	\end{equation*}
	where we used that, by~\eqref{eq:stressQ}, each component of~$a \!\left(\nabla \mathsf{Q}\right)$ is harmonic in~$\R^n$. Therefore, we get
	\begin{equation}
	\label{eq:est1-media}
		\abs{\overline{\zeta_{i}}} \leq \dashint_{B_\tau(x_0)} \,\abs*{a_i \!\left(\nabla v\right)} \, dx .
	\end{equation}
	Furthermore, we also have
	\begin{equation}
	\label{eq:est2-media}
		\begin{split}
			\norma*{\stressv}_{L^1(B_{\tau}(x_0))} &\leq \norma*{\nabla v}_{L^{p-1}(B_{\tau}(x_0) )}^{p-1} = \norma*{\nabla v - \nabla v(x_0)}_{L^{p-1}(B_{\tau}(x_0))}^{p-1} \\
			&\leq C \tau^{n+\alpha\left(p-1\right)},
		\end{split}
	\end{equation}
	for some~$C>0$ depending only on~$n$,~$p$, and~$\norma*{\kappa}_{L^\infty(\R^n)}$.
	
	Note that this is the only one of two points in the proof where we exploit~$C^{1,\alpha}$ estimates in a quantitative form. We remark that this is possible since, by taking advantage of~\eqref{eq:bounds-v}, of the fact that~$x_0 \in B_{\mathscr{R}}$, and of Theorem~1.1 in~\cite{diben}, there exist~$C>0$ and~$\alpha \in (0,1)$, depending at most on~$n$,~$p$, and~$\norma*{\kappa}_{L^\infty(\R^n)}$, such that
	\begin{equation*}
		\abs*{\nabla v (x)- \nabla v (x_0)} \leq C \,\abs*{x-x_0}^\alpha \quad \mbox{for every } x \in B_2(x_0).
	\end{equation*}
	Additionally, by Theorem~3 in~\cite{teix}, we can assume that~$\alpha \in (0,\alpha_M)$, where~$\alpha_M \in (0,1)$ is the maximal exponent of H\"older regularity for the gradient of~$p$-harmonic functions, which depends only on~$n$ and~$p$. As a consequence,~$\alpha$ depends only on~$n$ and~$p$.
	
	As a result,~\eqref{eq:est1-media} simplifies to
	\begin{equation}
	\label{eq:L1-toest-1}
		\abs{\overline{\zeta_{i}}} \leq C \tau^{\alpha\left(p-1\right)},
	\end{equation}
	for some~$C>0$ depending only on~$n$,~$p$, and~$\norma*{\kappa}_{L^\infty(\R^n)}$.
	
	Observe that~\eqref{eq:cond-r} yields that~$B_\tau(x_0) \subseteq B_r$, hence, by applying once more the Poincar\'e inequality of Equation~(7.45) in~\cite{gt}, we get
	\begin{equation}
	\label{eq:L1-toest-2}
		\norma{\zeta_{i} - \overline{\zeta_{i}}}_{L^q_{w}(B_{r})}^q \leq C \, r^{\frac{np}{p-1}} \, \mathcal{C}_P \!\left(r,\tau\right)^q \int_{B_{r}} v^{-n}\,\abs*{\nabla \zeta}^q \, dx \leq C \, r^{\frac{np}{p-1}} \, \mathcal{C}_P \!\left(r,\tau\right)^q \mathscr{F}_{\! q},
	\end{equation}
	where~$\mathcal{C}_P \!\left(r,\tau\right)$ is given by~\eqref{eq:def-CP} and we used~\eqref{eq:stimagrad}.
	
	Since for each component of~$\zeta$ we have
	\begin{equation*}
		\norma{\zeta_{i}}_{L^q_{w}(B_{r})}^q \leq 2^{q-1}\norma{\zeta_{i} - \overline{\zeta_{i}}}_{L^q_{w}(B_{r})}^q + 2^{q-1} \abs{\overline{\zeta_{i}}}^q \,\norma{v^{-n}}_{L^1(\R^n)},
	\end{equation*}
	from~\eqref{eq:L1-toest-1} and~\eqref{eq:L1-toest-2}, we conclude that
	\begin{equation}
	\label{eq:L1-stress}
		\norma{\zeta}_{L^q_{w}(B_{r})}^q \leq C_6 \left[r^{\frac{np}{p-1}} \, \mathcal{C}_P \!\left(r,\tau\right)^q \mathscr{F}_{\! q}+ \tau^{\alpha q\left(p-1\right)} \right] \eqqcolon C_6 \,\mathscr{G}_{q},
	\end{equation}
	for some~$C_6>0$ depending only on~$n$,~$p$,~$\norma*{\kappa}_{L^\infty(\R^n)}$, and~$q$.

	Note that~\eqref{eq:L1-stress} encloses a smallness information at the level of the stress field. We aim to transfer it at the level of the gradient first and of the function~$v-\mathsf{Q}$ later.
	
	It is well-known -- see, for instance,~\cite[(I), Chapter 12]{lindq-plap} -- that for~$p \geq 2$ it holds
	\begin{equation*}
		\abs*{b-a}^p \leq 2^{p-2} \left\langle \abs*{b}^{p-2} \, b - \abs*{a}^{p-2} \, a , b-a\right\rangle \quad \mbox{for every } a,b \in \R^n,
	\end{equation*}
	from which
	\begin{equation*}
		 \abs*{b-a}^{p} \leq 2^{\frac{p(p-2)}{p-1}} \,\abs*{\abs*{b}^{p-2} \, b - \abs*{a}^{p-2} \, a }^{\frac{p}{p-1}} \quad \mbox{for every } a,b \in \R^n.
	\end{equation*}
	Thus, for~$p \geq 2$,~\eqref{eq:L1-stress} entails
	\begin{equation}
		\label{eq:norma-p-grad->2}
		\int_{B_{r}} v^{-n}\,\abs*{\nabla \!\left(v-\mathsf{Q}\right) }^{p} \, dx \leq  2^{\frac{p(p-2)}{p-1}} C_6 \,\mathscr{G}_{\frac{p}{p-1}}.
	\end{equation}
	
	For~$1<p \leq 2$, it is well-established -- see, for instance,~\cite[(VII), Chapter 12]{lindq-plap} -- that
	\begin{equation*}
		\left(p-1\right) \left(1+\abs*{a}^2+\abs*{b}^2\right)^{\!\!\frac{p-2}{2}}\abs*{b-a}^2 \leq \left\langle \abs*{b}^{p-2} \, b - \abs*{a}^{p-2} \, a , b-a\right\rangle \quad \mbox{for all } a,b \in \R^n \setminus \left\{0\right\},
	\end{equation*}
	from which
	\begin{equation*}
		\abs*{b-a}^p \leq \frac{1}{\left(p-1\right)^p} \left(1+\abs*{a}^2+\abs*{b}^2\right)^{\!\!\frac{p(2-p)}{2}} \abs*{\abs*{b}^{p-2} \, b - \abs*{a}^{p-2} \, a }^p \quad \mbox{for all } a,b \in \R^n \setminus \left\{0\right\}.
	\end{equation*}
	We also notice that, taking advantage of~\eqref{eq:bound-gradv} and of the definition of~$\mathsf{Q}$, it holds
	\begin{equation}
	\label{eq:bounds-grad-E}
		\abs*{\nabla v} \leq 2 C_1 \, r^{\frac{1}{p-1}} \quad\mbox{and}\quad \abs*{\nabla \mathsf{Q}} \leq C \, r^{\frac{1}{p-1}} \quad\mbox{in } B_r,
	\end{equation}
	for some~$C>0$ depending only on~$n$,~$p$, and~$\norma*{\kappa}_{L^\infty(\R^n)}$.
	As a consequence, by exploiting~\eqref{eq:L1-stress} and~\eqref{eq:bounds-grad-E}, we get
	\begin{align}
	\notag
		\int_{B_{r}} v^{-n}\,\abs*{\nabla \!\left(v-\mathsf{Q}\right)}^p \, dx &\leq \frac{1}{\left(p-1\right)^p} \int_{B_{r}} v^{-n}\left(1+\abs*{\nabla v}^2+\abs*{\nabla \mathsf{Q}}^2\right)^{\!\!\frac{p(2-p)}{2}} \abs*{\zeta}^p \, dx \\
	\label{eq:norma-p-1-<2}
		&\leq C \, r^{\frac{p(2-p)}{p-1}} \int_{B_{r}} v^{-n}\,\abs*{\zeta}^p \, dx \leq C \, C_6 \, r^{\frac{p\left(2-p\right)}{p-1}} \, \mathscr{G}_{p},
	\end{align}
	for some~$C>0$ depending only on~$n$,~$p$, and~$\norma*{\kappa}_{L^\infty(\R^n)}$. Thus, from~\eqref{eq:norma-p-grad->2} and~\eqref{eq:norma-p-1-<2}, we read
	\begin{equation}
	\label{eq:norma-p-grad}
		\int_{B_{r}} v^{-n}\,\abs*{\nabla \!\left(v-\mathsf{Q}\right)}^{p} \, dx \leq C_7 \, r^{\frac{p\left(2-p\right)_+}{p-1}} \, \mathscr{G}_{\min\left\{p,\frac{p}{p-1}\right\}},
	\end{equation}
	for some~$C_7>0$ depending only on~$n$,~$p$, and~$\norma*{\kappa}_{L^\infty(\R^n)}$.
	
	By~\eqref{eq:bounds-v} and the fact that~$x_0 \in B_{\mathscr{R}}$, we easily deduce that
	\begin{gather}
	\label{eq:bounds-v-2}
		\hat{c}_1 \left(1+\abs*{x-x_0}\right)^\frac{p}{p-1} \leq v(x) \leq \widehat{C}_1 \left(1+\abs*{x-x_0}\right)^\frac{p}{p-1} \quad\mbox{for every } x \in \R^n, \\
	\label{eq:bounds-v-3}
		\hat{c}_1 \left(1+\abs*{x}\right)^\frac{p}{p-1} \leq v(x) \leq \widehat{C}_1 \left(1+\abs*{x}\right)^\frac{p}{p-1} \quad\mbox{for every } x \in \R^n,
	\end{gather}
	for a couple of constants~$\hat{c}_1, \widehat{C}_1>0$ depending only on~$n$,~$p$, and~$\norma*{\kappa}_{L^\infty(\R^n)}$.
	
	Let~$\sigma>p-1$ to be determined soon. Since~$v(x_0)=\mathsf{Q}(x_0)$, the fundamental theorem of calculus and Jensen inequality yield
	\begin{align}
	\label{eq:norma-p-func-1}
		\int_{B_{r}} v^{-\sigma}(x)\,\abs*{\left(v-\mathsf{Q}\right)\!(x)}^{p} \, dx &= \int_{B_{r}} v^{-\sigma}(x) \,\abs*{\left(v-\mathsf{Q}\right)\!(x)-\left(v-\mathsf{Q}\right)\!(x_0)}^{p}\, dx \\
	\notag
		&= \int_{B_{r}} v^{-\sigma}(x)\,\abs*{\int_{0}^{1} \frac{d}{dt} \left(v-\mathsf{Q}\right)\!\left(tx+(1-t)x_0\right) \, dt}^{p} dx \\
	\notag
		&\leq \int_{B_{r}}  v^{-\sigma}(x)\,\abs*{x-x_0}^p \int_{0}^{1} \,\abs*{ \nabla \!\left(v-\mathsf{Q}\right)\!\left(tx+(1-t)x_0\right)}^{p} \, dt \, dx.
	\end{align}
	By~\eqref{eq:bounds-v-2}, we also have
	\begin{align*}
		v^{-\sigma}(x)\,\abs*{x-x_0}^p &\leq C  \left(1+\abs*{x-x_0}\right)^{-\sigma\frac{p}{p-1}+p}  \left(1+\abs*{x-x_0}\right)^{-p} \abs*{x-x_0}^p \\
		&\leq  C \left(1+\abs*{x-x_0}\right)^{-\sigma\frac{p}{p-1}+p},
	\end{align*}
	for some~$C>0$ depending only on~$n$,~$p$, and~$\norma*{\kappa}_{L^\infty(\R^n)}$, therefore
	\begin{align}
	\label{eq:norma-p-func-2}
		\int_{B_{r}} &\, v^{-\sigma}(x)\,\abs*{x-x_0}^p \int_{0}^{1} \,\abs*{ \nabla \!\left(v-\mathsf{Q}\right)\!\left(tx+(1-t)x_0\right)}^{p} \, dt \, dx \\
	\notag
		&\leq C \int_{B_{r}} \int_{0}^{1} \left(1+\abs*{x-x_0}\right)^{-\sigma\frac{p}{p-1}+p} \,\abs*{\nabla \!\left(v-\mathsf{Q}\right)\!\left(tx+(1-t)x_0\right)}^{p} \, dt \, dx.
	\end{align}
	On the other hand, by taking advantage of the fact that~$x_0 \in B_{\mathscr{R}}$, one can directly verify that
	\begin{equation*}
		\frac{1+\abs*{x-x_0}}{1+\abs*{tx+(1-t)x_0}} \geq \frac{1}{1+\abs*{x_0}} \geq \frac{1}{1+\mathscr{R}} \quad\mbox{for every } t \in [0,1].
	\end{equation*}
	This, together with~\eqref{eq:bounds-v-3},~\eqref{eq:norma-p-func-1},~\eqref{eq:norma-p-func-2}, and Tonelli’s theorem, gives
	\begin{align}
	\label{eq:norma-p-func-3}
		\int_{B_{r}} &\, v^{-\sigma}(x)\,\abs*{\left(v-\mathsf{Q}\right)\!(x)}^{p} \, dx \\
	\notag
		&\leq C \int_{B_{r}} \int_{0}^{1} \left(1+\abs*{tx+(1-t)x_0}\right)^{-\sigma\frac{p}{p-1}+p} \,\abs*{\nabla \!\left(v-\mathsf{Q}\right)\!\left(tx+(1-t)x_0\right)}^{p} \, dt \, dx \\
	\notag
		&\leq C \int_{0}^{1} \int_{B_{r}} \left(1+\abs*{y}\right)^{-\frac{p}{p-1}\left(\sigma-p+1\right)} \,\abs*{\nabla \!\left(v-\mathsf{Q}\right)\!\left(y\right)}^{p} \, dy \, dt \\
	\notag
		&\leq C \int_{B_{r}} v^{-\left(\sigma-p+1\right)}(y) \,\abs*{\nabla \!\left(v-\mathsf{Q}\right)\!\left(y\right)}^{p} \, dy,
	\end{align}
	for some~$C>0$ depending only on~$n$,~$p$,~$\norma*{\kappa}_{L^\infty(\R^n)}$, and~$\sigma$. By choosing~$\sigma=n+p-1$ in~\eqref{eq:norma-p-func-3} and using~\eqref{eq:norma-p-grad}, we deduce
	\begin{equation}
	\label{eq:norma-p-func}
		\int_{B_{r}} v^{-n-p+1}\,\abs*{\left(v-\mathsf{Q}\right)}^{p} \, dx \leq C_8 \, r^{\frac{p\left(2-p\right)_+}{p-1}} \, \mathscr{G}_{\min\left\{p,\frac{p}{p-1}\right\}},
	\end{equation}
	for some~$C_8>0$ depending only on~$n$,~$p$, and~$\norma*{\kappa}_{L^\infty(\R^n)}$.
	
	Note that~\eqref{eq:norma-p-grad} and~\eqref{eq:norma-p-func} imply that~$\mathsf{Q}$ approximates~$v$ both at the zero and first order in an appropriate sense. However, when inverting~$\mathsf{Q}$ through the transformation in~\eqref{eq:defv}, to go back to~$u$, the result is not a~$p$-bubble of the form~\eqref{eq:pbubb}, but only a function belonging to the manifold of Talenti bubbles. This suggests that a further approximation is required.
	
	To this purpose, we define
	\begin{equation*}
		\mathcal{Q}(x) \coloneqq \frac{\lambda^\frac{p}{p-1}+\abs*{x-x_0}^\frac{p}{p-1}}{\lambda^\frac{1}{p-1} \, n^\frac{1}{p} \left(\frac{n-p}{p-1}\right)^{\!\!\frac{p-1}{p}}} \quad\mbox{with}\quad \lambda = \frac{1}{\overline{P}} \left(\frac{p}{p-1}\right)^{\! p-1} n^\frac{1}{p} \left(\frac{n-p}{p-1}\right)^{\!\!- \frac{(p-1)^2}{p}}\!.
	\end{equation*}
	A direct computation reveals that
	\begin{equation}
	\label{eq:diffQ}
		\nabla \mathcal{Q} = \nabla \mathsf{Q} \quad\mbox{in } \R^n,
	\end{equation}
	therefore~\eqref{eq:norma-p-grad} holds with~$\mathcal{Q}$ in place of~$\mathsf{Q}$, that is
	\begin{equation}
	\label{eq:norma-p-grad-2}
		\int_{B_{r}} v^{-n}\,\abs*{\nabla \!\left(v-\mathcal{Q}\right) }^{p} \, dx \leq C_7 \, r^{\frac{p\left(2-p\right)_+}{p-1}} \, \mathscr{G}_{\min\left\{p,\frac{p}{p-1}\right\}},
	\end{equation}
	Moreover, we have
	\begin{equation}
	\label{eq:diff-per-Linf}
		\mathsf{Q}(x) - \mathcal{Q}(x) = v(x_0) - \frac{1}{\overline{P} }\left(\frac{p}{n-p}\right)^{\! p-1} \quad\mbox{for every } x \in \R^n,
	\end{equation}
	and we shall now show that the right-hand side is quantitatively small. To this aim, we compute
	\begin{align*}
		 v(x_0) \overline{P} &- \left(\frac{p}{n-p}\right)^{\! p-1} \\
		 &= \dashint_{B_\mathsf{t}(x_0)} \frac{n(p-1)}{p} \frac{v(x_0)}{v} \,\abs*{\nabla v-\nabla v(x_0)}^p + \left(\frac{p}{n-p}\right)^{\! p-1} \frac{v(x_0)-v}{v} \, dx.
	\end{align*}
	Thus, by exploiting~\eqref{eq:regv},~\eqref{eq:bound-gradv}, as well as the fact that~$v$ is bounded below, and arguing as in the deduction of~\eqref{eq:est2-media}, we obtain
	\begin{equation}
	\label{eq:to-mult}
		\abs*{v(x_0) \overline{P} - \left(\frac{p}{n-p}\right)^{\! p-1}} \leq C \,\mathsf{t}^{\alpha p},
	\end{equation}
	for some~$C>0$ depending only on~$n$,~$p$, and~$\norma*{\kappa}_{L^\infty(\R^n)}$. Moreover, since we have already proven that~$x_0 \in B_{\mathscr{R}}$, with~$\mathscr{R}$ defined in~\eqref{eq:defR-rad}, the bounds~\eqref{eq:bound-gradv} and~\eqref{eq:bound-vx0} give that~$v \leq C$ in~$B_\mathsf{t}(x_0)$, for some~$C>0$ depending only on~$n$,~$p$, and~$\norma*{\kappa}_{L^\infty(\R^n)}$. As a consequence,
	\begin{equation}
	\label{eq:bound-mean-below}
		\overline{P} \geq \left(\frac{p}{n-p}\right)^{\! p-1} \dashint_{B_\mathsf{t}(x_0)}  v^{-1} \, dx \geq C,
	\end{equation}
	for some~$C>0$ depending only on~$n$,~$p$, and~$\norma*{\kappa}_{L^\infty(\R^n)}$. This, together with~\eqref{eq:P-Linf}, also implies that
	\begin{equation}
	\label{eq:bounds-lambda}
		\Lambda_1 \leq \lambda \leq \Lambda_2,
	\end{equation}
	for some~$\Lambda_1,\Lambda_2>0$ depending only on~$n$,~$p$, and~$\norma*{\kappa}_{L^\infty(\R^n)}$. In addition, dividing both sides of~\eqref{eq:to-mult} by~$\overline{P}$, and exploiting~\eqref{eq:bound-mean-below}, we conclude that
	\begin{equation*}
		\abs*{v(x_0) - \frac{1}{\overline{P}} \left(\frac{p}{n-p}\right)^{\! p-1}} \leq C \,\mathsf{t}^{\alpha p},
	\end{equation*}
	for some~$C>0$ depending only on~$n$,~$p$, and~$\norma*{\kappa}_{L^\infty(\R^n)}$. From the latter and~\eqref{eq:diff-per-Linf}, we read off
	\begin{equation*}
		\norma*{\mathsf{Q} - \mathcal{Q}}_{L^\infty(\R^n)} \leq C \,\mathsf{t}^{\alpha p},
	\end{equation*}
	which, in turn, implies
	\begin{equation}
	\label{eq:closQ}
		\norma*{\mathsf{Q} - \mathcal{Q}}_{L^p_\omega(\R^n)}^p \leq \norma*{\mathsf{Q} - \mathcal{Q}}_{L^\infty(\R^n)}^p \int_{\R^n} v^{-n-p+1} \, dx \leq C_9 \,\mathsf{t}^{\alpha p^2},
	\end{equation}
	for some~$C_9>0$ depending only on~$n$,~$p$, and~$\norma*{\kappa}_{L^\infty(\R^n)}$.

	Taking advantage of~\eqref{eq:norma-p-func} and~\eqref{eq:closQ}, we deduce
	\begin{align}
	\notag
		\int_{B_{r}} v^{-n-p+1} \,\abs*{v-\mathcal{Q}}^{p} \, dx
		&\leq 2^{p-1}  C_8 \, r^{\frac{p\left(2-p\right)_+}{2\left(p-1\right)}} \, \mathscr{G}_{\min\left\{p,\frac{p}{p-1}\right\}} + 2^{p-1} C_9 \,\mathsf{t}^{\alpha p^2} \\
	\label{eq:diff-vQ}
		&\leq C_{10} \left[ r^{\frac{p\left(2-p\right)_+}{p-1}} \, \mathscr{G}_{\min\left\{p,\frac{p}{p-1}\right\}} + \mathsf{t}^{\alpha p^2} \right] \!,
	\end{align}
	for some~$C_{10}>0$ depending only on~$n$,~$p$, and~$\norma*{\kappa}_{L^\infty(\R^n)}$.
	
	
	\subsection{Going back to~$u$.}
	\label{step:backtou}
	
	We are now ready to return to the level of~$u$ by defining the approximation for our original function and showing the smallness of the~$\mathcal{D}^{1,p}$-norm. Finally, we will adjust the parameter introduced so far in terms of~$\defi(u,\kappa)$.
	
	For this purpose, we define the function which approximates~$u$ by
	\begin{equation*}
		\mathcal{U} \coloneqq \mathcal{Q}^{-\frac{n-p}{p}},
	\end{equation*}
	inverting~$\mathcal{Q}$ through~\eqref{eq:defv}. In such a way,~$\mathcal{U}$ is a~$p$-bubble of the form~\eqref{eq:pbubb}, in particular~$\mathcal{U} = U_p[x_0,\lambda]$. Moreover, a direct computation reveals that
	\begin{equation*}
		\nabla u = \frac{p-n}{p} \, v^{-\frac{n}{p}} \,\nabla v \quad\mbox{and}\quad \nabla \mathcal{U} = \frac{p-n}{p} \, \mathcal{Q}^{-\frac{n}{p}} \,\nabla \mathcal{Q} \quad \mbox{in } \R^n.
	\end{equation*}
	Therefore, we see that
	\begin{multline}
	\label{eq:normp-toest-1}
			\norma*{\nabla \!\left(u-\mathcal{U}\right) }_{L^p(B_{r})}^p = \left(\frac{n-p}{p}\right)^{\! p} \int_{B_{r}} \,\abs*{v^{-\frac{n}{p}} \,\nabla v-\mathcal{Q}^{-\frac{n}{p}} \,\nabla \mathcal{Q}} \, dx \\
			\leq 2^{p-1} \left(\frac{n-p}{p}\right)^{\! p} \left\{\int_{B_{r}} v^{-n}  \,\abs*{\nabla \!\left(v-\mathcal{Q}\right)}^p \, dx + \int_{B_{r}} \,\abs*{v^{-\frac{n}{p}}-\mathcal{Q}^{-\frac{n}{p}}}^p \,\abs*{\nabla \mathcal{Q}}^p \, dx\right\} \!.
	\end{multline}
	We now aim to estimate both terms on the right hand side of~\eqref{eq:normp-toest-1}.
	
	By taking advantage of~\eqref{eq:diffQ} and~\eqref{eq:bounds-lambda}, we easily see that
	\begin{equation*}
		\mathcal{Q}^{-1} \,\abs*{\nabla\mathcal{Q}}^p \leq C,
	\end{equation*}
	for some~$C>0$ depending only on~$n$,~$p$, and~$\norma*{\kappa}_{L^\infty(\R^n)}$, from which
	\begin{align}
	\notag
		\abs*{v^{-\frac{n}{p}}-\mathcal{Q}^{-\frac{n}{p}}}^p \,\abs*{\nabla\mathcal{Q}}^p &= v^{-n} \mathcal{Q}^{-n} \,\abs*{v^{\frac{n}{p}}-\mathcal{Q}^{\frac{n}{p}}}^p \,\abs*{\nabla\mathcal{Q}}^p \\
	\label{eq:est-secnd-int}
		&\leq C v^{-n} \mathcal{Q}^{-n+1} \,\abs*{v^{\frac{n}{p}}-\mathcal{Q}^{\frac{n}{p}}}^p.
	\end{align}
	Then, the fundamental theorem of calculus and Jensen inequality entail
	\begin{align}
	\label{eq:est-secnd-int-1}
		\abs*{v^{\frac{n}{p}}-\mathcal{Q}^{\frac{n}{p}}}^p &= \abs*{\int_{0}^{1} \frac{d}{dt} \left(tv + (1-t)\mathcal{Q}\right)^{\frac{n}{p}} dt}^{p} \\
	\notag
		&\leq \left(\frac{n}{p}\right)^{\! p} \int_{0}^{1} \left(tv + (1-t)\mathcal{Q}\right)^{n-p} \, dt \,\abs*{v-\mathcal{Q}}^p.
	\end{align}
	In light of~\eqref{eq:bounds-v-2} and~\eqref{eq:bounds-lambda}, we have
	\begin{equation*}
		\hat{c}_2 \, v \leq \mathcal{Q} \leq \widehat{C}_2 \, v \quad\mbox{in } \R^n,
	\end{equation*}
	for a couple of constants~$\hat{c}_2, \widehat{C}_2>0$ depending only on~$n$,~$p$, and~$\norma*{\kappa}_{L^\infty(\R^n)}$. Therefore,
	\begin{align}
	\label{eq:est-secnd-int-2}
		v^{-n} \mathcal{Q}^{-n+1} \int_{0}^{1} \left(tv + (1-t)\mathcal{Q}\right)^{n-p} \, dt \,\abs*{v-\mathcal{Q}}^p &\leq 	v^{-n} \mathcal{Q}^{1-p} \,\abs*{v-\mathcal{Q}}^p \\
	\notag
		&\leq v^{-n-p+1} \,\abs*{v-\mathcal{Q}}^p \quad\mbox{on } \left\{v \leq \mathcal{Q}\right\}\!,
	\end{align}
	and
	\begin{align}
	\label{eq:est-secnd-int-3}
		v^{-n} \mathcal{Q}^{-n+1} \int_{0}^{1} \left(tv + (1-t)\mathcal{Q}\right)^{n-p} \, dt \,\abs*{v-\mathcal{Q}}^p &\leq 	v^{-p} \mathcal{Q}^{-n+1} \,\abs*{v-\mathcal{Q}}^p \\
	\notag
		&\leq \hat{c}_2^{1-n} \, v^{-n-p+1} \,\abs*{v-\mathcal{Q}}^p \quad\mbox{on } \left\{\mathcal{Q} \leq v\right\}\!.
	\end{align}
	Piecing~\eqref{eq:est-secnd-int}--\eqref{eq:est-secnd-int-3}, it follows that
	\begin{equation}
	\label{eq:normp-toest-3}
		\int_{B_{r}} \,\abs*{v^{-\frac{n}{p}}-\mathcal{Q}^{-\frac{n}{p}}}^p \,\abs*{\nabla \mathcal{Q}}^p \, dx \leq C \int_{B_{r}} v^{-n-p+1} \,\abs*{v-\mathcal{Q}}^{p} \, dx,
	\end{equation}
	for some constant~$C>0$ depending only on~$n$,~$p$, and~$\norma*{\kappa}_{L^\infty(\R^n)}$.
	
	Combining~\eqref{eq:normp-toest-1} with~\eqref{eq:norma-p-grad-2},~\eqref{eq:diff-vQ}, and~\eqref{eq:normp-toest-3}, we conclude that
	\begin{equation}
	\label{eq:D1p-int}
		\norma*{\nabla \!\left(u-\mathcal{U}\right) }_{L^p(B_{r})}^p \leq C_{11} \left[ r^{\frac{p\left(2-p\right)_+}{p-1}} \, \mathscr{G}_{\min\left\{p,\frac{p}{p-1}\right\}} + \mathsf{t}^{\alpha p^2} \right]\!,
	\end{equation} 
	for some~$C_{11}>0$ depending only on~$n$,~$p$, and~$\norma*{\kappa}_{L^\infty(\R^n)}$, with
	\begin{equation*}
		\mathscr{G}_{q} = r^{\frac{2np}{p-1}} \, \mathcal{C}_P \!\left(r,\tau\right)^q \left[1+ \mathcal{C}_P \!\left(r,\mathsf{t}\right)^q\right] \defi(u,\kappa)^{\frac{q}{2}}+ \tau^{\alpha q\left(p-1\right)}
	\end{equation*}
	and where~$\mathcal{C}_P \!\left(r,\cdot\right)$ is defined in~\eqref{eq:def-CP}.
	
	Moreover, by~\eqref{eq:cond-r}, it holds that~$B_1(x_0) \subseteq B_{r}$, therefore, exploiting also~\eqref{eq:bounds-lambda}, we have
	\begin{equation*}
		\abs*{\nabla \mathcal{U}(x)} \leq C \,\abs*{x}^{\frac{1-n}{p-1}} \quad \mbox{for } x \in \R^n \setminus B_{r},
	\end{equation*}
 	for some~$C>0$ depending only on~$n$,~$p$, and~$\norma*{\kappa}_{L^\infty(\R^n)}$. The latter and~\eqref{eq:bound-gradu} imply
 	\begin{equation}
 	\label{eq:D1p-ext}
 		\norma*{\nabla \!\left(u-\mathcal{U}\right) }_{L^p(\R^n \setminus B_{r})}^p \leq C \int_{\R^n \setminus B_{r}} \,\abs*{x}^{-\frac{p}{p-1} (n-1)} \, dx \leq C \, r^{-\frac{n-p}{p-1}},
 	\end{equation}
 	for some~$C>0$ depending only on~$n$,~$p$, and~$\norma*{\kappa}_{L^\infty(\R^n)}$.
 	
 	We finish this step by adjusting the parameters. We choose
 	\begin{align*}
 		r&\coloneqq \defi(u,\kappa)^{-\min\left\{\frac{q}{4\mathfrak{p}},\frac{\alpha q\left(p-1\right)^2}{16\left(n-1\right)p\left(2-p\right)_+}\right\}} = \defi(u,\kappa)^{-\mathsf{m}}, \\
 		\mathsf{t}=\tau&\coloneqq \defi(u,\kappa)^{\frac{1}{8(n-1)}},
 	\end{align*}
 	where
 	\begin{equation*}
 		q \coloneqq \min\left\{p,\frac{p}{p-1}\right\} \quad\mbox{and}\quad \mathfrak{p} \coloneqq 2nq+\frac{p}{p-1}\left[2n+\left(2-p\right)_+\right].
 	\end{equation*}
 	Thus, by summing~\eqref{eq:D1p-int} with~\eqref{eq:D1p-ext}, we deduce
 	\begin{equation*}
 		\norma*{\nabla \!\left(u-\mathcal{U}\right) }_{L^p(\R^n)} \leq \mathscr{C} \defi(u,\kappa)^{\vartheta},
 	\end{equation*}
 	some~$\mathscr{C}>0$, depending only on~$n$,~$p$, and~$\norma*{\kappa}_{L^\infty(\R^n)}$, and some~$\vartheta \in (0,1)$, depending only on~$n$ and~$p$. In addition, we observe that the above requirements on the parameters, in particular~\eqref{eq:cond-r}, are verified provided that
 	\begin{equation*}
 		\gamma \leq \min\left\{\frac{1}{2}, \delta, \left(\mathscr{R}+2\right)^{\!-\frac{1}{\mathsf{m}}} \right\},
 	\end{equation*}
 	where~$\gamma$ and~$\delta$ verifies~\eqref{eq:def-small} and~\eqref{eq:delta-to-fix}, respectively.

	
	\subsection{Conclusion.}
	\label{step:conclusion} 
	
	Here, we need to restore the subscript, as announced at the end of~\ref{step:struwe}, and therefore write~$u_\epsilon$ instead of~$u$.
	
	In~\ref{step:backtou}, we have proven that
	\begin{equation}
	\label{eq:conclusion}
		\norma*{\nabla \!\left(u_\epsilon-\mathcal{U}\right) }_{L^p(\R^n)} \leq \mathscr{C} \defi(u,\kappa)^{\vartheta}.
	\end{equation}
	
 	Recall that we set~$u_\epsilon = T_{-z_\epsilon/\lambda_\epsilon,\lambda_\epsilon} \!\left(u\right)$. Hence, point~\ref{it:inver} of Lemma~\ref{lem:sym} implies that~$u = T_{z_\epsilon,\lambda_\epsilon} \!\left(u_\epsilon\right)$. Defining~$\mathcal{U}_0 \coloneqq T_{z_\epsilon,\lambda_\epsilon} \!\left(\mathcal{U}\right)$, which remains a~$p$-bubble by point~\ref{it:stillbub} of Lemma~\ref{lem:sym},~\eqref{eq:conclusion} reads as
 	\begin{equation*}
 		\norma*{\nabla \!\left(u-\mathcal{U}_0\right) }_{L^p(\R^n)} \leq \mathscr{C} \defi(u,\kappa)^{\vartheta},
 	\end{equation*}
 	which is the desired conclusion.
 	
 	
 	\section{On the optimal stability result}
 	\label{sec:sharp-res}
	In the forthcoming paper~\cite{liu-zhang}, the authors prove that for any non-negative~$u \in \mathcal{D}^{1,p}(\R^n)$ satisfying~\eqref{eq:ipotesi-energ}, there exist a large constant~$C \geq 1$ and a~$p$-bubble~$\mathcal{U}_0$, of the form~\eqref{eq:pbubb}, such that the estimate
	\begin{equation}
 	\label{eq:ideal-est}
 		\norma*{u-\mathcal{U}_0}_{\mathcal{D}^{1,p}(\R^n)} \leq C \,\norma*{\Delta_p u + u^{\past-1}}_{\mathcal{D}^{-1,p'}\!(\R^n)}^{\min\left\{1,\frac{1}{p-1}\right\}} 
 	\end{equation}
 	holds. Hereafter,~$\mathcal{D}^{-1,p'}\!(\R^n)$ denotes the dual space of~$\mathcal{D}^{1,p}(\R^n)$.

 	Truthfully, the estimate in~\eqref{eq:ideal-est} is stronger than that in~\eqref{eq:close-toabub}. Indeed, by Sobolev embedding and duality, we have
 	\begin{equation*}
 		L^{(\past)'}\!(\R^n) \longhookrightarrow \mathcal{D}^{-1,p'}\!(\R^n),
 	\end{equation*}
 	which allows us to control the norm on the right-hand side of~\eqref{eq:ideal-est} via~$\defi(u,\kappa)$ for a solution to~\eqref{eq:maineq-bubb}. Moreover, the proof of Theorem~\ref{th:main-bubbles} fundamentally relies on the fact that~$u$ is a solution to~\eqref{eq:maineq-bubb}. \newline
 	
 	In the remainder of this section, we provide an example which shows the optimality of~\eqref{eq:ideal-est} for~$1<p<2$. To this aim, we adapt to our setting the approach used in~\cite{cfm} to show the optimality of the result therein -- see~\cite[Remark~1.2]{cfm}. Essentially, we construct a suitable perturbation of a~$p$-bubble by adding a small cut-off function, which does not significantly perturb the energy and ensures that the new function remains close to being a solution of~\eqref{eq:eqcritica}.

 	Let us fix~$U \coloneqq U_p[0,1]$ and consider a non-negative, radially symmetric, and radially decreasing cut-off function~$\phi \in C^\infty_c(B_1)$. For~$\varepsilon \in (0,1)$, define
 	\begin{equation*}
 		u_\varepsilon \coloneqq U +\varepsilon\phi,
 	\end{equation*}
 	which is a small, non-negative perturbation of~$U$. Since~$\phi$ is radial with respect to the origin, it easily follows that~$\mathcal{Z}_{u_\varepsilon} = \left\{0\right\}$ and that~$u_\varepsilon \in C^\infty(\R^n \setminus \left\{0\right\}) \cap C^1(\R^n)$.
 	
 	We shall show that~$u_\varepsilon \in \mathcal{D}^{1,p}(\R^n)$ verifies~\eqref{eq:ipotesi-energ} provided~$\varepsilon$ is sufficiently small, and we will provide precise estimates for both sides of~\eqref{eq:ideal-est}.
 	
 	We start by noticing that
 	\begin{equation*}
 		\nabla U (x) = -\abs*{\nabla U (x)} \,\frac{x}{\abs*{x}} \quad\mbox{and}\quad \nabla \phi (x) = -\abs*{\nabla \phi (x)} \,\frac{x}{\abs*{x}} \quad\mbox{for every } x \in \R^n \setminus \{0\},
 	\end{equation*}
 	therefore
 	\begin{align}
 	\notag
 		\abs*{\nabla u_\varepsilon}^2 &= \abs*{\nabla U}^2 +2\varepsilon \left\langle\nabla U, \nabla \phi \right\rangle + \varepsilon^2 \,\abs*{\nabla \phi}^2 = \abs*{\nabla U}^2 +2\varepsilon \,\abs*{\nabla U}\abs*{\nabla \phi} + \varepsilon^2 \,\abs*{\nabla \phi}^2 \\
 	\label{eq:grad-uep}
 		&= \left(\abs*{\nabla U}+\varepsilon\,\abs*{\nabla \phi}\right)^2 \quad\mbox{in } \R^n,
 	\end{align}
 	since the latter holds trivially at the origin. Notice that~\eqref{eq:grad-uep} is geometrically obvious, as~$\nabla U$ and~$\nabla \phi$ are parallel at each point and both vanish at the origin, thanks to our definition of~$\phi$. By H\"older inequality
 	\begin{equation*}
 		\int_{\R^n} \,\abs*{\nabla U}^{p-1} \abs*{\nabla \phi} \, dx \leq \norma*{\nabla U}_{L^p(\R^n)}^{p-1} \norma*{\nabla \phi}_{L^p(\R^n)},
 	\end{equation*}
 	hence a Taylor expansion, together with~\eqref{eq:energ-bubb} and~\eqref{eq:grad-uep}, yields
 	\begin{equation*}
 		S^n = \int_{\R^n} \,\abs*{\nabla U}^{p} \, dx \leq \int_{\R^n} \,\abs*{\nabla u_\varepsilon}^{p} \, dx = \int_{\R^n} \,\abs*{\nabla U}^{p} \, dx + O(\varepsilon) = S^n + O(\varepsilon).
 	\end{equation*}
 	As a result,~$u_\varepsilon$ satisfies~\eqref{eq:ipotesi-energ} provided~$\varepsilon$ is small enough.
 	
 	By taking~$\mathcal{U}_0=U$, for the left-hand side of~\eqref{eq:ideal-est} we trivially have
 	\begin{equation}
 	\label{eq:norma-lhs}
 		\norma*{u_\varepsilon-U}_{\mathcal{D}^{1,p}(\R^n)} = \varepsilon \,\norma*{\nabla\phi}_{L^{p}(\R^n)}.
 	\end{equation}
 	We then claim that
 	\begin{equation}
 	\label{eq:norma-D-1}
 		\norma*{\Delta_p u_\varepsilon + u_\varepsilon^{\past-1}}_{\mathcal{D}^{-1,p'}\!(\R^n)} = O(\varepsilon).
 	\end{equation}
 	By duality and density, we have
 	\begin{align}
 	\label{eq:est-sup}
 		\norma*{\Delta_p u_\varepsilon + u_\varepsilon^{\past-1}}_{\mathcal{D}^{-1,p'}\!(\R^n)} &= \sup_{\overset{\eta \in \mathcal{D}^{1,p}(\R^n)}{\norma*{\nabla\eta}_{L^{p}(\R^n)} = 1}} \,\abs*{\left\langle \Delta_p u_\varepsilon + u_\varepsilon^{\past-1}, \eta\right\rangle} \\
 	\notag
 		&= \sup_{\overset{\eta \in C^\infty_c(\R^n)}{\norma*{\nabla\eta}_{L^{p}(\R^n)} = 1}} \,\abs*{\int_{\R^n} - \left\langle \abs*{\nabla u_\varepsilon}^{p-2} \,\nabla u_\varepsilon, \nabla\eta\right\rangle + u_\varepsilon^{\past-1} \eta \, dx},
 	\end{align}
 	and we notice that, by testing~\eqref{eq:eqcritica} with~$\eta \in C^\infty_c(\R^n)$, we can write
 	\begin{multline}
 	\label{eq:sup-spezzato}
 		\int_{\R^n} - \left\langle \abs*{\nabla u_\varepsilon}^{p-2} \,\nabla u_\varepsilon, \nabla\eta\right\rangle + u_\varepsilon^{\past-1} \eta \, dx \\
 		= \int_{\R^n} \left\langle \abs*{\nabla U}^{p-2} \,\nabla U - \abs*{\nabla u_\varepsilon}^{p-2} \,\nabla u_\varepsilon , \nabla\eta\right\rangle + \left(u_\varepsilon^{\past-1} - U^{\past-1}\right) \eta \, dx.
 	\end{multline}
 	We shall now estimate both terms on the right-hand side of~\eqref{eq:sup-spezzato}.
 	
 	For the second term, H\"older and Sobolev inequalities entail that
 	\begin{equation*}
 			\int_{\R^n} U^{\past-2} \phi \,\abs*{\eta} \, dx \leq \norma*{U}_{L^{\past}\!(\R^n)}^{\past-2} \norma*{\phi}_{L^{\past}\!(\R^n)} \norma*{\eta}_{L^{\past}\!(\R^n)} \leq S^{-2} \norma*{U}_{L^{\past}\!(\R^n)}^{\past-2} \norma*{\nabla\phi}_{L^{p}(\R^n)}
 	\end{equation*}
 	for every~$\eta \in C^\infty_c(\R^n)$ with~$\norma*{\nabla\eta}_{L^{p}(\R^n)}= 1$. Consequently, a Taylor expansion gives that
 	\begin{equation}
 	\label{eq:est-2ndterm-sup}
 		\sup_{\overset{\eta \in C^\infty_c(\R^n)}{\norma*{\nabla\eta}_{L^{p}(\R^n)} = 1}} \,\abs*{\int_{\R^n} \left(u_\varepsilon^{\past-1} - U^{\past-1}\right) \eta \, dx} = O(\varepsilon).
 	\end{equation}
 
 	Exploiting the well-known inequality -- see, for instance,~\cite[Lemma~2.1]{dam-comp} --
 	\begin{equation*}
 		\abs*{\abs*{b}^{p-2} \, b - \abs*{a}^{p-2} \, a} \leq C_p \left(\abs*{b}+\abs*{a}\right)^{p-2} \abs*{a-b} \quad \mbox{for every } a,b \in \R^n \mbox{ with } \abs*{b}+\abs*{a}>0,
 	\end{equation*}
 	we can estimate the first term on the right-hand side of~\eqref{eq:sup-spezzato} deducing
 	\begin{equation*}
 		\abs*{\int_{\R^n}\left\langle \abs*{\nabla U}^{p-2} \,\nabla U - \abs*{\nabla u_\varepsilon}^{p-2} \,\nabla u_\varepsilon , \nabla\eta\right\rangle dx} \leq C_p \int_{\R^n} \left(\abs*{\nabla u_\varepsilon}+\abs*{\nabla U}\right)^{p-2} \varepsilon\,\abs*{\nabla\phi} \abs*{\nabla\eta} \, dx
 	\end{equation*}
 	whence, by means of H\"older inequality, we get
 	\begin{equation*}
 		\int_{\R^n} \left(\abs*{\nabla u_\varepsilon}+\abs*{\nabla U}\right)^{p-2} \varepsilon\,\abs*{\nabla\phi} \abs*{\nabla\eta} \, dx \leq \varepsilon \norma*{\nabla\phi}_{L^{p}(\R^n)} \norma*{\abs*{\nabla u_\varepsilon}+\abs*{\nabla U}}_{L^{p}(\R^n)}^{p-2}
 	\end{equation*}
 	for every~$\eta \in C^\infty_c(\R^n)$ with~$\norma*{\nabla\eta}_{L^{p}(\R^n)}= 1$. Taking into account that both~$u_\varepsilon$ and~$U$ satisfy~\eqref{eq:ipotesi-energ} for~$\varepsilon$ small enough, we conclude that
 	\begin{equation}
 	\label{eq:est-1stterm-sup}
 		\sup_{\overset{\eta \in C^\infty_c(\R^n)}{\norma*{\nabla\eta}_{L^{p}(\R^n)} = 1}} \,\abs*{\int_{\R^n} \left\langle \abs*{\nabla U}^{p-2} \,\nabla U - \abs*{\nabla u_\varepsilon}^{p-2} \,\nabla u_\varepsilon , \nabla\eta\right\rangle dx} = O(\varepsilon).
 	\end{equation}
 	As a result,~\eqref{eq:est-sup}--\eqref{eq:est-1stterm-sup} imply the validity of~\eqref{eq:norma-D-1}.
 	
 	Finally, assuming that~\eqref{eq:ideal-est} holds for some exponent $\vartheta$, the estimates~\eqref{eq:norma-lhs} and~\eqref{eq:norma-D-1} lead to a contradiction unless~$0<\vartheta \leq 1$. 
	
 	
 	\section{An application to quasi-symmetry}
 	\label{sec:application}
 	
 	In the recent paper~\cite{plap}, extending the work in~\cite{ccg}, the second author proved quasi-symmetry for non-negative, non-trivial weak solutions to~\eqref{eq:maineq-bubb} using a quantitative version of the moving plane method. Unfortunately, the dependence on the deficit, which is precisely~\eqref{eq:def_cfm}, is of logarithmic-type due to technical issues in the propagation of the weak Harnack inequality -- see also~\cite{ccg}.
 	
 	Here, we show that by exploiting Corollary~\ref{cor:main-bubbles} and the analytic tools developed in~\cite{ds-har} and further analyzed in~\cite{plap}, we can derive an approximate symmetry result for~\eqref{eq:maineq-bubb} under perhaps more natural assumptions than those of Theorem~1.3 in~\cite{plap} and with a power-like dependence on the deficit~\eqref{eq:def_cfm}. Specifically, we will prove the following result.
 	
 	\begin{theorem}
 		\label{th:qsym-consequence}
 		Let~$n \in \N$ be an integer,~$2<p<n$, and let~$\kappa \in L^\infty(\R^n) \cap C^{1,1}_{\loc}(\R^n)$ be a positive function. Let~$u \in \mathcal{D}^{1,p}(\R^n)$ be a positive weak solution to~\eqref{eq:maineq-bubb} satisfying~\eqref{eq:ip-energ-k0}, where~$\kappa_0(u)$ is defined in~\eqref{eq:defk0}, and the bound
 		\begin{equation}
 			\label{eq:univ-bound-u-Linf}
 			\norma*{u}_{L^\infty(\R^n)} \leq C_0,
 		\end{equation}
 		for some~$C_0 \geq 1$.
 		
 		Then, there exist a point~$\mathcal{O} \in \R^n$, a large constant~$C>0$, and a constant~$\vartheta' \in (0,1)$ such that
 		\begin{equation}
 			\label{eq:quasisym-point-est-toprove}
 			\abs*{u(x)-u(y)} \leq C \defi(u,\kappa)^{\vartheta'}
 		\end{equation}
 		for every~$x,y \in \R^n$ satisfying~$\abs*{x-\mathcal{O}}=\abs*{y-\mathcal{O}}$. Moreover, if~$u_\Theta$ denotes any rotation of~$u$ centered at~$\mathcal{O}$, we also have
 		\begin{equation}
 			\label{eq:quasisym-D1p-est-toprove}
 			\norma*{u-u_\Theta}_{\mathcal{D}^{1,p}(\R^n)} 
 			\leq C \defi(u,\kappa)^{\vartheta'}.
 		\end{equation}
 		The constant~$C$ depends only on~$n$,~$p$,~$C_0$,~$\norma*{\kappa}_{L^\infty(\R^n)}$, and~$\kappa_0(u)$, while~$\vartheta'$ depends only on~$n$ and~$p$.
 	\end{theorem}
 	
 	As mentioned, Theorem~\ref{th:qsym-consequence} should be compared with Theorem~1.3 in~\cite{plap}. Clearly, condition~\eqref{eq:univ-bound-u-Linf} is a necessary assumption to fix the scale of~$u$ and is present in both results -- explicitly here and implicitly in~\cite{plap} through the global decay assumption. The decay assumption in~\cite{plap} is replaced here by the energy bound~\eqref{eq:ipotesi-energ}. Both serve to prevent the occurrence of bubbling phenomena and are therefore needed in some form. Finally, apart from the improvement in the estimates, another novelty in Theorem~\ref{th:qsym-consequence} is that the lower bound for the gradient of~$u$ is no longer required. This condition was derived in~\cite{sciu} as part of the proof of the symmetry result and imposed a priori in~\cite{plap} for quantitative reasons. Here, however, this assumption is unnecessary since~\eqref{eq:quasisym-point-est-toprove} is obtained by comparing~$u$ with a multiple of a~$p$-bubble, whose gradient is explicitly known.
 	
 	\begin{proof}[Proof of Theorem~\ref{th:qsym-consequence}]
 		Let us set~$\kappa_0=\kappa_0(u)$. By applying Corollary~\ref{cor:main-bubbles}, there exists a function~$\mathcal{U} \in \mathcal{D}^{1,p}(\R^n)$ of the form
 		\begin{equation*}
 			\mathcal{U} = \kappa_0(u)^{\frac{1}{p-\past}} \,\mathcal{U}_0,
 		\end{equation*}
 		for some~$p$-bubble~$\mathcal{U}_0 \coloneqq U_p[\mathcal{O},\lambda]$, such that
 		\begin{equation}
 			\label{eq:est-quasisym-1}
 			\norma*{u-\mathcal{U}}_{\mathcal{D}^{1,p}(\R^n)} \leq C \defi(u,\kappa)^{\vartheta}.
 		\end{equation}
 		The constant~$C$ depends only on~$n$,~$p$,~$\norma*{\kappa}_{L^\infty(\R^n)}$, and~$\kappa_0$, while~$\vartheta$ depends only on~$n$ and~$p$. Moreover, as follows from the proof of Theorem~\ref{th:main-bubbles}, the scaling factor~$\lambda>0$ depends only on~$n$,~$p$, and~$\norma*{\kappa}_{L^\infty(\R^n)}$.
 		
 		From now on, the constant~$C$ will denote a quantity that may vary from line to line, but depends only on~$n$,~$p$,~$\norma*{\kappa}_{L^\infty(\R^n)}$, and~$\kappa_0$.
 		
 		Let~$\Theta$ be any rotation centered at~$\mathcal{O}$ and define~$u_\Theta(x) \coloneqq u(\Theta x)$ and~$\mathcal{U}_\Theta(x) \coloneqq \mathcal{U}(\Theta x)$ for every~$x \in \R^n$. Since~$\mathcal{U}$ is radially symmetric with respect to~$\mathcal{O}$ we clearly have that~$\mathcal{U}_\Theta=\mathcal{U}$. Therefore, a change of variables gives
 		\begin{equation*}
 			\norma*{u_\Theta-\mathcal{U}}_{\mathcal{D}^{1,p}(\R^n)} = \norma*{u_\Theta-\mathcal{U}_\Theta}_{\mathcal{D}^{1,p}(\R^n)} \leq C \defi(u,\kappa)^{\vartheta}.
 		\end{equation*}
 		Combining the latter with~\eqref{eq:est-quasisym-1} via the triangle inequality, we infer~\eqref{eq:quasisym-D1p-est-toprove}.
 		
 		We plan to prove~\eqref{eq:quasisym-point-est-toprove} by applying the local boundedness comparison inequality of Theorem~2.3 in~\cite{plap} -- see also~\cite[Theorem~3.2]{ds-har} for the original result. More precisely, we will show that
 		\begin{equation}
 			\label{eq:point-est-qiasisym-toprove}
 			\sup_{B_1(x)} \!\left(u - \mathcal{U}\right)_+ \leq C \defi(u,\kappa)^{\vartheta'} \quad\mbox{for every } x \in \R^n,
 		\end{equation}
 		for some~$\vartheta' \in (0,1)$ depending only on~$n$ and~$p$.
 		
 		To this end, we observe that
 		\begin{equation}
 			\label{eq:comparison-torewrite}
 			\Delta_p u + \kappa u^{\past-1} = \Delta_p \,\mathcal{U} + \kappa_0 \, \mathcal{U}^{\past-1} \quad \mbox{in } \R^n.
 		\end{equation}
 		Defining
 		\begin{equation*}
 			c (x) \coloneqq
 			\begin{dcases}
 				\frac{\mathcal{U}^{\past-1}(x)-u^{\past-1}(x)}{\mathcal{U}(x)-u(x)}	& \quad \mbox{if } \mathcal{U}(x) \neq u(x), \\
 				0													& \quad \mbox{if } \mathcal{U}(x) = u(x),
 			\end{dcases}
 		\end{equation*}
 		and setting~$g \coloneqq \left(\kappa-\kappa_0\right)u^{\past-1} \in L^{(\past)'}\!(\R^n) \cap L^\infty(\R^n)$, we can rewrite~\eqref{eq:comparison-torewrite} as
 		\begin{equation}
 			\label{eq:comparison-rewritten}
 			\Delta_p u + \kappa_0 \, c \, u + g = \Delta_p \,\mathcal{U} + \kappa_0 \, c \; \mathcal{U} \quad \mbox{in } \R^n.
 		\end{equation}
 		Now, we note that~$\norma*{\mathcal{U}}_{L^\infty(\R^n)} \leq C$ and also~$\norma*{\nabla\mathcal{U}}_{L^\infty(\R^n)} \leq C$. Moreover, by Theorem~1 in~\cite{diben} and assumption~\eqref{eq:univ-bound-u-Linf}, we deduce that~$\norma*{\nabla u}_{L^\infty(\R^n)} \leq C$. Consequently, we have
 		\begin{equation}
 			\label{eq:est-c-bubb}
 			c \in L^\infty(\R^n) \quad\mbox{with } 0 \leq c \leq \left(\past-1\right) \max\left\{u,\mathcal{U}\right\}^{\past-2} \leq C.
 		\end{equation}
 		
 		By~\eqref{eq:univ-bound-u-Linf}, we can establish~\eqref{eq:quasisym-point-est-toprove} assuming that~$\defi(u,\kappa) \leq \gamma$ for some~$\gamma \in (0,1)$ depending only on~$n$,~$p$,~$\norma*{\kappa}_{L^\infty(\R^n)}$, and~$\kappa_0$. Moreover, up to a translation, we may also assume that~$\mathcal{O}$ is the origin. Thus, by repeating the argument in~\ref{step:decay} of the proof of Theorem~\ref{th:main-bubbles}, we deduce that
 		\begin{equation}
 			u(x) \leq \frac{C}{1+\abs*{x}^\frac{n-p}{p-1}} \quad\mbox{for every } x \in \R^n.
 		\end{equation}
 		Now, we observe that
 		\begin{equation*}
 			\abs*{\nabla \mathcal{U}(x)} = C \left(\lambda^{\frac{p}{p-1}}+\abs*{x}^{\frac{p}{p-1}}\right)^{\!-\frac{n}{p}} \abs*{x}^{\frac{1}{p-1}} \quad\mbox{for every } x \in \R^n.
 		\end{equation*}
 		Fixing~$R_0=\lambda$ and~$R>0$, and setting~$B_R \coloneqq B_R(0)$, this enables us to conclude that
 		\begin{align}
 			\label{eq:int-peso-bubble}
 			\int_{B_{R_0}} \frac{1}{\abs*{\nabla \mathcal{U}}^{\left(p-1\right)r}} \, dx &\leq C, \\
 			\notag
 			\int_{B_R \setminus B_{R_0}} \frac{1}{\abs*{\nabla \mathcal{U}}^{\left(p-1\right)r}} \, dx &\leq C R^{n+(n-1)r} \quad\mbox{for every } r \in \left(\frac{p-2}{p-1},1\right)\!,
 		\end{align}
 		where~$C$ now depends on~$r$ as well.
 		
 		Since~\eqref{eq:comparison-rewritten}--\eqref{eq:int-peso-bubble} suffice to repeat the argument in Step~3 of the proof of Theorem~1.3 in~\cite{plap}, we infer the validity of~\eqref{eq:point-est-qiasisym-toprove}. Furthermore, applying the same reasoning to the function~$\mathcal{U}-u$, we obtain
 		\begin{equation*}
 			\sup_{B_1(x)} \!\left(\mathcal{U}-u\right)_+ \leq C \defi(u,\kappa)^{\vartheta'} \quad\mbox{for every } x \in \R^n,
 		\end{equation*}
 		possibly for a smaller~$\vartheta'$. The latter, together with~\eqref{eq:point-est-qiasisym-toprove}, entails
 		\begin{equation}
 			\label{eq:Linfty-u-bub}
 			\norma*{u-\mathcal{U}}_{L^\infty(\R^n)} \leq C \defi(u,\kappa)^{\vartheta'}.
 		\end{equation}
 		Since this argument also applies to the function~$u_\Theta-\mathcal{U}_\Theta$, we similarly get
 		\begin{equation*}
 			\norma*{u_\Theta-\mathcal{U}}_{L^\infty(\R^n)} = \norma*{u_\Theta-\mathcal{U}_\Theta}_{L^\infty(\R^n)} \leq C \defi(u,\kappa)^{\vartheta'},
 		\end{equation*}
 		possibly for a smaller~$\vartheta'$. Combining this with~\eqref{eq:Linfty-u-bub} via the triangle inequality, we establish~\eqref{eq:quasisym-point-est-toprove}. This concludes the proof.
 	\end{proof}
 	
 	
 	\section*{Acknowledgments} 
 	\noindent The authors have been partially supported by the “INdAM - GNAMPA Project”, CUP \#E5324001950001\# and by the Research Project of the Italian Ministry of University and Research (MUR) PRIN 2022 “Partial differential equations and related geometric-functional inequalities”, grant number 20229M52AS\_004.
	
	The authors thank Alessio Figalli for informing them that Gemei Liu and Yi Ru-Ya Zhang were simultaneously working on this project and for sharing a draft of~\cite{liu-zhang}.

	The authors are sincerely indebted to Carlo Alberto Antonini and Alberto Farina for sharing with them a previously unpublished manuscript -- coauthored with G.C.\ -- which already contained, in essence,~\ref{step:differ-id} and some computations from~\ref{step:approx} in the proof of Theorem~\ref{th:main-bubbles}. M.G.\ further extends his sincere thanks to Matteo Cozzi for valuable discussions on this paper's topic.
 	
 	Moreover, this work was completed while M.G.\ was visiting the Institut f\"ur Mathematik at Goethe-Universit\"at, Frankfurt am Main, whose kind hospitality is gratefully acknowledged.


\end{document}